\documentclass[12pt]{article}
\usepackage{a4wide}
\usepackage{amsmath, amsthm, amsfonts, amssymb, bbm, bm}
\usepackage{graphics}
\usepackage{xypic}
\usepackage[all]{xy}
\usepackage[english]{babel}
\usepackage[font=small,format=plain,labelfont=bf,up]{caption}

\usepackage{ifpdf} 
\ifpdf
\usepackage{epstopdf}
\usepackage{hyperref}
\else
\usepackage[hypertex]{hyperref}
\fi

\newcommand\ipic[1]		{\raisebox{-0.5\height}{\scalebox{.54}{\includegraphics{pic#1.eps}}}}

\theoremstyle{plain}

\newtheorem{theorem}{Theorem}
\newtheorem{lemma}[theorem]{Lemma}
\newtheorem{proposition}[theorem]{Proposition}

\theoremstyle{definition}

\newtheorem{remark}[theorem]{Remark}
\newtheorem{example}[theorem]{Example}
\newtheorem{definition}[theorem]{Definition}

 \numberwithin{equation}{section}
 \numberwithin{theorem}{section}

\DeclareMathOperator{\End}{End}
\DeclareMathOperator{\Endu}{\underline{End}}
\DeclareMathOperator{\Hom}{Hom}

\DeclareMathOperator{\Rep}{Rep}
\DeclareMathOperator{\ev}{ev}
\DeclareMathOperator{\coev}{coev}

\DeclareMathOperator{\sdim}{sdim}
\DeclareMathOperator{\Ind}{Ind}
\DeclareMathOperator{\Indtw}{\Ind_\mathrm{tw}}
\DeclareMathOperator{\Indtwutw}{\Ind_{(\mathrm{tw})}}

\def\be#1\ee{\begin{equation}#1\end{equation}}
\def\ba#1\ea{\begin{align}#1\end{align}}

\newcommand\nxt{\noindent\raisebox{.08em}{\rule{.44em}{.44em}}\hspace{.4em}}

\newcommand\vect{{\mathcal{V}\hspace{-.5pt}ect}}
\newcommand\svect{{\mathcal{S}\mathcal{V}\hspace{-.5pt}ect}}

\newcommand\eps           {\varepsilon}
\newcommand\id            {id}
\newcommand\Id            {I\hspace{-1pt}d}
\newcommand\one           {{\bf1}}
\newcommand\fd           {^\mathrm{fd}}
\newcommand\lf           {^\mathrm{lf}}

\newcommand\Pgs	{P_{\hspace{-1.5pt}\mathrm{gs}}}
\newcommand\h	{\mathfrak{h}}

\newcommand\Cb            {\mathbb{C}}

\newcommand\Rb            {\mathbb{R}}
\newcommand\Zb            {\mathbb{Z}}

\newcommand\Cc            {\mathcal{C}}

\newcommand\Ic            {\mathcal{I}}

\newcommand\Vc            {\mathcal{V}}

\newcommand{\void}[1]{}

\newcommand{\blockA}[3]{\ensuremath{
     \setlength{\unitlength}{1mm}
     \begin{picture}(10,15)(0,0)
          \put( 0, 4){\line(1,0){10}}
          \put(10, 4){\line(0,1){10}}
          \put( 2, 5){\makebox(0,0)[lb]{\small $#1$}}
          \put(11,13){\makebox(0,0)[lt]{\small $#2$}}
          \put(10, 3){\makebox(0,0)[ct]{\small $#3$}}
     \end{picture}}}

\newcommand{\blockB}[1]{\ensuremath{
     \setlength{\unitlength}{1mm}
     \begin{picture}(10,15)(0,0)
          \put( 0, 4){\line(1,0){10}}
          \put( 8, 5){\makebox(0,0)[rb]{\small $#1$}}
     \end{picture}}}

\newcommand{\blockC}[3]{\ensuremath{
     \setlength{\unitlength}{1mm}
     \begin{picture}(10,15)(0,0)
          \put( 0, 4){\line(1,0){10}}
          \put(10, 4){\line(0,1){10}}
          \put( 5, 5){\makebox(0,0)[cb]{\small $#1$}}
          \put(11,13){\makebox(0,0)[lt]{\small $#2$}}
          \put(10, 3){\makebox(0,0)[ct]{\small $#3$}}
     \end{picture}}}

\newcommand{\blockD}[3]{\ensuremath{
     \setlength{\unitlength}{1mm}
     \begin{picture}(8,15)(0,0)
          \put( 0, 4){\line(1,0){8}}
          \put( 8, 4){\line(0,1){10}}
          \put( 1, 5){\makebox(0,0)[cb]{\small $#1$}}
          \put( 9,13){\makebox(0,0)[lt]{\small $#2$}}
          \put( 8, 3){\makebox(0,0)[ct]{\small $#3$}}
     \end{picture}}}

\newcommand{\blockE}[0]{\ensuremath{
     \setlength{\unitlength}{1mm}
     \begin{picture}(3,15)(0,0)
          \put( 0, 4){\line(1,0){3}}
     \end{picture}}}

\newcommand{\blockUp}[1]{\ensuremath{
     \setlength{\unitlength}{1mm}
     \put( 0,10){#1} }}

\newcommand{\cbB}[7]{\blockA{#1}{#2}{#6}\blockC{#5}{#3}{#7}\blockB{#4}}
\newcommand{\cbC}[7]{\blockA{#1}{}{#6}%
     \blockUp{\blockD{#5}{#2}{#7}\blockB{#3}}%
     \blockE\blockB{#4}}
\newcommand{\bL}[1]{\raisebox{-3mm}{#1}}

\begin{document}

\thispagestyle{empty}
\def\thefootnote{\fnsymbol{footnote}}
\begin{flushright}
ZMP-HH/12-17\\
Hamburger Beitr\"age zur Mathematik 448
\end{flushright}
\vskip 3em
\begin{center}\LARGE
A braided monoidal category for free super-bosons
\end{center}

\vskip 2em
\begin{center}
{\large 
Ingo Runkel~\footnote{Email: {\tt ingo.runkel@uni-hamburg.de}}}
\\[1em]
Fachbereich Mathematik, Universit\"at Hamburg\\
Bundesstra\ss e 55, 20146 Hamburg, Germany
\end{center}

\vskip 2em
\begin{center}
  September 2012
\end{center}
\vskip 2em

\begin{abstract}
The chiral conformal field theory of free super-bosons is generated by weight one currents whose mode algebra is the affinisation of an abelian Lie super-algebra $\h$ with non-degenerate super-symmetric pairing. The mode algebras of a single free boson and of a single pair of symplectic fermions arise for even$|$odd dimension $1|0$ and $0|2$ of $\h$, respectively.

In this paper, the representations of the untwisted mode algebra of free super-bosons are equipped with a tensor product, a braiding, and an associator. In the symplectic fermion case, i.e.\ if $\h$ is purely odd, the braided monoidal structure is extended to representations of the $\Zb/2\Zb$-twisted mode algebra. 
The tensor product is obtained by computing spaces of vertex operators. The braiding and associator are determined by explicit calculations from three- and four-point conformal blocks. 
\end{abstract}

\setcounter{footnote}{0}
\def\thefootnote{\arabic{footnote}}

\newpage

{\small

\tableofcontents

}

\newpage

\section{Introduction}

Consider the mode algebra 
\be \label{eq:mode-alg-intro}
  [a_m,b_n] = m \, (a,b) \, \delta_{m+n,0} \cdot K \ ,
\ee
where $K$ is central, $a,b$ are elements of a vector space $\h$, $(-,-)$ is a non-degenerate bilinear form on $\h$ and $m,n \in \Zb$. If the bracket is anti-symmetric and the bilinear form symmetric, this is the mode algebra of $\dim(\h)$ free bosons. If, conversely, the bracket is symmetric and $(-,-)$ antisymmetric, non-degeneracy requires $\dim(\h)$ to be even and \eqref{eq:mode-alg-intro} is the mode algebra of $\dim(\h)/2$ pairs of symplectic fermions. Both cases can be treated on the same footing if one takes $\h$ to be a super-vector space and $(-,-)$ to be super-symmetric. One arrives at what Kac called free super-bosons \cite{Kac:1998}.

Denote the mode algebra \eqref{eq:mode-alg-intro} by $\hat\h$. Representations of $\hat\h$ which are `bounded below' are easy to describe: they are controlled by representations of the zero mode algebra on the subspace of ground states. It is worth pointing out that even in the free boson case, i.e.\ when $\h$ is purely (parity-)even, non-semisimple $\hat\h$-modules are included; this is in contrast to the more familiar case of unitary free boson models where the action of the zero modes can be diagonalised.

Starting from the category of bounded $\hat\h$-representations, one can now try to use the monodromy and factorisation properties of conformal blocks to endow this category with a braided monoidal structure. Actually, as outlined in Section \ref{sec:into-fsb} below, this computation is very straightforward and the result is not surprising: it gives an instance of Drinfeld's category for metric Lie algebras specialised to the abelian case \cite{Drinfeld:1990,Davydov:2010od}. 

The outcome is more interesting if one includes $\Zb/2\Zb$-twisted representations, i.e.\ representations of the mode algebra \eqref{eq:mode-alg-intro} with $m,n \in \Zb+\frac12$. 
Most of the effort (and pages) will be spent on this case.
In the approach I will take, the existence of a braided monoidal structure on twisted and untwisted representations together requires $\h$ to be purely odd. This is the symplectic fermion case, and the results are outlined in Section \ref{sec:into-sf} below. (If $\h$ has an even component, the tensor product of two twisted representations does not satisfy the required finiteness properties.)

\medskip

The main motivation behind this work is to provide the basic categorical data needed for the classification of conformal field theories whose chiral algebra includes the (even part of the) symplectic fermion vertex operator super-algebra. Symplectic fermion models are one of the most basic and best studied examples of logarithmic conformal field theories. They first appeared in \cite{Kausch:1995py} and their representation theory and fusion rules were investigated in \cite{Gaberdiel:1996np,Fuchs:2003yu,Abe:2005,Feigin:2005xs,Nagatomo:2009xp,Abe:2011ab,Tsuchiya:2012ru}.
Bulk and boundary conformal field theories built from symplectic fermions were studied in \cite{Gaberdiel:1998ps,Kawai:2001ur,Ruelle:2002jy,Bredthauer:2002ct,Gaberdiel:2006pp}. Some recent works focus on the relation to models with $\mathfrak{gl}(1|1)$ symmetry \cite{Read:2007qq,LeClair:2007aj,Creutzig:2008an,Creutzig:2011cu,Gainutdinov:2011ab}.

The results of the present paper constitute the first explicit computation of the associativity and braiding isomorphisms on the category of chiral algebra representations for a logarithmic conformal field theory.

\subsection{Free super-bosons}\label{sec:into-fsb}

Let $\h$ be a finite-dimensional abelian Lie super-algebra with non-degenerate supersymmetric bilinear form (see Section \ref{sec:lie-super} for details and conventions). The category $\Rep_{\flat,1}(\hat\h)$ of bounded-below $\hat\h$-modules on which $K$ acts as the identity (Section \ref{sec:h-mod-untwist}) is equivalent to the category $\Rep(\h)$ of $\h$-modules (Theorem \ref{thm:ind-equiv-untw}). For the sake of this introduction, we restrict our attention to finite dimensional $\h$-modules and to $\hat\h$-modules with finite-dimensional spaces of ground states; this will be indicated by the superscript `fd'. In the main text, this is relaxed to local finiteness.

Let $A \in \Rep\fd(\h)$ and $\hat B, \hat C \in \Rep\fd_{\flat,1}(\hat\h)$. By a vertex operator from $A \otimes \hat B$ to $\hat C$ we mean a map $V(x)$ depending smoothly on a positive real number $x$, and which satisfies the standard mode exchange relations (Definition \ref{def:vertex-op})
\be
  a_m \, V(x) = V(x) \, \big( x^m \cdot a \otimes \id_{\hat B} + \id_A \otimes a_m \big) \ .
\ee
One should think of $V(x)(a\otimes \hat b)$ as a ground state $a$ inserted at point $x$, mapping $\hat b$ into (a completion of) $\hat C$. Such vertex operators can be expressed as normal ordered exponentials and one finds that the space of vertex operators is naturally isomorphic to $\Hom_{\Rep(\h)}(A \otimes B,C)$, where $B,C$ are the representations of $\h$ on the space of ground states of $\hat B$, $\hat C$ (Theorem \ref{thm:G-bijective}). Thus, as expected, the spaces of vertex operators from any two representations into a third are controlled solely by the action of the zero mode algebra $\h$.  

Denote the space of vertex operators from $A \otimes \hat B$ to $\hat C$ by $\Vc_{A,B}^C$. If we fix $A$ and $B$, we can think of this as a functor from $\Rep\fd(\h)$ into vector spaces. The tensor product `$\ast$' on $\Rep\fd(\h)$ induced by the vertex operators is defined to be the representing object of this functor. That is, it is an object $A \ast B \in \Rep\fd(\h)$ such that $\Hom_{\Rep(\h)}(A \ast B,C)$ and $\Vc_{A,B}^C$ are naturally isomorphic (as functors in the argument $C$). As mentioned above, Theorem \ref{thm:G-bijective} implies $\Vc_{A,B}^C \cong \Hom_{\Rep(\h)}(A \otimes B,C)$, and so we simply have 
\be\label{eq:intro-tensor-untw}
  A \ast B = A \otimes B \ ,
\ee   
where the right hand side is the tensor product of $\h$-modules (Theorem \ref{thm:*-tensor}).

The monodromy properties of analytically continued vertex operators provide us with a braiding isomorphism $c_{A,B} : A \ast B \to B \ast A$ on $\Rep\fd(\h)$. Evaluating the asymptotic behaviour of a product $V(1)V(x)$ of two vertex operators for $x \to 1$ results in associativity isomorphisms $\alpha_{A,B,C} : A \ast (B\ast C) \to (A \ast B) \ast C$. These are computed in Proposition \ref{prop:braiding-iso-from-3pt} and Section \ref{sec:assoc-compute} and take the simple form
\be\label{eq:intro-a-c-untwist}
  \alpha_{A,B,C}  = \id_{A \otimes B \otimes C}
  \quad , \qquad
  c_{A,B} = \tau_{A,B} \circ \exp\!\big(-  i \pi \Omega \big) \ .
\ee
Here, $\tau_{A,B} : A \otimes B \to B \otimes A$ is the exchange of tensor factors for super-vector spaces (which involves a parity sign if both arguments are odd) and $\Omega \in \h \otimes \h$ is the copairing associated to $(-,-)$. (The associator of super-vector spaces was suppressed in writing `$\id_{A \otimes B \otimes C}$'.)
The exponential in the expression for the braiding converges because $A \otimes B$ is finite-dimensional. Indeed, this is the reason to impose finite-dimensionality in the first place  (or at least local finiteness as in the main text).

It is proved in Proposition \ref{prop:C0-monoidal} that the data \eqref{eq:intro-a-c-untwist} turn $\Rep\fd(\h)$ into a braided monoidal category. As mentioned above, this is a (very simple, since $\h$ is abelian) instance of Drinfeld's category for metric Lie algebras.

\subsection{Symplectic fermions}\label{sec:into-sf}

Let now $\h$ be purely odd. In this case we include representations of the twisted mode algebra $\hat\h_\mathrm{tw}$, which is the same as \eqref{eq:mode-alg-intro}, but with $m,n \in \Zb+\frac12$. Since there are no zero modes, the category of bounded-below $\hat\h_\mathrm{tw}$-representations is equivalent to the category $\svect$ of super-vector spaces (Theorem \ref{thm:ind-equiv-tw}). Consider the $\Zb/2\Zb$-graded category $\Cc$ with components
\be
  \Cc = \Cc_0 + \Cc_1
  \qquad , \quad \Cc_0 = \Rep(\h)
  ~~, \quad \Cc_1 = \svect \ .
\ee
(Since $\h$ is purely odd, it acts nilpotently and we can drop the assumption of finite-dimensionality.) As in the previous case, one defines spaces of vertex operators $\Vc_{A,B}^C$ for $A,B,C \in \Cc$. Via the same representing object condition, one obtains a tensor product on $\Cc$ (Theorem \ref{thm:*-tensor}): in addition to \eqref{eq:intro-tensor-untw} we have
\be
A \ast B ~=~  
\left\{\rule{0pt}{2.0em}\right.
\hspace{-.5em}\raisebox{.7em}{
\begin{tabular}{ccll}
   $A$ & $B$ & $A \ast B$ &
\\
 $\Cc_0$ & $\Cc_1$ & $F(A) \otimes B$ & $\in~\Cc_1$
\\
 $\Cc_1$ & $\Cc_0$ & $A \otimes F(B)$ & $\in~\Cc_1$
\\
 $\Cc_1$ & $\Cc_1$ & $U(\h) \otimes A \otimes B$ & $\in~\Cc_0$
\end{tabular}}
\ee
Here, $F$ is the forgetful functor $\Rep(\h) \to \svect$ and $U(\h)$ is the universal enveloping algebra of the abelian Lie super-algebra $\h$. Since $\h$ is purely odd, $U(\h)$ is finite-dimensional.\footnote{
  $U(\h)$ is the exterior algebra of $\h$, provided $\h$ is understood as a vector space rather than a super-vector space. In $\svect$, $U(\h)$ is the symmetric algebra of $\h$ (since $\h$ is purely odd).}

As above, one obtains braiding isomorphisms from the analytic continuation of vertex operators. In addition to \eqref{eq:intro-a-c-untwist} one finds (Sections \ref{sec:3pt-braid} and \ref{sec:C-hexagon}):
\be
c_{A,B} ~=~
\left\{\rule{0pt}{2.0em}\right.
\hspace{-.5em}\raisebox{.7em}{
\begin{tabular}{ccl}
 $A$ & $B$ & $c_{A,B} ~:~ A \ast B \to B \ast A$
\\[.3em]
 $\Cc_0$ & $\Cc_1$ & $\tau_{A,B} \circ \exp\!\big( \tfrac{i \pi}{2} \cdot \Omega^{(11)}\big)$
\\[.3em]
 $\Cc_1$ & $\Cc_0$ & $\tau_{A,B} \circ \exp\!\big( \tfrac{i \pi}{2} \cdot \Omega^{(22)}\big) \circ \big( \id_A \otimes \omega_B \big)$
\\[.3em]
 $\Cc_1$ & $\Cc_1$ & $e^{-i \pi  \frac{N}4} \cdot
  \big(\id_{U(\h)} \otimes \tau_{A,B}\big) \circ \exp\!\big(\! -\tfrac{i \pi}{2} \cdot \Omega^{(11)}\big)\circ \big( \id_{U(\h) \otimes A} \otimes \omega_B \big)$
\end{tabular}}
\ee
Here $N$ is the number of pairs of symplectic fermions, $N = \dim(\h)/2$, and as above, $\Omega \in \h \otimes \h$ is the copairing associated to $(-,-)$. The indices $i,j$ in $\Omega^{(ij)}$ determine on which tensor factors the two copies of $\h$ act. Finally, $\omega_B$ denotes the parity involution on the super-vector space $B$.

The most tedious part of the computation is to find the associator on $\Cc$. This is done by explicitly computing the four-point blocks for the eight possible ways to choose the representations from $\Cc_0$ and $\Cc_1$. These blocks are listed in Table \ref{tab:4pt-blocks}; they have been computed by conformal field theory methods, in particular by using the Knizhnik-Zamolodchikov differential equation. Because these methods lie outside of the formalism set up in Sections \ref{sec:mode+rep} and \ref{sec:VO+tensor}, one would now have to prove that the resulting blocks indeed agree with the composition of vertex operators. I have only been able to do this for two of the eight cases and for this reason, Section \ref{sec:4pt-assoc} is equipped with a `$\star$' -- in the context of the present formalism, the results in Section \ref{sec:4pt-assoc} should be treated as conjectures.

Nonetheless, the explicit four-point blocks allow one to determine the associativity isomorphisms $\alpha_{A,B,C}$ on $\Cc$. These are listed in Section \ref{sec:C-pentagon}. It is proved in Theorems \ref{thm:C-associator} and \ref{thm:C-braiding} that the data $\ast$, $\alpha$, $c$ turn $\Cc$ into a braided monoidal category. This is done by explicitly verifying the 16 pentagon and 16 hexagon identities, or, more precisely, by referring to \cite{Davydov:2012xg} where this has already been done in a more general setting.

\medskip

A determined person should be able to reproduce the braided monoidal structure presented here in the formalism of vertex operator algebras. Since the symplectic fermion vertex operator algebra is $C_2$-cofinite \cite{Abe:2005}, the tensor product theory of \cite{Huang:2010,Tsuchiya:2012ru} can be applied. As a step in this direction, a few intertwining operators were recently provided in \cite{Abe:2011ab}. A treatment in this setting would have the practical and conceptual advantage that the general results of \cite{Huang:2010} (and those announced in \cite{Tsuchiya:2012ru}) would guarantee that the pentagon and hexagon axioms hold.

\medskip

The paper concludes with two small applications to symplectic fermions. In Section \ref{sec:torusblock}, torus conformal blocks with insertions from $U(\h)$ and their relation to characters are discussed. In Section \ref{sec:Uh-algebra}, $U(\h)$ is turned into an algebra in $\Cc$ and the resulting multiplication is compared to a boundary operator product expansion from \cite{Gaberdiel:2006pp}.

\bigskip\noindent
{\bf Acknowledgements:}
It is my pleasure to thank
David B\"ucher,
Alexei Davydov,
J\"urgen Fuchs,
Matthias Gaberdiel,
Terry Gannon,
Yi-Zhi Huang,
Antun Milas,
Martin Mombelli,
David Ridout,
Christoph Schweigert,
Volker Schomerus,
Alexei Semikhatov and
Simon Wood
for helpful discussions.
J\"urgen Fuchs,
Matthias Gaberdiel, 
David Ridout, 
Alexei Semikhatov 
and the referee
provided many valuable comments on a draft of this paper.
I would also like to thank the Beijing International Center for Mathematical Research for hospitality during my stay from mid-May to mid-June in 2011, during which most of the research presented in this paper was done.

\section{Mode algebra and representations}\label{sec:mode+rep}

\subsection{Metric abelian Lie super-algebras}\label{sec:lie-super}

A {\em super-vector space} over a field $k$ is a $\Zb/2\Zb$-graded $k$-vector space $V = V_0 \oplus V_1$. If  $\dim(V_i)=n_i$, we write $\dim(V) = n_0+n_1$, $\sdim(V) = n_0-n_1$, and we also say that $V$ has dimension $n_0|n_1$. Let
\be
  k^{m|n} := (k^m)_0 \oplus (k^n)_1 
\ee
be the standard $m|n$-dimensional super-vector space over $k$. Given a homogeneous element $v$ of a super-vector space $V$, we write $|v| \in \{0,1\}$ for its parity; in formulas involving $|v|$ it is implicitly understood that $v$ is homogeneous.

Denote by $\Hom(V,W)$ the space of all {\em even} linear maps from $V$ to $W$. The category $\svect$ has super-vector spaces over $k$ as objects, and $\Hom(V,W)$ is the space of morphisms from $V$ to $W$. It is a symmetric monoidal category with symmetry map 
\be \label{eq:svect-tau}
\tau_{V,W} : V \otimes W \to W \otimes V
\quad , \quad \tau_{V,W}(v \otimes w) = (-1)^{|v||w|} \cdot w \otimes v \ .
\ee
The full subcategory of finite-dimensional super-vector spaces is denoted by $\svect\fd$.

\medskip

In a Lie super-algebra $\mathfrak{g}$, the Lie bracket is an even bilinear map $[-,-] : \mathfrak{g} \times \mathfrak{g} \to \mathfrak{g}$ which is anti super-symmetric. That is, the parity of $[a,b]$ is $|a|+|b|$, and $[a,b] = -(-1)^{|a||b|} [b,a]$. For more on Lie super-algebras, see \cite{Kac:1977,Scheunert:1979}. 
An example is $\Endu(V)$, the space of {\em all} linear maps from $V$ to $V$, which is itself a super-vector space. The Lie bracket is given by the super-commutator, $[A,B] = AB - (-1)^{|A||B|}BA$, for $A,B \in \Endu(V)$. A more trivial example is the zero bracket, $[a,b]=0$ for all $a,b \in \mathfrak{g}$; in this case the Lie super-algebra is called {\em abelian}.

Let $\{e_i\}_{i \in \Ic}$ be a homogeneous basis of $\mathfrak{g}$ and fix an ordering on $\Ic$. By the Poincar\'e-Birkhoff-Witt Theorem for Lie super-algebras (see \cite[\S\,2.3]{Scheunert:1979}), a basis of the universal enveloping algebra $U(\mathfrak{g})$ is given by
\be
  \Big\{  \, e_{i_1} e_{i_2}  \cdots e_{i_m} ~\Big|~ m \ge 0 \text{ and } 
  i_1  
  \makebox[0pt][l]{\raisebox{-.3em}{\footnotesize $(=)$}}\raisebox{.2em}{\hspace{3pt}$<$\,\,} \, 
  i_2 
  \makebox[0pt][l]{\raisebox{-.3em}{\footnotesize $(=)$}}\raisebox{.2em}{\hspace{3pt}$<$\,\,} \, 
  \cdots 
  \makebox[0pt][l]{\raisebox{-.3em}{\footnotesize $(=)$}}\raisebox{.2em}{\hspace{3pt}$<$\,\,} \, 
  i_m \Big\} \ ,
\ee
where the strict inequality $i_k<i_{k+1}$ is imposed if $e_{i_k}$ and $e_{i_{k+1}}$ are odd.

A {\em metric Lie super-algebra} is a pair $(\mathfrak{g},(-,-))$, where $\mathfrak{g}$ is a Lie super-algebra over a field $k$ and $(-,-) : \mathfrak{g} \times \mathfrak{g} \to k$ is an even, non-degenerate, invariant, supersymmetric pairing. That is, the pairing is non-degenerate and it satisfies
\begin{itemize}
\item $(a,b) = 0$ if $|a| \neq |b|$ (even), 
\item $(a,b) = (-1)^{|a||b|}(b,a)$ (supersymmetric), 
\item $([a,b],c) = (a,[b,c])$ (invariant). 
\end{itemize}

\begin{definition}
Let $(\h,(-,-))$ be a metric abelian Lie super-algebra over $\Cb$. The {\em affinisation} $\hat \h$ of $\h$ is the super-vector space $\h \otimes \Cb[t^{\pm 1}] \, \oplus \,\Cb\,K$, where $t$ is a parity-even formal parameter, together with Lie bracket
\be\label{eq:affine-commutator}
  [a_m,b_n] = m \, (a,b) \, \delta_{m+n,0} \cdot K
  \quad , \quad \text{where} \quad  x_m := x \otimes t^m ~~\text{for}~~x\in \h\,,~m\in\Zb \ ,
\ee
and where the generator $K$ is central.
\end{definition}

\begin{remark}
The two simplest cases are when $\h$ has dimension $1|0$, in which case $\hat\h$ is the mode algebra of a single free boson, and when $\h$ has dimension $0|2$, in which case $\hat\h$ is the mode algebra of a single pair of symplectic fermions \cite{Kausch:1995py,Kausch:2000fu}. Note that dimension $0|1$ is not allowed as there is no non-degenerate supersymmetric pairing.
\end{remark}

A representation of a Lie super-algebra $\mathfrak{g}$ is a super-vector space $V$ together with an even linear map $\rho : \mathfrak{g} \to \Endu(V)$ which is a homomorphism of Lie super-algebras. The map $\rho$ will usually be omitted from formulas. We will make use of three abelian categories of $\h$-representations:
\begin{center}
\begin{tabular}{lp{28em}}
$\Rep(\h)$ & All $\h$-representations and all $\h$-intertwiners.
\\[.3em]
$\Rep(\h)\lf$ & The full subcategory of $\h$-representations which are {\em locally finite}: 
  for any vector $v$ in the representation, the subspace $U(\h).v$ is finite-dimensional.
\\[.3em]
$\Rep(\h)\fd$ & The full subcategory of finite-dimensional $\h$-representations.
\end{tabular}
\end{center}
For $\h \neq \{0\}$, none of these categories is semisimple, irrespective of the dimensions of the even and odd component of $\h$. For example, $U(\h)$ is in $\Rep(\h)$ but not in $\Rep(\h)\lf$. And the `filtered dual' $\bigcup_{n \ge 0} U_n(\h)^* \subset U(\h)^*$, where $U_n(\h)$ is spanned by words of at most $n$ generators, is in $\Rep(\h)\lf$ but not in $\Rep(\h)\fd$.

\bigskip

For the rest of this paper we fix the following conventions:
\begin{itemize}
\item All vector spaces will be over the complex numbers.
\item We fix a non-zero finite-dimensional abelian metric Lie super-algebra $\h$ with pairing $(-,-)$; $\hat\h$ refers to the affinisation of $\h$.
\item We choose a homogeneous basis $\{ \alpha^i \}_{i \in \Ic}$ of $\h$ and denote by $\{ \beta^i \}_{i \in \Ic}$ the dual basis in the sense that
\be
  (\alpha^i,\beta^j) = \delta_{i,j} \ .
\ee
\item The copairing for $(-,-)$ is denoted by $\Omega$,
\be \label{eq:copair-Omega-def}
  \Omega = \sum_{i \in \Ic} \beta^i \otimes \alpha^i \quad \in~ \h \otimes \h \ .
\ee
\end{itemize}
Note that since $(\beta^j,\alpha^i) = (-1)^{|\alpha^i|} \delta_{i,j}$, the basis dual to $\{ \beta^i \}_{i \in \Ic}$ is $\{ (-1)^{|\alpha^i|} \alpha^i \}_{i \in \Ic}$. The copairing is independent of the choice of basis. In particular,
\be \label{eq:Omega-symmetric}
  \Omega = \sum_{i \in \Ic} (-1)^{|\alpha^i|} \, \alpha^i \otimes \beta^i = \tau_{\h,\h}(\Omega) \ .
\ee

\subsection{Untwisted representations}\label{sec:h-mod-untwist}

\begin{definition} \label{def:bounded-below-rep}
An $\hat\h$-representation $M$ is {\em bounded below} if for each $x \in M$ there exists $N_x>0$ such that
$$
  a^1_{m_1} a^2_{m_2} \cdots a^L_{m_L} x = 0
  \quad \text{for all} ~~ L>0 ~,~~ a^i \in \h ~,~~ m_i \in \Zb ~\text{with}~ m_1+\dots+m_L > N_x \ .
$$
The abelian category of $\hat\h$-representations which are bounded below and on which $K$ acts as the identity map will be denoted by $\Rep_{\flat,1}(\hat\h)$.
\end{definition}

In other words, if $M$ is bounded below, each element $x$ of $M$ is annihilated by every sufficiently positive composition of modes, where the total degree can depend on $x$. One important property of $M \in \Rep_{\flat,1}(\hat\h)$ is that the grading operator $H$ can be defined (cf.\,\cite[Sect.\,3.5]{Kac:1998}):
\be \label{eq:H-op-def}
  H = \sum_{m>0} \sum_{i \in \Ic} \beta_{-m}^i \alpha^i_{m}   \ .
\ee
For any $x \in M$, the infinite sum over $m$ in $Hx$ reduces to a finite sum. The operator $H$ satisfies $[H,a_m] = -m \, a_m$ for all $a\in \h$, $m \in \Zb$.

Let $\hat\h_{\ge 0}$ be the subalgebra of $\hat\h$ spanned by all $a_m$ with $m \ge 0$. A representation $R$ of $\h$ becomes a representation of $\hat\h_{\ge 0} \oplus \Cb K$ by declaring that $a_m$ acts as zero for $m>0$, that $a_0$ acts as $a \in \h$, and that $K$ acts as the identity. With this convention, we define the induced $\hat\h$ representation of $R$ as
\be
  \Ind(R) = U(\hat\h) \otimes_{U(\hat\h_{\ge 0} \oplus \Cb K)} R \ .
\ee
By construction, $\Ind(R)$ is an $\hat\h$-module which is bounded below. On morphisms $f : R \to S$ of $\h$-modules induction acts as $\Ind(f) = \id_{U(\hat\h)} \otimes_{U(\hat\h_{\ge 0} \oplus \Cb K)} f$.

\begin{theorem} \label{thm:ind-equiv-untw}
$\Ind : \Rep(\h) \to \Rep_{\flat,1}(\hat\h)$ is an equivalence of $\Cb$-linear categories.
\end{theorem}

For the convenience of the reader, in Appendix \ref{app:ind-equiv} we adapt the proof in \cite[Sect.\,1.7]{Frenkel:1987} and \cite[Sect.\,3.5]{Kac:1998} to the present (non-semisimple) setting. The inverse equivalence is given by passing to the subspace of {\em singular vectors} or {\em ground states} in an $\hat\h$-module $M$, that is, to the subspace of all vectors in $M$ annihilated by all $a_m$ with $m>0$.

Let $\Rep\lf_{\flat,1}(\hat\h)$ (resp.\ $\Rep\fd_{\flat,1}(\hat\h)$) be the full subcategory of $\Rep_{\flat,1}(\hat\h)$ consisting of representations $M$ whose subspace of ground states is locally finite as an $\h$-module (resp.\ finite-dimensional). Then also
\be
  \Ind : \Rep\lf(\h) \to \Rep_{\flat,1}\lf(\hat\h)
  \quad , \quad
  \Ind : \Rep\fd(\h) \to \Rep_{\flat,1}\fd(\hat\h)
\ee
are equivalences.

\begin{remark} (i) 
Let $\Cb$ be the trivial $\h$-representation. Then $\Ind(\Cb)$ can be turned into a vertex operator super-algebra with stress tensor $T = \tfrac12 \sum_{i \in \Ic} \beta^i_{-1} \alpha^i_{-1} \one$, see \cite[Sect.\,1.9]{Frenkel:1987} and \cite[Sect.\,3.5]{Kac:1998}. (Actually, there is a family of stress tensors parameterised by $\h_0$, the even part of $\h$.) The resulting action of the Virasoro algebra on modules $M \in\Rep_{\flat,1}(\hat\h)$ is
\ba
L_m &= \tfrac12 \sum_{i \in \Ic} \Big( 
\sum_{k \in \Zb_{<0}} \beta^i_k \alpha_{m-k}^i
+
\sum_{k \in \Zb_{\ge0}} \beta^i_{m-k} \alpha_{k}^i \Big) \qquad , ~ m \neq 0 \ ,
\nonumber \\
L_0 &=\tfrac12 \sum_{i \in \Ic}\beta^i_0 \alpha_{0}^i + H \quad , \qquad
C = \sdim(\h) \cdot \id_M \ ,
\label{eq:virasoro-untwisted}
\ea
with $H$ as in \eqref{eq:H-op-def}. The above expression for $L_m$, $m\neq 0$, makes the normal ordering explicit. It can also be written more compactly as 
\be \label{eq:Lm-nonzero}
  L_m = \tfrac12 \sum_{k \in \Zb} \sum_{i\in \Ic} \beta^i_k \alpha_{m-k}^i \ .
\ee   
This follows from basis independence, since, for $k+l \neq 0$, $\sum_{i \in\Ic} \beta^i_k \alpha^i_l = \sum_{i \in\Ic} (-1)^{|\alpha_i|} \alpha^i_l \beta^i_k = \sum_{i \in\Ic} \beta^i_l \alpha^i_k$.

\smallskip\noindent
(ii) Let $V$ be a vertex operator algebra. Following \cite{Nahm:1994by,Li:1998}, for a weak $V$-module $W$ set $C_1(W) = \mathrm{span}_{\Cb}\{ v_{-1} w \,|\, w \in W \,,\, v \in V \,,\, \text{$v$ has positive $L_0$-eigenvalue} \,\}$. Then $W$ is called $C_1$-cofinite if $W/C_1(W)$ is finite-dimensional. Since (in the standard grading convention for vertex algebras) we have $(\partial^n v)_{-1} = n!\, v_{-n-1}$, it is clear that the representations in $\Rep_{\flat,1}\fd(\hat\h)$ are $C_1$-cofinite. One important property of $C_1$-cofinite modules is that products of intertwining operators satisfy differential equations \cite[Thm.\,11.6]{Huang:2010} -- provided they are also quasi-finite-dimensional and generalized $V$-modules (as is the case for $\Rep_{\flat,1}\fd(\hat\h)$).
\end{remark}

The induced $\hat\h$-module $\Ind(R)$ for an $\h$-module $R$ is a direct sum of $H$-eigenspaces,
\be
 \Ind(R) = \bigoplus_{m \in \Zb_{\ge 0}} \Ind(R)_m
 \quad , \quad
 \text{where}
 \quad H\big|_{ \Ind(R)_m} = m \cdot \id_{ \Ind(R)_m} \ .
\ee
Note that $\Ind(R)$ is in general not a direct sum of $L_0$-eigenspaces (and not even of generalised $L_0$-eigenspaces, e.g.\ if $R=U(\h)$); for $H$ this does not happen as it does not involve zero modes. We will denote the algebraic completion of the above direct sum as\footnote{
  There is a potential pitfall with the above definition of $\overline\Ind(R)$ which originates from using the $H$-grading instead of the (in general non-existent) $L_0$-grading. Namely, consider the case $\h = \Cb^{1|0}$ and $R = \bigoplus_{m \in \Zb} \Cb_m$, where $\Cb_m$ is the one-dimensional $\h$-module on which the generator of $\h$ acts by multiplication by $m$. Then the vector $1_m \in \Cb_m$ has $L_0$-eigenvalue $\frac12 m^2$ and $H$-eigenvalue $0$. The infinite sum $\sum_{m \in \Zb} 1_m$ is therefore an element of the algebraic completion of $\Ind(R)$ with respect to the $L_0$-gradation, but it is not in the completion with respect to the $H$-gradation, i.e.\ $\sum_{m \in \Zb} 1_m \notin \overline\Ind(R)$. Since the $L_0$-gradation is often used in conformal field theory, this point should be kept in mind.}
\be
  \overline\Ind(R) =  \prod_{m \in \Zb_{\ge 0}} \Ind(R)_m \ .
\ee
Thus, elements of $\Ind(R)$ are finite sums of homogeneous vectors of distinct degrees, while elements of $\overline\Ind(T)$ can be formal infinite sums of such vectors. The {\em restricted dual} of $\Ind(R)$ is defined as
\be
  \Ind(R)' = \bigoplus_{m \in \Zb_{\ge 0}} \!\! \big(\Ind(R)_m\big)^*  ~ \subset \Ind(R)^* \ .
\ee
The space $\Ind(R)'$ becomes an $\hat h$-module which is bounded below by setting, for $\varphi \in \Ind(R)'$ and $v \in \Ind(R)$, 
\be \label{eq:am-dual-action}
  (a_m \varphi)(v) ~=~ (-1)^{|\varphi||a|}\,\varphi(-a_{-m} v)\ . 
\ee

\subsection{Indecomposable representations for $\h$ of dimension $0|2$}\label{sec:d=0|2-reps}

Let $\h = \Cb^{0|2}$ with basis $\chi^+$, $\chi^-$ and pairing $(\chi^+,\chi^-) = 1$. The mode algebra $\hat\h$ is (in the notation of \cite[Sect.\,2.1]{Kausch:2000fu})
\be
  [\chi^\alpha_m,\chi^\beta_n] = m \, d^{\alpha\beta} \delta_{m+n,0}
  \qquad , \quad \text{where} \quad d^{+-}=1~,~~d^{-+}=-1~,~~d^{++}=d^{--}=0 \ .
\ee
The indecomposable representations in $\Rep_{\flat,1}\fd(\hat\h)$ can be described explicitly. By Theorem \ref{thm:ind-equiv-untw} we reduce the problem to finding the indecomposable representations in $\Rep\fd(\h)$. The latter are classified in \cite{Zaitsev:1970}. On the other hand, if $\h$ has dimension $0|2n$ with $n>1$, the representation type is wild and indecomposable representations cannot be classified.

\medskip

To prepare the classification, we need to introduce two series of representations of $\h= \Cb^{0|2}$:
\begin{itemize}
\item
The representations $\xi_{k,\eps,\delta}$, where $k \in \Zb_{\ge 0}$ and $\eps,\delta \in \{0,1\}$, have underlying vector space $V = \Cb^{k+\eps+\delta|k+1}$. Fix a basis $e^{[\eps]}, e^1, \dots, e^k, e^{[\delta]}$ of $V_0$, where the basis vector $e^{[\eps]}$ is present only for $\eps{=}1$, and the basis vector $e^{[\delta]}$ is present only for $\delta{=}1$ (and $e^1,\dots,e^k$ are  present only for $k \ge 1$). Fix also a basis $o^1,\dots,o^{k+1}$ of $V_1$. Then
\ba
  &\chi^+ e^i = o^{i+1} ~,~~
  \chi^- e^i = o^i ~~ \text{for} ~~ 1 \le i \le k
  \quad ,  \quad 
  \chi^\pm o^i = 0 ~~ \text{for all $i$} \ ,
  \nonumber \\
  &
  \chi^+ e^{[\eps]} = o^1
  ~~ , \quad
  \chi^- e^{[\eps]} = 0
  ~~ , \quad
  \chi^+ e^{[\delta]} = 0 
  ~~ , \quad
  \chi^- e^{[\delta]} = o^{k+1} \ .
\label{eq:xi-action}
\ea
The subspace annihilated by both $\chi^+$ and $\chi^-$ is $V_1$, which has dimension $0|k{+}1$. The corresponding subspace of $\xi_{k,\eps,\delta} \otimes \Cb^{0|1}$ has dimension $k{+}1|0$. Thus the representations $\xi_{k,\eps,\delta}$ and $\xi_{k,\eps,\delta} \otimes \Cb^{0|1}$ are all mutually non-isomorphic. 
\item
The representations $\eta_{\mu,n}$, where $\mu \in \Cb^\times$ and $n \in \Zb_{>0}$, have underlying vector space $V = \Cb^{n|n}$. Denote by $e^1,\dots,e^n$ the standard basis of $V_0$ and by $o^1,\dots,o^n$ the standard basis of $V_1$. Then $\chi^+$ acts via a `Jordan block' and $\chi^-$ is `diagonal':
\be\label{eq:nu-action}
  \chi^+ e^i = \mu \cdot o^i + o^{i+1} ~\text{for}~1\le i < n
  ~~,~~
  \chi^+ e^n = \mu \cdot o^n
  ~~,~~
  \chi^- e^i = o^i \ ,
\ee
and again $\chi^\pm o^i = 0$ for all $i$.
The subspace annihilated by both $\chi^+$ and $\chi^-$ is $V_1$ of dimension $0|n$, and so $\eta_{\mu,n}$ and $\eta_{\mu,n}\otimes  \Cb^{0|1}$ are all mutually non-isomorphic. 
\end{itemize} 
We can ask whether $\eta_{\mu,n}$ could be isomorphic to some $\xi_{k,\eps,\delta}$ with $k=n{-}1$. Comparing \eqref{eq:xi-action} and \eqref{eq:nu-action}, we see that this would require $\eps=0$, $\delta=1$ (to make $\chi^-$ into an isomorphism $V_0 \to V_1$) and in particular $\mu=0$, which is excluded.

The representations $U(\h)$ and $U(\h) \otimes \Cb^{0|1}$ are the only indecomposable representations on which $\chi^+ \chi^-$ acts non-trivially \cite{Zaitsev:1970}. In particular, they cannot be isomorphic to one of  $\xi_{k,\eps,\delta}$ or $\eta_{\mu,n}$ or their parity shifted counterparts. 

We can now present:

\begin{theorem} \cite{Zaitsev:1970} Every non-zero finite-dimensional indecomposable representation of $\h= \Cb^{0|2}$ is isomorphic to precisely one of
\begin{enumerate}
\item $\xi_{k,\eps,\delta}$, where $k \in \Zb_{\ge 0}$ and $\eps,\delta \in \{0,1\}$; its dimension is $k{+}\eps{+}\delta\,|\,k{+}1$,
\item $\eta_{\mu,n}$, where $\mu \in \Cb^\times$ and $n \in \Zb_{>0}$; its dimension is $n|n$,
\item the regular representation of $\h$ on $U(\h)$ of dimension $2|2$,
\end{enumerate}
or to $\Cb^{0|1}$ tensored with one of the representations in 1.--3.
\end{theorem}

The two irreducible representations are $\xi_{0,0,0} \otimes \Cb^{0|1}$ of dimension $1|0$ and $\xi_{0,0,0}$ of dimension $0|1$. Their projective covers are $U(\h)$ and $U(\h)\otimes \Cb^{0|1}$, respectively.

\subsection{Twisted representations}

Let $\hat\h_\mathrm{tw}$ be the Lie algebra $\h \otimes (t^{\frac12} \Cb[t^{\pm 1}]) \oplus \Cb K$. We write $a_m := a \otimes t^m$, where now $m \in \Zb + \tfrac12$. As before, $K$ is central and
\be
  [a_m,b_n] = m \, (a,b) \, \delta_{m+n,0}
  \quad , \quad \text{where} \quad  a,b \in \h\,,~m,n\in\Zb+\tfrac12 \ .
\ee

A bounded-below $\hat\h_\mathrm{tw}$-representation and the category $\Rep_{\flat,1}(\hat\h_\mathrm{tw})$ are defined in the same way as in the untwisted case of Definition \ref{def:bounded-below-rep}. For a representation $M \in \Rep_{\flat,1}(\hat\h_\mathrm{tw})$ one can define the endomorphism $H$ as in \eqref{eq:H-op-def}, but with the sum over positive $m$ in $\Zb+\tfrac12$.

\begin{remark} \label{rem:twisted-reps}
(i) In conformal field theory language (or vertex operator algebra language), $\hat\h_\mathrm{tw}$ is the mode algebra for a representation around which the currents have the $\Zb/2\Zb$-monodromy $a(z) \mapsto -a(z)$, see \cite{Dixon:1986qv,Zamolodchikov:1987ae} and \cite[Sect.\,9]{Frenkel:1987} for the purely even and \cite{Kausch:2000fu} for the purely odd case. 

\smallskip\noindent
(ii) The Virasoro algebra acts on $M \in \Rep_{\flat,1}(\hat\h_\mathrm{tw})$ via $C = \sdim(\h) \cdot \id_M$ and
\be\label{eq:Virasoro-modes-twisted}
L_m = \tfrac12 \sum_{k \in \Zb+\frac12} \sum_{i \in \Ic} \beta^i_k \alpha_{m-k}^i ~~ (m \neq 0) 
\quad , \quad
L_0 = \hspace{-.8em} \sum_{m \in \Zb_{\ge 0} + \frac12} \sum_{i\in\Ic} \beta_{-m}^i \alpha_m^i +\tfrac{\sdim \h}{16} \ .
\ee
\end{remark}

Since $\hat\h_\mathrm{tw}$ contains no zero modes, the induction functor now starts with a super-vector space. Namely, let $\hat\h_{\mathrm{tw},>0}$ be the span of all $a_m$ with $m>0$, let $\hat\h_{\mathrm{tw},>0}$ act trivially on a super-vector space $V$, and let $K$ act by the identity map on $V$. The $\hat\h_\mathrm{tw}$-module induced by $V$ is then
\be
  \Indtw(V) = U(\hat\h_\mathrm{tw}) \otimes_{U(\hat\h_{\mathrm{tw},>0} \oplus \Cb K)} V \ .
\ee
Via the same argument as in Appendix \ref{app:ind-equiv} we obtain:

\begin{theorem} \label{thm:ind-equiv-tw}
$\Indtw : \svect \to \Rep_{\flat,1}(\hat\h_\mathrm{tw})$ is an equivalence of $\Cb$-linear categories.
\end{theorem}

In particular, since $\svect$ is semisimple with two simple objects (namely $\Cb^{1|0}$ and $\Cb^{0|1}$), so is $\Rep_{\flat,1}(\hat\h_\mathrm{tw})$.
As in the untwisted case we define $\Rep_{\flat,1}\fd(\hat\h_\mathrm{tw})$ as the full subcategory of representations with a finite dimensional subspace of ground states. By construction, $\Rep_{\flat,1}\fd(\hat\h_\mathrm{tw})$ is equivalent to $\svect\fd$. 

Just as $\Ind(R)$ for an $\h$-module $R$, the induced modules $\Indtw(X)$ for a super-vector space $X$ are direct sums of $H$-eigenspaces $\Indtw(X)_m$ ($m \in \Zb_{\ge 0}$) and we have $X \equiv \Indtw(X)_0$. Denote by $\overline\Indtw(X)$ the direct product of $H$-eigenspaces. The restricted dual $\Indtw(X)'$ is turned into a bounded-below $\hat h_\mathrm{tw}$-module by the same definition as in the untwisted case.

\section{Vertex operators and the tensor product functor}\label{sec:VO+tensor}

This section and Section \ref{sec:4pt-assoc} form the technical core of the paper. In this section, spaces of vertex operators are defined and computed. These spaces determine the tensor product functor and the braiding isomorphisms. The composition of two vertex operators, computed (and partially conjectured) in Section \ref{sec:4pt-assoc}, will later on fix the associator.

\subsection{Vertex operators} \label{sec:vertexop}


From Theorems \ref{thm:ind-equiv-untw} and \ref{thm:ind-equiv-tw} we see that the $\Zb/2\Zb$-graded category of representations of the untwisted and twisted mode algebra is equivalent to $\Rep(\h)$ and $\svect$, respectively. In the following we will only deal with locally finite $\h$-representations. This is necessary to ensure convergence of exponential series of zero modes occurring below. To abbreviate the notation, let us set
\be
  \Cc = \Cc_0 + \Cc_1 \quad , \quad \text{where} \quad
  \Cc_0 = \Rep\lf(\h) ~~,~~
  \Cc_1 = \svect \ .
\ee

\begin{definition} \label{def:vertex-op}
Let $A,B,C \in \Cc$. A {\em vertex operator from} $A,B$ {\em to} $C$ is a map $V$ with source and target as in
\begin{center}
\begin{tabular}{ccccl@{}c@{}l}
  & $A$ & $B$ & $C$ &&&
\\
a) & $\Cc_0$ & $\Cc_0$ & $\Cc_0$ & 
$\Rb_{>0} \times (A \otimes \Ind(B))$ &$\,\to~$& $\overline\Ind(C)$
\\
b) & $\Cc_0$ & $\Cc_1$ & $\Cc_1$ &
$\Rb_{>0} \times (A \otimes \Indtw(B))$ &$\,\to~$& $\overline\Indtw(C)$
\\
c) & $\Cc_1$ & $\Cc_0$ & $\Cc_1$ &
$\Rb_{>0} \times (A \otimes \Ind(B))$ &$\,\to~$& $\overline\Indtw(C)$
\\
d) & $\Cc_1$ & $\Cc_1$ & $\Cc_0$ &
$\Rb_{>0} \times (A \otimes \Indtw(B))$ &$\,\to~$& $\overline\Ind(C)$
\end{tabular}
\end{center}
such that 
\begin{enumerate}
\item[(i)] for all $x \in \Rb_{>0}$, $V(x)$ is an even linear map; 
for all $a\in A$, $\hat b \in \Indtwutw(B)$, $\hat\gamma \in \Indtwutw(C)'$, the function $x \mapsto \langle \hat \gamma , V(x)  (a \otimes \hat b) \rangle$ from $\Rb_{>0}$ to $\Cb$ is smooth;
\item[(ii)] $L_{-1} \circ V(x) - V(x)\circ (\id_A \otimes L_{-1}) = \tfrac{d}{dx} V(x)$, where $\tfrac{d}{dx} V(x)$ is defined in terms of matrix elements as in (i);
\item[(iii)] for all $a \in \h$,
\begin{itemize}
\item For all $m \in \Zb$ (case a) or $m \in \Zb {+} \frac12$ (case b)
\be
  a_m \circ V(x) = V(x) \circ \big( x^m \cdot a \otimes \id + \id \otimes a_m \big)
\ee
\item For all $m \in \Zb$ (case c) or $m \in \Zb {+} \frac12$ (case d)
\be\label{eq:mode-past-twisted}
  T(a)^{x,-}_{m+\frac12} \circ V(x) = i \, x^{\frac12} \cdot V(x) \circ \big( \id \otimes T(a)^{x,+}_m \big) \ ,
\ee
where $T(a)^{x,\eps}_m = \sum_{k=0}^\infty {{1/2} \choose k} (-x)^{-\eps k} \cdot a_{m + \eps k}$ for $\eps \in \{ \pm1\}$.
\end{itemize}
\end{enumerate}
The vector space of all vertex operators from $A,B$ to $C$ will be denoted by $\Vc_{A,B}^C$.
\end{definition}

\begin{remark}\label{rem:VO-def}
(i) Here, the commutation conditions in part (iii) are taken as a definition, but they are really obtained from the usual contour deformation argument. For cases c) and d) this is detailed in Section \ref{sec:contour-int}. Note that in cases c) and d) of part (iii), for $u \in \overline\Indtwutw(C)$, the infinite sum in $T(a)^{x,-}_{m+\frac12}u$ does not truncate because the modes $a_{m-k}$ become arbitrarily negative. Nonetheless, because the $H$-grading is bounded below by $0$, the contribution of $T(a)^{x,-}_{m+\frac12}u$ in each graded component of $\overline\Indtwutw(C)$ is a finite sum. 
\\[.3em]
(ii) In the theory of vertex operator algebras, one considers intertwining operators which have $\Ind(A) \otimes \Ind(B)$ as source instead of just $A \otimes \Ind(B)$. Logarithmic intertwining properties are treated in \cite{Milas:2001,Huang:2010,Tsuchiya:2012ru}. 
In the present explicit and simple example, I found it more straightforward to use the restricted definition above.
\\[.3em]
(iii) Vertex operators for twisted representations have been considered in many places in various contexts (though not precisely in the one treated here, to my knowledge). A small selection of references is  \cite[Sect.\,9]{Frenkel:1987} and \cite{Dolan:1989vr,Xu:1995,Gaberdiel:1996kf,Gaberdiel:2006pp}.
\end{remark}

Let $\Pgs$ be the map which projects to the ground states of an induced representation. That is, for $A \in \Cc$ we have $\Pgs : \overline\Indtwutw(A) \to A$. Denote by $F : \Rep(h) \to \svect$ the forgetful functor which takes a representation of $\h$ to the underlying super-vector space. To relate the spaces of vertex operators to morphisms in $\Cc$, define the map
\be\label{eq:VO-via-homs}
G ~:~ \Vc_{A,B}^C ~\longrightarrow ~
\left\{\rule{0pt}{2.8em}\right.\hspace{-.5em}\raisebox{.7em}{
\begin{tabular}{ccccl}
  & $A$ & $B$ & $C$ &
\\
a) & $\Cc_0$ & $\Cc_0$ & $\Cc_0$ & 
$\Hom_{\Rep(\h)}(A \otimes_{\Rep(\h)} B,C)$ ,
\\
b) & $\Cc_0$ & $\Cc_1$ & $\Cc_1$ &
$\Hom(F(A) \otimes B,C)$ ,
\\
c) & $\Cc_1$ & $\Cc_0$ & $\Cc_1$ &
$\Hom(A \otimes F(B),C)$ ,
\\
d) & $\Cc_1$ & $\Cc_1$ & $\Cc_0$ &
$\Hom(A \otimes B,F(C))$ .
\end{tabular}}
\ee
as
\be \label{eq:G-map-def}
  G(V) = \Pgs \circ V(1)|_{A \otimes B} ~~ : ~~ A \otimes B \to C \ .
\ee  
The tensor product $A \otimes_{\Rep(\h)} B$ in case a) stands for the tensor product of $\h$-representations. The underlying super-vector space is $A \otimes B$ and the action is via the coproduct $\Delta(x) = x \otimes 1 + 1 \otimes x$ ($x \in \h$) on $U(\h)$. That the image of $G$ in case a) is indeed an $\h$-intertwiner follows from condition (iii) in Definition \ref{def:vertex-op}.

Given $\h$-representations $R_1, \dots, R_n$, consider the product $W = R_1 \otimes R_2 \otimes \cdots \otimes R_n$ and define, for $1 \le i,j \le n$,
\be \label{eq:Omega-ij}
  \Omega^{(ij)} = \sum_{k \in \Ic} (\beta^k)^{(i)} \circ (\alpha^k)^{(j)}  \ ,
\ee
where $(\beta^k)^{(i)}$ acts on $R_i$ and $(\alpha^k)^{(j)}$ acts on $R_j$. With these notations, we can describe the restriction of vertex operators to ground states.

\begin{proposition} \label{prop:V-on-gnd}
Let $V \in \Vc_{A,B}^C$ and $f = G(V)$. Then $V$ restricted to ground states takes the form
\be\label{eq:VO-on-ground-states}
\Pgs \circ V(x)\big|_{A \otimes B} ~=~  
\left\{\rule{0pt}{2.8em}\right.
\hspace{-.5em}\raisebox{.7em}{
\begin{tabular}{ccccl}
  & $A$ & $B$ & $C$ &
\\
\rm a) & $\Cc_0$ & $\Cc_0$ & $\Cc_0$ & 
$f \circ \exp\!\big(\ln(x)\cdot \Omega^{(12)}\big)$ ,
\\
\rm b) & $\Cc_0$ & $\Cc_1$ & $\Cc_1$ &
$f \circ \exp\!\big(-\tfrac12\ln(x)\cdot \Omega^{(11)}\big)$ ,
\\
\rm c) & $\Cc_1$ & $\Cc_0$ & $\Cc_1$ &
$f \circ \exp\!\big(-\tfrac12\ln(x)\cdot \Omega^{(22)}\big)$ ,
\\
\rm d) & $\Cc_1$ & $\Cc_1$ & $\Cc_0$ &
$\exp\!\big(\tfrac12\ln(x)\,(-\tfrac{\sdim \h}4 + \Omega^{(11)})\big) \circ f$ ,
\end{tabular}}
\ee
each of which is an even linear map $A \otimes B \to C$.
\end{proposition}

\begin{proof}
In each of the four cases, we will give a first order differential equation which is satisfied by both sides of \eqref{eq:VO-on-ground-states}. Since for $x=1$, both sides of \eqref{eq:VO-on-ground-states} are equal to $f$, this will prove the proposition.

By condition (ii) in Definition \ref{def:vertex-op} we have, for all ground states $\gamma \in C^*$, $a \in A$, $b \in B$, 
\be
  \tfrac{d}{dx} \langle \gamma , V(x) (a \otimes b) \rangle = -\langle \gamma , V(x) (a \otimes L_{-1} b) \rangle \ .
\ee
From \eqref{eq:virasoro-untwisted} and \eqref{eq:Virasoro-modes-twisted} we see that
\be
  L_{-1}b = \sum_{i \in \Ic} \beta^i_{-1} \alpha^i_0 b
  \quad (B \in \Cc_0) \qquad , \qquad
  L_{-1}b = \tfrac12 \sum_{i \in \Ic} \beta^i_{-\frac12} \alpha^i_{-\frac12} b
  \quad (B \in \Cc_1) \ .
\ee

\noindent
{\em Case a).} 
By condition (iii) in Definition \ref{def:vertex-op} we have 
\be
  \langle \gamma , V(x) \, (\id \otimes (\beta^i_{-1} \alpha^i_0)) (a \otimes b) \rangle ~=~ 
- x^{-1} \langle \gamma , V(x)\,(\beta^i \otimes \alpha^i_0) (a \otimes b) \rangle \ , 
\ee
so that 
$\langle \gamma , V(x) (a \otimes L_{-1} b) \rangle = - x^{-1} \langle \gamma , V(x) \circ \Omega^{(12)} (a \otimes b) \rangle$. The resulting first order differential equation $\tfrac{d}{dx} \Pgs \circ V(x)|_{A\otimes B} = x^{-1}\,\Pgs \circ V(x) \circ \Omega^{(12)}|_{A\otimes B}$ is solved by expression a) in \eqref{eq:VO-on-ground-states}.

\medskip\noindent
{\em Case b).} 
Applying condition (iii) twice gives 
\be
  \langle \gamma , V(x)\, (\id \otimes (\beta^i_{-\frac12} \alpha^i_{-\frac12})) (a \otimes b) \rangle ~=~
  x^{-1} \langle \gamma , V(x) \, ((\beta^i \alpha^i) \otimes \id) (a \otimes b) \rangle\ , 
\ee
so that $\tfrac{d}{dx} \Pgs \circ V(x)|_{A\otimes B} = -\tfrac12 x^{-1}\,\Pgs \circ V(x) \circ \Omega^{(11)}|_{A\otimes B}$. This is solved by case b) on the right hand side of  \eqref{eq:VO-on-ground-states}.

\medskip\noindent
{\em Case c).} In condition (iii) we defined $T(\beta^i)^{x,+}_{-1} = \beta^i_{-1}
- \tfrac12 x^{-1} \beta^i_0 + \dots$, so that 
\ba
&  \langle \gamma , V(x)\,(\id \otimes (\beta^i_{-1} \alpha^i_0)) (a \otimes b) \rangle 
\nonumber\\
& = 
\langle \gamma , V(x)\, (\id \otimes (T(\beta^i)^{x,+}_{-1} \alpha^i_0)) (a \otimes b) \rangle +
\tfrac12 x^{-1} \langle \gamma ,  V(x)\, (\id \otimes (\beta^i_0 \alpha^i_0)) (a \otimes b) \rangle  \ .
\label{eq:x-dep-gs-case_c}
\ea
The first summand vanishes by condition (iii) and we are left with
$\tfrac{d}{dx} \Pgs \circ V(x)|_{A\otimes B} = -\tfrac12 x^{-1}\,\Pgs \circ V(x) \circ \Omega^{(22)}|_{A\otimes B}$. This shows case c) in \eqref{eq:VO-on-ground-states}.

\medskip\noindent
{\em Case d).} Here, $T(\beta^i)^{x,+}_{-\frac12} = \beta^i_{-\frac12} - \tfrac12 x^{-1} \beta^i_{\frac12} + \dots$. Using this and condition (iii), we find
\ba
&
\langle \gamma , V(x) \, (\id \otimes (\beta^i_{-\frac12} \alpha^i_{-\frac12})) (a \otimes b) \rangle 
\nonumber\\
&=
-i x^{-\frac12} \langle \gamma , T(\beta^i)^{x,-}_{0} \, V(x) \,  (\id \otimes  \alpha^i_{-\frac12}) (a \otimes b) \rangle +
\tfrac12 x^{-1} \langle \gamma , V(x) \, (\id \otimes (\beta^i_{\frac12} \alpha^i_{-\frac12})) (a \otimes b) \rangle
\nonumber\\
&=
-i x^{-\frac12} \langle \gamma , \beta^i_0 \, V(x) \,  (\id \otimes  T(\alpha^i)^{x,+}_{-\frac12}) (a \otimes b) \rangle +
\tfrac14 (\beta^i,\alpha^i) x^{-1} \langle \gamma ,  V(x)(a \otimes b) \rangle
\nonumber\\
&=
- x^{-1} \langle \gamma , \beta^i_0 \alpha^i_0 \, V(x) (a \otimes b) \rangle +
\tfrac14 (-1)^{|\alpha^i|} x^{-1} \langle \gamma ,  V(x)(a \otimes b) \rangle \ .
\ea
Thus, $\tfrac{d}{dx} \Pgs \circ V(x)|_{A\otimes B} = x^{-1} \big\{ \tfrac12 \Omega^{(11)}- \tfrac18 \sdim(\h)  \big\} \circ \Pgs \circ V(x)|_{A\otimes B}$, which shows case d) in \eqref{eq:VO-on-ground-states}.
\end{proof}

\subsection{Two examples: $\h$ of dimension $1|0$ and $0|2$}\label{sec:examples-VO}
 
\subsubsection*{Single free boson}
 
Take $\h = \Cb^{1|0}$ and choose $J \in \h$ with $(J,J)=1$. The affinisation $\hat\h$ is the mode algebra for a single free boson $[J_m,J_n] = m \, \delta_{m+n,0} \, K$. Let $\Cb_p$ be the one-dimensional representation of $\h$ for which $J$ acts by multiplication with $p$. Let $v_p$ be a highest weight vector in $\Ind(\Cb_p)$ and let $\sigma$ be a highest weight vector in $\Indtw(\Cb)$. From \eqref{eq:virasoro-untwisted} and \eqref{eq:Virasoro-modes-twisted} one sees that the eigenvalue of $C$ is $1$ and that
\be \label{eq:L0-eval-freebos}
  L_0 \, v_p = \tfrac12 p^2 \, v_p
  \quad , \qquad
  L_0 \,\sigma = \tfrac{1}{16} \, \sigma \ ,
\ee
as required for the conformal weights of a primary field in the untwisted and twisted sector of a single free boson. Then $\Omega^{(12)}(v_p \otimes v_q) = pq \cdot v_p \otimes v_q$, etc., and Proposition \ref{prop:V-on-gnd} gives:
\be\begin{array}{cccl}
\text{a)} & 
\big\langle\,v_r^*\,,\,V(x)\,(v_p \otimes v_q)\,\big\rangle &=& \text{(const)} \cdot \delta_{r,p+q} \, x^{pq}
\\[.3em]
\text{b)} & 
\big\langle\,\sigma^*\,,\,V(x)\,(v_p \otimes \sigma)\,\big\rangle &=& \text{(const)} \cdot x^{-\frac12 p^2}
\\[.3em]
\text{c)} & 
\big\langle\,\sigma^*\,,\,V(x)\,(\sigma \otimes v_q)\,\big\rangle &=& \text{(const)} \cdot x^{-\frac12 q^2}
\\[.3em]
\text{d)} & 
\big\langle\,v_r^*\,,\,V(x)\,(\sigma \otimes \sigma)\,\big\rangle &=& \text{(const)} \cdot x^{-\frac18 + \frac12 r^2}
\end{array}
\ee
Each of these is of the form $\langle a^*, V(x) (b \otimes c) \rangle = \text{(const)} \cdot x^{h_a-h_b-h_c}$, where $h_a$ is the conformal weight of $a$ as given in \eqref{eq:L0-eval-freebos}. The factor $\delta_{r,p+q}$ in case a) arises as according to \eqref{eq:VO-via-homs}, $G(V)$ has to be an $\h$-intertwiner from $\Cb_p \otimes \Cb_q$ to $\Cb_r$, and these can only be non-zero for $p+q=r$ (charge conservation).

\subsubsection*{Single pair of symplectic fermions}

Take $\h = \Cb^{0|2}$ with basis $\chi^\pm$ and metric $(\chi^+,\chi^-)=1$ as in Section \ref{sec:d=0|2-reps}. This amounts to a single pair of symplectic fermions. According to \eqref{eq:virasoro-untwisted}, the central charge is $c=-2$. The copairing \eqref{eq:copair-Omega-def} and its powers are
\be \label{eq:d=0|2-example-Omega}
  \Omega = \chi^- \otimes \chi^+ - \chi^+ \otimes \chi^- 
  ~~ , \quad
  \Omega^2 = -2 \cdot (\chi^+ \chi^-)  \otimes (\chi^+ \chi^-)
  ~~ , \quad
  \Omega^n = 0 ~~ (n>2) \ .
\ee

Consider the $2|2$-dimensional regular representation $U(\h)$. To match the usual nomenclature for symplectic fermions (as e.g.\ in \cite{Gaberdiel:1996np,Kausch:2000fu,Gaberdiel:2006pp}), let us write $1 \equiv \omega$ for the unit in $U(\h)$, as well as
\be
  \hat\Omega = L_0 \omega = - \chi^+ \chi^- \omega \ .
\ee
(The notation in \cite{Gaberdiel:1996np,Kausch:2000fu,Gaberdiel:2006pp} is $\Omega$ instead of $\hat\Omega$, but this could be confused with the copairing).
Then a basis of $U(\h)$ is $\{ \omega , \chi^+ \omega , \chi^- \omega , \hat\Omega \}$. Case a) in Proposition \ref{prop:V-on-gnd} -- with $A=B=C=U(\h)$ and applied to $\omega \otimes \omega$ -- becomes
\ba
  f \circ e^{ \ln(x) \cdot \Omega^{(12)}}(\omega \otimes \omega)
  = f(\omega \otimes \omega) &+ \ln(x) \cdot f( \chi^- \omega \otimes \chi^+ \omega - \chi^+ \omega\otimes \chi^-\omega)
  \nonumber \\ 
  &- (\ln(x))^2 \cdot f(\hat\Omega \otimes \hat\Omega) 
  \label{eq:case-a-omom}
\ea
for an $\h$-intertwiner $f : U(\h) \otimes U(\h) \to U(\h)$. To describe all such intertwiners, we need a little lemma on Hopf algebras. Recall Sweedler's notation for the coproduct $\Delta(a) = \sum_{(a)} a_{(1)} \otimes a_{(2)}$.

\begin{lemma}
For a Hopf algebra $H$ in $\svect$, the following map is an isomorphism:
$$
  \Hom(H,H) \to \Hom_{\Rep(H)}(H \otimes H, H) ~~ , \quad
  g \mapsto \Big( a \otimes b \mapsto \sum_{(a)} a_{(1)} \cdot g\big(S(a_{(2)}) \cdot b\big) \Big) \ .
$$
\end{lemma}

\begin{proof}
Denote the above map by $\varphi$. Firstly, to see that $\varphi(g)$ is an $H$-module map, compute
\ba
  \varphi(g)\big( \Delta(c) \cdot (a \otimes b) \big)
  &= \sum_{(a),(c)} c_{(1)} a_{(1)} g\big( S(c_{(2)} a_{(2)}) c_{(3)} b \big)
  \nonumber\\
  &= \sum_{(a),(c)} c_{(1)} a_{(1)} g\big( S(a_{(2)}) S(c_{(2)}) c_{(3)} b \big)
  \nonumber\\
  &  \overset{(*)}= \sum_{(a)} c \, a_{(1)} g\big( S(a_{(2)}) b \big) = c \cdot \varphi(g)(a \otimes b) \ .
\ea
In step (*) the property $\sum_{(x)} S(x_{(1)}) x_{(2)} = \eps(x) \cdot 1$ of the antipode was used. 

Injectivity of $\varphi$ follows from $\varphi(g)(1\otimes b) = g(b)$. For surjectivity, let $f : H \otimes H \to H$ be an $H$-module map. Then for $g(x) = f(1 \otimes x)$ we have
\be
  \varphi(g)(a \otimes b)
=  \sum_{(a)}  a_{(1)} f\big(1 \otimes  S( a_{(2)}) b \big)
\overset{(*)}=  \sum_{(a)}  f\big(a_{(1)}  \otimes a_{(2)}  S( a_{(3)}) b \big)
\overset{(**)}=   f(a \otimes b) \ .
\ee
Here, (*) is the intertwining property $x\cdot f(u \otimes v) = f(\Delta(x) \cdot u \otimes v)$ of $f$, and (**) is the property $\sum_{(x)} S(x_{(1)}) x_{(2)} = \eps(x) \cdot 1$ of the antipode.
\end{proof}

The above lemma shows that the most general $\h$ intertwiner $U(\h) \otimes U(\h) \to U(\h)$ is of the form $f(a \otimes b) = \sum_{(a)} a_{(1)} g\big(S(a_{(2)}) b\big)$ for an arbitrary even linear map $g : U(\h) \to U(\h)$. The various coproducts in $U(\h)$ are
\ba
& \Delta(1) = 1 \otimes 1
\quad , \qquad
\Delta(\chi^\pm) = \chi^\pm \otimes 1 + 1 \otimes \chi^\pm \ ,
 \nonumber \\
&\Delta(\chi^+ \chi^-) = (\chi^+ \chi^-) \otimes 1 + \chi^+ \otimes \chi^- - \chi^- \otimes \chi^+ + 1 \otimes (\chi^+ \chi^-) \ .
\ea
Inserting all this into \eqref{eq:case-a-omom} gives (recall that $\omega=1$ in $U(\h)$)
\ba
  f \circ \exp( \ln(x) \cdot \Omega^{(12)})(\omega \otimes \omega)
  = g(\omega) &+ \ln(x) \cdot ( \chi^- g(\chi^+ \omega) - \chi^+ g(\chi^- \omega) + 2 g(\chi^+\chi^- \omega) )
  \nonumber \\ 
  &- \ln(x)^2 \cdot (\chi^+ \chi^- g( \chi^+ \chi^- \omega) ) \ ,
  \label{eq:sympferm-VA-gndstate}
\ea
in agreement with \cite[Eqn.\,(D.6)]{Gaberdiel:2006pp}. For example, if $g$ maps the basis vector $\hat\Omega$ to $\omega$ and the rest to zero, we have the following three-point functions for the vertex operator $V(x)$:
\ba
  \big\langle\,\omega^* \, , \, V(x) \,(\omega \otimes \omega) \big\rangle &= -2 \ln(x) \ ,
  \nonumber \\
  \big\langle\,(\chi^\pm \omega)^* \, , \, V(x) \,(\omega \otimes \omega) \big\rangle &= 0 \ ,
  \nonumber \\
  \big\langle\,\hat\Omega^* \, , \, V(x) \,(\omega \otimes \omega) \big\rangle
  &= - \ln(x)^2 \ ,
\ea
where $(-)^*$ denotes elements of the dual basis.

\subsection{Normal ordered exponentials}

This section contains a formulation of vertex operators in terms of normal ordered exponentials which can accommodate non-semisimple representations and parity odd generators. 

\subsubsection*{Untwisted representations}

For $\h$ purely even (the free boson case), and when restricting to irreducible representations, the vertex operators are the text-book normal ordered exponentials (see Example \ref{ex:VOs} below). Vertex operators between indecomposable representations are constructed in \cite[Thm.\,7.4]{Milas:2008}. The new contribution here is to allow for odd directions.

\medskip

To start with, we need the following map:\footnote{
  In the context of free boson vertex operators (assume that $\h$ is $1|0$-dimensional for simplicity), the map $Q$ is usually written as $\exp( (\text{const}) \pi_0)$, where  $\pi_0$ is an additional operator which is conjugate to $a_0$ in the sense that $[a_m,\pi_0] = \delta_{m,0}$. Here we avoid the problem of specifying what space $\pi_0$ acts on and give $Q$ directly.}


\begin{lemma} \label{lem:Qf-exists}
Let $R,S,T \in \Rep(\h)$ and let $f : R \otimes S \to T$ be an $\h$-intertwiner. There is a unique even linear map $Q_f : R \otimes \Ind(S) \to \Ind(T)$ such that 
\begin{enumerate}
\item for all $r \in R$, $s \in S$ we have $Q_f(r \otimes s) = f(r \otimes s)$,
\item for all $a \in \h$, $m \in \Zb$ we have
$a_m \circ Q_f = Q_f \circ \big( \delta_{m,0} \cdot a \otimes \id \,+\, \id \otimes a_m \big)$.
\end{enumerate}
\end{lemma}

\begin{proof}
If $Q_f$ exists, it is clear that it is unique since by property 1 and 2, for all $a^i \in \h$, $m_i>0$,
\be
  Q_f\big(\, r \,\otimes\, a^1_{-m_1} \cdots a^k_{-m_k} s \,\big) \,=\, (-1)^{|r|(|a^1|+\cdots+|a^k|)}\, a^1_{-m_1} \cdots a^k_{-m_k} f(r \otimes s) \ ,
\ee
and the vectors $a^1_{-m_1} \cdots a^k_{-m_k} s$ span $\Ind(S)$. The existence proof is also easy but uses some notation introduced in the appendix and is therefore moved to Appendix \ref{app:lemma-Qf-exists}. 
\end{proof}

From property 2 we see that
\be
 H \circ Q_f = Q_f \circ (\id_R \otimes H) \ .
\ee
Write $\overline{R \otimes \Ind(S)}$ for the algebraic completion of $R \otimes \Ind(S)$ with respect to the $\id \otimes H$-grading. The above equality shows that $Q_f$ extends to a linear map $Q_f : \overline{R \otimes \Ind(S)} \to \overline\Ind(T)$.

\medskip

We now turn to exponentials of modes. To ensure these are well-defined we restrict ourselves to locally finite representations for the domain of the exponential maps. Pick thus
\be
  R,S \in \Rep\lf(\h)
  \quad , \quad
  T \in \Rep(\h) \ .
\ee
Define the three maps
\be \label{eq:E-op-untw}
\begin{array}{r@{\,=\,}lcr@{\,\to\,}l}
E_+(x) &
\displaystyle
\exp\!\Big(\sum_{m>0} \tfrac{x^{-m}}{-m} \sum_i \beta^i \otimes \alpha^i_m \Big)
&:& R \otimes \Ind(S) & R \otimes \Ind(S) \ ,
\\[1.3em]
E_0(x) &
\displaystyle
\exp\!\Big( \ln(x) \cdot \sum_i \beta^i \otimes \alpha^i_0 \Big) 
&:& R \otimes \Ind(S) & R \otimes \Ind(S) \ ,
\\[1.3em]
E_-(x) & 
\displaystyle
\exp\!\Big(\sum_{m<0} \tfrac{x^{-m}}{-m} \sum_i \beta^i \otimes \alpha^i_m \Big)
&:& R \otimes \Ind(S) & \overline{R \otimes \Ind(S)} \ .
\end{array}
\ee
The map $E_+(x)$ indeed has image in $R\otimes \Ind(S)$ since on a given element of $R\otimes \Ind(S)$, the exponential reduces to a finite sum. For $E_0(x)$, note that $[\id_R \otimes H, \sum_i \beta^i \otimes \alpha^i_0] = 0$ and pick an $\id \otimes H$-eigenvector $v$ of $R\otimes \Ind(S)$. Since $R,S \in \Rep\lf(\h)$, the subspace of the $(\id \otimes H)$-eigenspace generated by the action of $\beta^i \otimes \alpha^i_0$ on $v$ is finite-dimensional, and so the exponential $E_0$ converges to an endomorphism on that subspace. Finally, we define
\be \label{eq:Vf-untwist-def}
  V_f(x) = Q_f \circ E_-(x) \circ E_0(x) \circ E_+(x)
  \quad : ~ R \otimes \Ind(S) \to \overline\Ind(T) \ .
\ee

\begin{lemma}\label{lem:Vf-VO-untwisted}
$V_f$ is a vertex operator from $R,S$ to $T$. 
\end{lemma}

Before turning to the proof, let us recover the more accustomed free boson case:

\begin{example} \label{ex:VOs}
Single free boson, $\h = \Cb^{1|0}$: Let $\Cb_p$ denote the irreducible representation of $J$-eigenvalue $p$ as in Section \ref{sec:examples-VO}. Choose $R = \Cb_r$, $S = \Cb_s$. We have $\Omega = J \otimes J$ and so $J \otimes J_m$ acts on $R \otimes \Ind(S)$ as $r \, J_m$. The normal ordered exponential \eqref{eq:Vf-untwist-def} becomes
\be
  V_f(x) = x^{rs} \cdot Q_f \, 
  \exp\!\Big(r \sum_{m<0} \tfrac{x^{-m}}{-m} J_m \Big) \,
  \exp\!\Big(r \sum_{m>0} \tfrac{x^{-m}}{-m} J_m \Big) \ ,
\ee
which is the more familiar expression for vertex operators in the free boson case. 
\end{example}

The notation needed for the proof of lemma \ref{lem:Vf-VO-untwisted} will be set up slightly more generally in order to avoid repetition later on. Fix representations $R_1, \dots, R_s \in \Rep\lf(\h)$ and consider the product 
\be\label{eq:W-product-aux}
  W = R_1 \otimes R_2 \otimes \cdots \otimes R_{s-1} \otimes \Ind(R_s) \ .
\ee
For $a \in \h$ let $a^{(i)}$ be the endomorphism which acts as $a$ on the $i$'th tensor factor and as identity on the others, e.g.\ $a^{(1)} \equiv a \otimes \id_{R_2} \otimes \cdots \otimes \id_{\Ind(R_s)}$. We will also write $a^{(s)}_m \equiv \id_{R_1} \otimes \cdots \otimes \id_{R_{s-1}} \otimes a_m$. 
Similar to \eqref{eq:Omega-ij}, on a product of the form $W$ we set, for $1 \le i,j < s$,
\be \label{eq:Omega-ij-v1}
  \Omega^{(ij)} = \sum_{k \in \Ic} (\beta^k)^{(i)} \circ (\alpha^k)^{(j)} \qquad , \qquad
  \Omega^{(is)}_m = \sum_{k \in \Ic} (\beta^k)^{(i)} \circ (\alpha^k_m)^{(s)} \ ,
\ee
where $\{\alpha^k\}$ and $\{\beta^k\}$ is the dual basis pair from Section \ref{sec:lie-super}. For $m \neq 0$ define $E^{(is)}_m(x)  = \exp( \frac{x^{-m}}{-m} \Omega_m^{(is)})$. The normal ordered exponentials in \eqref{eq:E-op-untw}, acting on the $i$'th and the last factor in the product $W$, can be written as
\be \label{eq:E-ij-v1}
  E^{(is)}_+(x) = \prod_{m>0} E^{(is)}_m(x)
  ~~ , \quad
  E^{(is)}_0(x) = \exp( \ln(x) \Omega_0^{(is)})
  ~~ , \quad
  E^{(is)}_-(x) = \prod_{m<0} E^{(is)}_m(x)
\ee  
With this notation, the vertex operator \eqref{eq:Vf-untwist-def} is $V_f = Q_f \circ E_-^{(12)}(x) \circ E_0^{(12)}(x) \circ E_+^{(12)}(x)$. We shall need that for all $h \in \h$, $m,n \in \Zb$,
\ba
  h^{(s)}_m \Omega^{(is)}_n
  &= ~
  \sum_{k} h^{(s)}_m \, (\beta^k)^{(i)} \, (\alpha^k_n)^{(s)}
  =~
  \sum_{k} (-1)^{|h||\beta^k|} \, (\beta^k)^{(i)} \, h^{(s)}_m \, (\alpha^k_n)^{(s)}
\nonumber \\
  &=~
  \sum_{k} (-1)^{|h||\beta^k|}  \,  (\beta^k)^{(i)} \big( (-1)^{|h||\alpha^k|} (\alpha^k_n)^{(s)}  \,  h^{(s)}_m + [h^{(s)}_m , (\alpha^k_n)^{(s)} ]  \big)
\nonumber \\
  &=~ \Omega^{(is)}_n \, h^{(s)}_m ~+~ 
  m \, \delta_{m+n,0}  \sum_{k} (-1)^{|h||\beta^k|}  \,  (h,\alpha^k) \cdot (\beta^k)^{(i)}  
\nonumber \\
  &=~   \Omega^{(is)}_n \, h^{(s)}_m ~+~ m\, \delta_{m+n,0}  \, h^{(i)} \ .
  \label{eq:h-Om-comm}
\ea
In the last step we used the second of the sum rules ($h \in \h$)
\be\label{eq:metric-sum-rules}
  h = \sum_{k} (\alpha^k,h) \cdot \beta^k = \sum_{k} (-1)^{|h||\beta^k|} \, (\beta^k,h) \cdot \alpha^k \ .
\ee
Since $[h^{(i)},\Omega^{(is)}_n] = 0$, the identity \eqref{eq:h-Om-comm} implies that for $n \neq 0$,
\ba
  h^{(s)}_m  \, E^{(is)}_n(x) 
  &=  
  E^{(is)}_n(x) \, E^{(is)}_n(x)^{-1} \, h^{(s)}_m  \, E^{(is)}_n(x)
  =
  E^{(is)}_n(x) \, e^{-\frac{x^{-n}}{-n} \mathrm{ad}_{\Omega^{(is)}_n}}(h^{(s)}_m)
\nonumber \\
  &=
  E^{(is)}_n(x) \, \big( \, h^{(s)}_m + x^m \delta_{m+n,0} \cdot h^{(i)} \, \big) \ .
  \label{eq:h-Em-comm}
\ea

\begin{proof}[Proof of Lemma \ref{lem:Vf-VO-untwisted}]
We need to verify the properties in case a) of Definition \ref{def:vertex-op}. Conditions (i) and (iii) are easy, (ii) takes a small calculation.

\medskip\noindent
\nxt Condition (i): In a given matrix element, $E_\pm(x)$ result in finite sums and $E_0(x)$ is smooth for each summand. 

\medskip\noindent
\nxt Condition (iii): For $m=0$ this is immediate from property 2 of $Q_f$ in Lemma \ref{lem:Qf-exists} and for $m\neq 0$ combine property 2 with \eqref{eq:h-Em-comm}.

\medskip\noindent
\nxt Condition (ii): As usual one verifies that $[L_m,a_n] = - n a_{m+n}$ for all $a \in \h$ and $m,n\in \Zb$ (combine \eqref{eq:affine-commutator}, \eqref{eq:virasoro-untwisted} and \eqref{eq:metric-sum-rules}). From this it follows that for $m \neq 0$
\be \label{lem:Vf-untw-aux3}
  \big[\id \otimes L_{-1}, E^{(12)}_m(x)\big] = \big[ \id \otimes L_{-1},  \tfrac{x^{-m}}{-m} 
  \sum_{k } \beta^k \otimes \alpha^k_m \big]  E^{(12)}_m(x)
  = x^{-m} \, \Omega^{(12)}_{m-1} \, E^{(12)}_m(x) \ .
\ee
Note that we implicitly used $[\Omega^{(12)}_{m-1},\Omega^{(12)}_{m}]=0$ for all $m \in \Zb$. (In the twisted case this will cause a complication for $m=\frac12$, see Lemma \ref{lem:Vf-VO-twisted} below.) It follows that
\ba
   \big[\id \otimes L_{-1}, E_-(x) \big] &= \Big( \sum_{m=-\infty}^{-1} x^{-m} \, \Omega^{(12)}_{m-1} \Big) E_-(x) \ ,
\nonumber \\
   \big[\id \otimes L_{-1}, E_0(x)\, \big] &= 0 \ ,
\nonumber \\
   \big[\id \otimes L_{-1}, E_+(x) \big] &= \Big( \sum_{m=1}^\infty x^{-m} \, \Omega^{(12)}_{m-1} \Big) E_+(x) \ .
\ea
For the $x$-derivatives, on the other hand,  we find
\ba
   \tfrac{d}{dx}  E_-(x) &= \Big( \sum_{m=-\infty}^{-1} x^{-m-1}\, \Omega^{(12)}_{m} \Big) E_-(x) 
  = \Big( \sum_{m=-\infty}^{0} x^{-m} \,\Omega^{(12)}_{m-1} \Big) E_-(x) 
   \ ,
\nonumber \\
   \tfrac{d}{dx}  E_0(x) &= x^{-1} \,\Omega^{(12)}_0 E^{(12)}_0(x) \ ,
\nonumber \\
   \tfrac{d}{dx}  E_+(x) &= \Big( \sum_{m=1}^\infty x^{-m-1} \,\Omega^{(12)}_{m} \Big) E_+(x) 
   =  \Big( \sum_{m=2}^\infty x^{-m} \,\Omega^{(12)}_{m-1} \Big) E_+(x) \ .
\ea
From this we read off
\ba
  \big[ \, \id \otimes L_{-1} \,,\, E_-(x) E_0(x) E_+(x) \,\big] + \Omega^{(12)}_{-1} \, E_-(x) E_0(x) E_+(x) ~=~ \tfrac{d}{dx} \big( E_-(x) E_0(x) E_+(x) \big)
    \ .
    \label{lem:Vf-untw-aux1}
\ea
The difference will cancel against a corresponding term coming from $Q_f$. Namely,
\ba
L_{-1} \, Q_f 
& \overset{(*)}= Q_f
\sum_{i \in \Ic} \sum_{k \in\Zb} \tfrac12 
\big( \delta_{k,0} \beta^i \otimes \id + \id \otimes \beta^i_k\big)
\big( \delta_{k,-1} \alpha^i \otimes \id + \id \otimes \alpha^i_{-k-1} \big)
\nonumber \\
&= Q_f \, \big( \Omega^{(12)}_{-1} + \id \otimes L_{-1} \big) \ ,
    \label{lem:Vf-untw-aux2}
\ea
where (*) amounts to combining the expression \eqref{eq:Lm-nonzero} for $L_{-1}$ with property 2 in Lemma \ref{lem:Qf-exists}. Together, \eqref{lem:Vf-untw-aux1} and \eqref{lem:Vf-untw-aux2} establish condition (ii).
\end{proof}

\subsubsection*{Twisted representations}

The construction of vertex operators from (untwisted)$\times$(twisted)$\to$(twisted) proceeds parallel to the untwisted case. To start with, define the map $\tilde Q_f$ via:

\begin{lemma} \label{lem:Qf-exists-tw}
Let $R \in \Rep(\h)$ and $X,Y \in \svect$ and let $f : R \otimes X \to Y$ be an even linear map. There is a unique even linear map $\tilde Q_f : R \otimes \Indtw(X) \to \Indtw(Y)$ such that 
\begin{enumerate}
\item for all $r \in R$, $x \in X$ we have $\tilde Q_f(r \otimes x) = f(r \otimes x)$,
\item for all $a \in \h$, $m \in \Zb+\tfrac12$ we have
$a_m \circ \tilde Q_f = \tilde Q_f \circ (\id \otimes a_m )$.
\end{enumerate}
\end{lemma}

The proof is a simplified version of the proof of Lemma \ref{lem:Qf-exists} given in Appendix \ref{app:lemma-Qf-exists} (due to the absence of zero modes) and has been omitted.

\medskip

Fix $R \in \Rep\lf(\h)$ and $X,Y \in \svect$. As before, the reason to take $R$ locally finite is to make the normal ordered exponential well-defined. Define maps
$\tilde E_+(x) : R \otimes \Indtw(X) \to R \otimes \Indtw(X)$
and
$\tilde E_-(x) : R \otimes \Indtw(X) \to \overline{R \otimes\Indtw(X)}$
by the same expressions as $E_\pm(x)$ in \eqref{eq:E-op-untw}, except that the summand $m$ lies in $\Zb+\frac12$. The expression for $\tilde E_0$ is different from the untwisted case, namely
\be
  \tilde E_0(x) = \exp\big(- \tfrac12 \ln(x) \, \Omega^{(11)}\big) 
  ~:~ R \otimes \Indtw(X) \to R \otimes \Indtw(X) \ .
\ee
Finally, set
\be \label{eq:Vf-twisted-def}
  V_f(x) = \tilde Q_f \circ \tilde E_-(x) \circ \tilde E_0(x) \circ \tilde E_+(x) \ .
\ee  
If $\h = \Cb^{1|0}$ and $R$ is one-dimensional, this agrees with \cite[Eqn.\,(3.16)]{Dolan:1989vr}.

\begin{lemma} \label{lem:Vf-VO-twisted}
$V_f$ is a vertex operator from $R,X$ to $Y$.
\end{lemma}

Before giving the proof, let us adapt the notation used in the untwisted case as follows. On a product $W$ as in \eqref{eq:W-product-aux} with $R_i \in \Cc_0$ for $i=1,\dots,s-1$ and $R_s \in \Cc_1$ define $\Omega^{(is)}_m$ as in \eqref{eq:Omega-ij-v1}, but with $m \in \Zb+\frac12$. If we write $\tilde E^{(is)}_m(x)  = \exp( \frac{x^{-m}}{-m} \Omega_m^{(is)})$ for $m \in \Zb+\frac12$, the expressions for $\tilde E_\pm$, $\tilde E_0$ take the form
\be \label{eq:tilde-E^is_m-def}
  \tilde E^{(is)}_+(x) = \prod_{m>0} \tilde E^{(is)}_m(x)
  ~~ , \quad
  \tilde E^{(is)}_0(x) = \exp\!\big( - \tfrac12 \ln(x) \, \Omega^{(ii)}\big)
  ~~ , \quad
  \tilde E^{(is)}_-(x) = \prod_{m<0} \tilde E^{(is)}_m(x) \ .
\ee  
The commutation relations \eqref{eq:h-Om-comm} and \eqref{eq:h-Em-comm} still hold, but now with $m,n \in\Zb + \tfrac12$. Finally we need the following identity, which is an immediate consequence of \eqref{eq:h-Om-comm}: for all $1 \le i,j <s$
\be
  \big[ \, \Omega^{(is)}_m \,,\, \Omega^{(js)}_n \big] ~=~ m \, \delta_{m+n,0} \cdot \Omega^{(ij)}  \ .
\label{eq:Omis-Omjs-com}
\ee
This identity holds for $m,n \in \Zb$ and $m,n \in \Zb+\tfrac12$ alike.

\begin{proof}[Proof of Lemma \ref{lem:Vf-VO-twisted}]
Condition (i) in case b) of Definition \ref{def:vertex-op} holds by construction and condition (iii) is immediate from property 2 of $\tilde Q_f$ in Lemma \ref{lem:Qf-exists-tw} and \eqref{eq:h-Em-comm}. It remains to verify condition (ii). As in \eqref{lem:Vf-untw-aux3} in the untwisted case, one checks that for $m \neq \tfrac12$
\be \label{lem:Vf-tw-aux3}
  \big[\,\id \otimes L_{-1} \,,\, \tilde E^{(12)}_m(x) \,\big] 
  ~=~ x^{-m} \, \Omega^{(12)}_{m-1} \, \tilde E^{(12)}_m(x) \ .
\ee
For $m=\tfrac12$ this fails as $\Omega^{(12)}_{1/2}$ and $[L_{-1},\Omega^{(12)}_{1/2}]$ do not commute. To treat this case, first note that $[\id \otimes L_{-1},\Omega^{(12)}_{1/2}] = - \tfrac12 \Omega^{(12)}_{-1/2}$ and then verify inductively that
\be
   \big[ \, \id \otimes L_{-1} \,,\, (\Omega^{(12)}_{1/2})^n \big] ~=~ - \tfrac n2 \cdot \Omega^{(12)}_{-1/2} \,(\Omega^{(12)}_{1/2})^{n-1} - \tfrac{n(n-1)}8 \cdot \Omega^{(11)} \, (\Omega^{(12)}_{1/2})^{n-2} \ .
\ee
From this one finds an extra contribution to \eqref{lem:Vf-tw-aux3} in the case $m=\frac12$:
\be \label{lem:Vf-tw-aux4}
  \big[\, \id \otimes L_{-1} \,,\, \tilde E^{(12)}_{1/2}(x) \,\big] 
  ~=~ \big( x^{-\frac12}  \, \Omega^{(12)}_{-1/2} - \tfrac12 x^{-1} \Omega^{(11)} \big) \, \tilde E^{(12)}_{1/2}(x) \ .
\ee
Altogether this gives
\ba
   \big[\id \otimes L_{-1}, \tilde E_-(x) \big] &= \Big( \sum_{m=-\infty}^{-1/2} x^{-m} \, \Omega^{(12)}_{m-1} \Big) \tilde E_-(x) \ ,
\nonumber \\
   \big[\id \otimes L_{-1}, \tilde E_0(x)\, \big] &= 0 \ ,
\nonumber \\
   \big[\id \otimes L_{-1}, \tilde E_+(x) \big] &= 
   - \tfrac12 x^{-1} \Omega^{(11)} \, \tilde E_+(x) ~+~
   \Big( \sum_{m=1/2}^\infty x^{-m} \, \Omega^{(12)}_{m-1} \Big) \tilde E_+(x) \ .
\ea
For the $x$-derivatives the above complication of non-commuting modes does not arise as all modes $\Omega^{(12)}_m$ with $m>0$ commute amongst themselves, and the same holds for all modes with $m<0$. Thus we simply get (after shifting the summating index as in \eqref{lem:Vf-untw-aux1})
\ba
   \tfrac{d}{dx}  \tilde E_-(x) &= \Big( \sum_{m=-\infty}^{1/2} x^{-m} \,\Omega^{(12)}_{m-1} \Big) \tilde E_-(x) 
   \ ,
\nonumber \\
   \tfrac{d}{dx}  \tilde E_0(x) &= - \tfrac12 x^{-1} \,\Omega^{(11)} \tilde E_0(x) \ ,
\nonumber \\
   \tfrac{d}{dx}  \tilde E_+(x) &= \Big( \sum_{m=3/2}^\infty x^{-m} \,\Omega^{(12)}_{m-1} \Big) \tilde E_+(x) \ .
\ea
It follows that
\be
  \big[ \, \id \otimes L_{-1} \,,\, \tilde E_-(x) \tilde E_0(x) \tilde E_+(x) \,\big]  ~=~ \tfrac{d}{dx} \big( \tilde E_-(x) \tilde E_0(x) \tilde E_+(x) \big)
    \ .
\ee
From property 2 in Lemma \ref{lem:Qf-exists-tw} one sees that $L_{-1} \tilde Q_f = \tilde Q_f ( \id \otimes L_{-1})$, and together this establishes condition (ii).
\end{proof}

\subsection{Tensor product functor}

The tensor product functor will be defined indirectly via representing objects. Recall that
\be
  \Cc = \Cc_0 + \Cc_1 \quad , \quad \text{where} \quad
  \Cc_0 = \Rep\lf(\h) ~~,~~
  \Cc_1 = \svect \ .
\ee
The map $\Cc \times \Cc \times \Cc \to \vect$, $(A,B,C) \mapsto \Vc_{A,B}^C$ becomes a functor, contravariant in the first two arguments and covariant in the last, if we assign to a triple of morphisms $(f,g,h) : (A,B,C) \to (A',B',C')$ the linear map
$\Vc_{f,g}^h = \overline\Indtwutw(h) \circ (-) \circ (f \otimes \Indtwutw(g))$ from $\Vc_{A',B'}^C$ to $\Vc_{A,B}^{C'}$.

\begin{definition} \label{def:*-tensor-def}
The {\em tensor product} $A \ast B$ of two objects $A,B \in \Cc$ is the representing object in $\Rep(\h)+\svect$ for the functor $\Vc_{A,B}^{(-)} : \Cc \to \vect$.
\end{definition}

We will see in a moment that for $X,Y \in \Cc_1$ the representing object $X \ast Y$ is a locally finite $\h$-module if and only if $\h$ is purely odd. This will be discussed in more detail in Remark \ref{rem:tensor-not-finite} below.

\begin{remark}
A tensor product theory for logarithmic modules was developed in \cite{Huang:2010} and announced in \cite{Tsuchiya:2012ru}. In both approaches, the tensor product is a representing object for spaces of intertwining operators (by definition in \cite[Def.\,4.15]{Huang:2010} or as a consequence in \cite[Thm.\,29]{Tsuchiya:2012ru}).
\end{remark}

By definition, the tensor product of $A,B \in \Cc$ is an object $A \ast B \in \Rep(\h)+\svect$, together with a family of natural isomorphisms (in $C$) $\{ \beta_{A,B}^C \}_{C \in \Cc}$, where
\be
  \beta_{A,B}^C : \Hom_{\Cc}(A \ast B,C) \xrightarrow{~\sim~} \Vc_{A,B}^C \ .
\ee
If they exist, representing objects are unique up to unique isomorphism. Since $\Vc_{A,B}^C$ is functorial in $A,B$, the tensor product `$\ast$' can be defined on morphisms as 
\be
 \raisebox{2em}{\xymatrix{
\Hom_\Cc(A' \ast B',C) \ar[r]^(.65){\beta^C_{A',B'}} 
\ar[d]_{(-) \circ (f \ast g)} & 
\Vc^C_{A',B'}
\ar[d]^{\Vc^{\id}_{f,g}} \\
\Hom_\Cc(A \ast B,C) \ar[r]^(.65){\beta^C_{A,B}} 
& 
\Vc^C_{A,B} 
}}
\ee
where $f : A\to A'$, $g: B \to B'$ and for all $C$. In other words, $f \ast g : A \ast B \to A' \ast B'$ is obtained by specialising to $C = A' \ast B'$ and acting on $\id_{A' \ast B'}$:
\be
  f \ast g = (\beta^{A' \ast B'}_{A,B})^{-1} \circ \Vc_{f,g}^{\id} \circ \beta_{A',B'}^{A' \ast B'}(\id_{A' \ast B'}) \ .
\ee  

The main technical tool in finding the representing objects is the map $G$ defined in \eqref{eq:G-map-def}. We will need:

\begin{theorem} \label{thm:G-bijective}
The map $G$ in \eqref{eq:G-map-def} is an isomorphism.
\end{theorem}

\begin{proof}
{\em Injectivity:}
Let $V \in  \Vc_{A,B}^C$ and set $f = G(V)$. We need to show that $f=0$ implies $V(x)=0$ for all $x$. Proposition \ref{prop:V-on-gnd} shows that if $f=0$, then the restriction of $V(x)$ to ground-states is zero. But using condition (iii) in Definition \ref{def:vertex-op}, any matrix element $\langle \hat \gamma , V(x) \, a \otimes \hat b \rangle$ for $a\in A$, $\hat b \in \Indtwutw(B)$, $\hat\gamma \in \Indtwutw(C)'$ can be reduced to a sum of polynomials in $x^{\frac12}$ multiplying matrix elements of $V(x)$ on ground states only. The latter are zero, and so $V(x) = 0$.

\medskip\noindent
{\em Surjectivity:} From the explicit construction of $V_f$ via normal ordered exponentials in \eqref{eq:Vf-untwist-def} and \eqref{eq:Vf-twisted-def} we can read off directly that $G(V_f) = f$, so that $G$ is surjective in cases a) and b). For cases c) and d) the proof of surjectivity is more tedious and has been moved to Appendix \ref{app:G-surjective-cd}. 
\end{proof}

Using this result, it is easy to give the tensor product in each of the cases a)--d):

\begin{theorem} \label{thm:*-tensor}
The tensor product $A \ast B$ can be chosen as follows ($F$ is the forgetful functor):
\be\label{eq:*-tensor}
A \ast B ~=~  
\left\{\rule{0pt}{2.8em}\right.
\hspace{-.5em}\raisebox{.7em}{
\begin{tabular}{cccll}
  & $A$ & $B$ & $A \ast B$ &
\\
\rm a) & $\Cc_0$ & $\Cc_0$ & $A \otimes_{\Rep(\h)} B$ & $\in~\Cc_0$
\\
\rm b) & $\Cc_0$ & $\Cc_1$ & $F(A) \otimes B$ & $\in~\Cc_1$
\\
\rm c) & $\Cc_1$ & $\Cc_0$ & $A \otimes F(B)$ & $\in~\Cc_1$
\\
\rm d) & $\Cc_1$ & $\Cc_1$ & $U(\h) \otimes A \otimes B$ & $\in~\Rep(\h)$
\end{tabular}}
\ee
A possible choice for the corresponding natural isomorphisms $\beta_{A,B}^C$ is
\be\label{eq:*-tensor-beta}
\beta_{A,B}^C ~:~ f \mapsto 
\left\{\rule{0pt}{2.8em}\right.
\hspace{-.5em}\raisebox{.7em}{
\begin{tabular}{cccl}
  & $A$ & $B$ & $\beta_{A,B}^C(f)$ 
\\[.2em]
\rm a) & $\Cc_0$ & $\Cc_0$ & $G^{-1}(f) \circ \exp\!\big(  \ln(4) \, \Omega^{(12)} \big)$
\\[.2em]
\rm b) & $\Cc_0$ & $\Cc_1$ & $G^{-1}(f) \circ \exp\!\big(\! - \ln(4) \, \Omega^{(11)} \big)$
\\[.2em]
\rm c) & $\Cc_1$ & $\Cc_0$ & $G^{-1}(f) \circ \exp\!\big(\! - \ln(4) \, \Omega^{(22)} \big)$
\\[.2em]
\rm d) & $\Cc_1$ & $\Cc_1$ & $G^{-1}(f \circ (1\otimes \id))$
\end{tabular}}
\ee
In case d), $1$ is the unit of $U(\h)$.
\end{theorem}

A quick comment before giving the proof: omitting all the factors involving the unnatural looking $\ln(4)$ gives another possible choice for the natural isomorphisms $\beta$. However, these factors are chosen with hindsight to make the associator as simple as possible.

\begin{proof}
It remains to show that $\beta_{A,B}^C$ is natural in $C$ in cases a)--d) and that it is an isomorphism in case d). Let us start with the latter statement. For $R \in \Rep(\h)$ and $X \in \svect$ we have the adjunction isomorphism $\Hom(X,F(R)) \cong \Hom_{\Rep(\h)}(U(\h) \otimes X, R)$ which sends the even linear map $g: X \to R$ to the $\h$-module map $a \otimes x \mapsto a.g(x)$. Its inverse sends $f : U(\h) \otimes X \to R$ to $x \mapsto f(1 \otimes x)$. The map $\beta$ is the composition of this inverse with $G^{-1}$.

In cases a)--c), naturality amounts to checking that for all $f: C \to C'$ the diagram
\be
 \raisebox{2em}{\xymatrix{
\Vc_{AB}^C  \ar[r]^(.5){\overline{\Indtwutw}(f) \circ(-)} 
\ar[d]_{G} &  \Vc^{C'}_{A,B}
\ar[d]^{G} \\
\Hom(A \otimes B,C) \ar[r]^(.5){f \circ (-)} 
& 
\Hom(A \otimes B,C')
}}
\label{eq:tens-thm-aux1}
\ee
commutes.  (In case a) the image of $G$ actually lies in $\Hom_{\Rep(\h)}$.) Substituting the definition of $G$ in \eqref{eq:G-map-def} shows that \eqref{eq:tens-thm-aux1} is implied by $f \circ \Pgs = \Pgs \circ\overline{\Indtwutw}(f)$.

For case d) we need to supplement this by another commuting diagram. Namely for all $A,B \in \svect$, $C,C' \in \Rep(\h)$ and $f: C \to C'$ an $\h$-intertwiner,
\be
 \raisebox{2em}{\xymatrix{
\Hom(A \otimes B,C) \ar[r]^(.5){f \circ (-)}  \ar[d]^{\sim}
& 
\Hom(A \otimes B,C') \ar[d]^{\sim}
\\
\Hom_{\Rep(\h)}(U(\h) \otimes A \otimes B,C) \ar[r]^(.5){f \circ (-)} 
& 
\Hom_{\Rep(\h)}(U(\h) \otimes A \otimes B,C')
}}
\label{eq:tens-thm-aux2}
\ee
commutes. This is immediate from the definition of the isomorphism given above and the $\h$-intertwining property of $f$.
\end{proof}

\begin{example} \label{ex:reps}
(i) Single free boson, $\h = \Cb^{1|0}$: Let $\Cb_p$ be the irreducible representation of $J$-eigenvalue $p$ as in Section \ref{sec:examples-VO}. Let $T = \Cb^{1|0}$ be the one-dimensional even super-vector space in $\Cc_1$. The tensor product given in Theorem \ref{thm:*-tensor} reads
\be
\mathrm{a)}~
\Cb_p \ast \Cb_q \cong \Cb_{p+q}
~,~~~
\mathrm{b)}~
\Cb_p \ast T \cong T
~,~~~
\mathrm{c)}~
T \ast \Cb_p \cong T
~,~~~
\mathrm{d)}~
T \ast T \cong U(\h) \ .
\ee
In case d), the space $\Hom_\Cc(T \ast T,\Cb_p) \cong \Hom_{\Rep(\h)}(U(\h),\Cb_p)$ is one-dimensional for all $p\in\Cb$, as required. However, the universal enveloping algebra $U(\h)$ is not locally finite and thus not an object of $\Cc_0$.

\smallskip\noindent
(ii) Single pair of symplectic fermions, $\h = \Cb^{0|2}$: Since $\h$ is purely odd, $U(\h)$ is finite dimensional and the tensor product closes on $\Cc$. In $\svect$, $U(\h)$ is the symmetric algebra of $\h$; if we forget the parity-grading (i.e.\ apply the forgetful functor to vector spaces) then $U(\h)$ becomes the exterior algebra of $\h$. Let $T = \Cb^{1|0} \in \Cc_1$ be one of the two simple objects in $\Cc_1$. Case d) again states that $T \ast T \cong U(\h)$. This corresponds to the statement that in the $W_{1,2}$-model (the even part of a pair of symplectic fermions) the fusion product of the irreducible representation of lowest $L_0$-weight $-\frac18$ with itself gives the projective cover of the vacuum representation \cite{Gaberdiel:1996np} (in the notation used there, $(\Vc_{-1/8} \otimes \Vc_{-1/8})_\mathrm{f} = \mathcal{R}_0$).
\end{example}

\begin{remark} \label{rem:tensor-not-finite}
(i)
As illustrated in the example above, the tensor product $X \ast Y$, for $X,Y \in \Cc_1$, lies in $\Cc_0$ (i.e.\ is locally finite) if and only if $\h$ is purely odd. In the latter case we are dealing with symplectic fermions and the corresponding vertex operator algebra is $C_2$-cofinite \cite{Abe:2005}. For general $\h$, the tensor product $\ast$ (together with the associator given in Section \ref{sec:assoc-compute}) turns $\Cc_0$ into a tensor category and $\Cc_1$ into a bimodule category over $\Cc_0$. 

\noindent
One can extend the definition of $\ast$ to all of $\Rep(\h) + \svect$ by the same expressions as in \eqref{eq:*-tensor}. However, this will cause problems with the associators as these involve exponentials of $\Omega$, see Section \ref{sec:assoc-compute} below.

\smallskip\noindent
(ii)
For $\h$ purely odd (the symplectic fermion case), the dimension of the spaces of vertex operators $\Vc_{A,B}^C$ for $A,B,C$ simple was derived in the vertex operator algebra setting in \cite{Abe:2011ab}. The $\Zb/2\Zb \times \Zb/2\Zb$ rule found in \cite[Thm.\,5.7]{Abe:2011ab} is recovered in the present setting as follows:
Denote by $\Pi$ the parity shift functor on $\svect$ and as above let  $T \in \Cc_1$ be the even simple object in $\svect$. The four simple objects in $\Cc$ (for $\h$ purely odd) are
\be \label{eq:sympferm-simple}
  \one \,,\, \Pi\one ~\in~ \Cc_0
  \quad  , \quad
  T\,,\,\Pi T ~\in~ \Cc_1 \ .
\ee
If we think of $\Pi\one$ and $T$ as generators of a $\Zb/2\Zb \times \Zb/2\Zb$, one can verify from Theorem \ref{thm:*-tensor} that for $A,B,C$ simple, $\dim\Hom(A \ast B,C)$ is $1$ if the three corresponding group elements add to 0 and it is 0 else. 
Note, however, that this does not mean that the tensor product of simple objects is simple. Indeed, $T \ast T \cong U(\h)$, which is simple only for $\h = \{0\}$.
\end{remark}

In the following we will denote the vertex operator corresponding to $f : A \ast B \to C$ by $V(f;x)$. More specifically, we set
\be \label{eq:V(f;x)-def}
V(f;-) = \beta_{A,B}^C(f) ~:~ \Rb_{>0} \times(A \otimes \Indtwutw(B)) \to \overline\Indtwutw(C) \ .
\ee
Note that $V(f;-)$ restricted to ground states looks slightly different from the expressions in Proposition \ref{prop:V-on-gnd} due to the $\ln(4)$-factors in $\beta_{A,B}^C$. The explicit expressions are:
\be\label{eq:V(f)-on-ground-states}
\Pgs \circ V(f;x)\big|_{A \otimes B} ~=~  
\left\{\rule{0pt}{3.6em}\right.
\hspace{-.5em}\raisebox{.7em}{
\begin{tabular}{cccl}
  & $A$ & $B$ & 
\\
\rm a) & $\Cc_0$ & $\Cc_0$ & 
$f \circ \exp\!\big((\ln(x)+\ln(4))\cdot \Omega^{(12)}\big)$ 
\\[.5em]
\rm b) & $\Cc_0$ & $\Cc_1$ & 
$f \circ \exp\!\big(- ( \tfrac12\ln(x)+\ln(4) ) \cdot \Omega^{(11)}\big)$ 
\\[.5em]
\rm c) & $\Cc_1$ & $\Cc_0$ & 
$f \circ \exp\!\big(-( \tfrac12\ln(x)+\ln(4) )\cdot \Omega^{(22)}\big)$ 
\\[.5em]
\rm d) & $\Cc_1$ & $\Cc_1$ &
$\exp\!\big(\tfrac12\ln(x)\,(-\tfrac{\sdim \h}4 + \Omega^{(11)})\big) \circ f \circ (1 \otimes \id_{A \otimes B})$ 
\end{tabular}}
\ee

\section{Three-point blocks and the braiding}\label{sec:3pt-braid}

We have now gathered enough information about vertex operators and tensor product to extract the braiding and associativity isomorphisms. We will do this as usual \cite{Moore:1989vd} by investigating the behaviour of three-point conformal blocks under analytic continuation (braiding) and that of four-point blocks under taking different limits of coinciding points (associator).

The formalism in Sections \ref{sec:mode+rep} and \ref{sec:VO+tensor} has not been developed far enough to apply the general approach of \cite{Huang:2010} or \cite{Tsuchiya:2012ru}, which would guarantee that we obtain a braided monoidal structure on the category of representations. One down-side of the present approach therefore is, that we have to check pentagon and hexagon explicitly; this is done in Section \ref{sec:braided-mon-cat}. On the plus side, the present approach allows one to arrive at an explicit answer with basic means and in not too many pages (though, it appears, still quite a few).

\medskip

Let $A,B,C \in \Cc$ and let $f$ be morphism from $A \ast B$ to $C$. By a conformal three-point block on the complex plane, restricted to ground states, I mean a function of the form
\be
  \nu(f;z,w)
  := \Pgs \circ V(f;z{-}w)\big|_{A \otimes B}
  ~ : ~ A \otimes B \to C \ ,
\ee
where
\be\label{eq:braiding-branchcut}
  (z,w) \in C_\mathrm{slit} := \big\{ (u,v) \in \Cb \times \Cb \,\big|\, u\,{-}\,v \notin i\, \Rb_{\le 0} \big\} \ .
\ee
That is, $z-w$ is non-zero and does not lie on the negative imaginary axis. The function $\nu$ is initially defined for $z-w$ real and positive, and is then extended by analytic continuation.
The restriction that $(z,w)$ has to lie in $C_\mathrm{slit}$ is to some extend arbitrary. It was made to, firstly, make the domain of $\nu(f;z,w)$ simply connected, and, secondly, to single out a path for the analytic continuation which exchanges $z$ and $w$. 

\medskip

To extract the braiding -- or rather a map $\tilde c$ that will lead to the braiding in Section \ref{sec:C-hexagon} -- we will implement the following procedure. Suppose for now that $A,B \in \Cc_0$ and consider a conformal block $\nu(f;z,w)$ for $f : A \ast B \to C$. Think of this as ``$A$ inserted at $z$ and $B$ inserted at $w$''. We take $\nu(f;z,w)$ to be initially defined for $z,w \in \Rb$ and $z>w$. In this situation, ``$A$ is inserted to the right of $B$''. We then analytically continue $\nu(f;z,w)$ to all of $C_\mathrm{slit}$. This allows one to evaluate $\nu(f;z,w)$ for $z,w \in \Rb$ with $z<w$, i.e.\ for ``$A$ inserted to the left of $B$''. The result can be compared to a conformal three point block $\nu(f',w,z)$ for $f' : B \ast A \to C$ and $w,z \in \Rb$, $w>z$. 

If we allow $A$ or $B$ to be objects in $\Cc_1$, the above procedure should be modified slightly. To explain how, we first need to look at the $\Zb/2\Zb$-automorphism $\varphi$ of the mode algebra $\hat\h$ (resp.\ $\hat\h_\mathrm{tw}$). 
The automorphism $\varphi$ acts as $\varphi(K) = K$ and $\varphi(a_m) = -a_m$ for all $a \in \h$ and $m \in \Zb$ (resp.\ $m \in \Zb + \tfrac12$). The pullback of representations along $\varphi$ gives a $\Zb/2\Zb$-action on $\Rep(\hat\h)$ and $\Rep(\hat\h_\mathrm{tw})$. 
Next we will transport the pullback functor to $\Cc$. Let $f$ be any automorphism of $\hat\h$ and $R \in \Rep(\h)$, $X \in \svect$. Then we have natural isomorphisms
\be
  f^* \Ind(R) \xrightarrow{~~f^{-1} \otimes id_R~~} \Ind( f^*|_{\h} R ) \quad , \quad 
  f^* \Indtw(X) \xrightarrow{~~f^{-1} \otimes id_X~~} \Indtw( X ) \ .
\ee
In more detail, the first isomorphism is
\ba
f^* \Ind(R) 
&= 
{}_{f}U(\h) \otimes_{U(\h_{\ge 0} \oplus \Cb K)} R
= 
{}_{f}U(\h)_{f} \otimes_{U(\hat\h_{\ge 0} \oplus \Cb K)} {}_{f}R
\nonumber \\
&\xrightarrow{~~f^{-1} \otimes \id_R~~}
U(\h) \otimes_{U(\hat\h_{\ge 0} \oplus \Cb K)} {}_{f}R
=
\Ind(f^*|_{\h} R) \ .
\ea
The same holds for $\hat\h_\mathrm{tw}$ with the additional observation that $\hat\h_{\mathrm{tw},> 0}$ acts trivially on super-vector spaces, so that actually $f^* \circ \Indtw \cong \Indtw$. 
Write $G(\varphi)$ for the pair of functors $(\varphi^*|_{\h},\Id)$ on $\Cc_0 + \Cc_1$. The above argument shows that induction relates the pullback by the $\Zb/2\Zb$-automorphism $\varphi$ to the functor $G$ on $\Cc$. 

\medskip

For $A \in \Cc_j$ write $|A|=j$ and let $\varphi^0 = \id$, $\varphi^1 = \varphi$. We will abbreviate ${}^{|A|}\hspace{-1pt}B := G(\varphi^{|A|})(B)$. The modification of the above analytic continuation procedure is that $f'$ should be taken to be a map from ${}^{|A|}\hspace{-1pt}B \ast A$ to $C$. In the present setting, this may seem a moot point since one easily checks from \eqref{eq:*-tensor} that in all four instances ${}^{|A|}\hspace{-1pt}B \ast A$ is equal to $B \ast A$ (and not just isomorphic). The following remark gives a reason to include $G$ anyway.

\begin{figure}[bt]
\begin{center}
\ipic{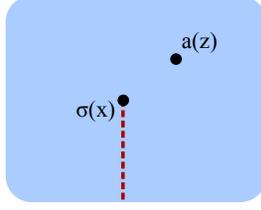}
\end{center}
\caption{Convention for the position of the branch cut for currents $a(z)$ in the presence of a single twist field insertion $\sigma(x)$.}
\label{fig:branch-cut-twisted}
\end{figure}

\begin{figure}
\begin{center}
\begin{tabular}{rrr}
\raisebox{2.4em}{a)}~
\ipic{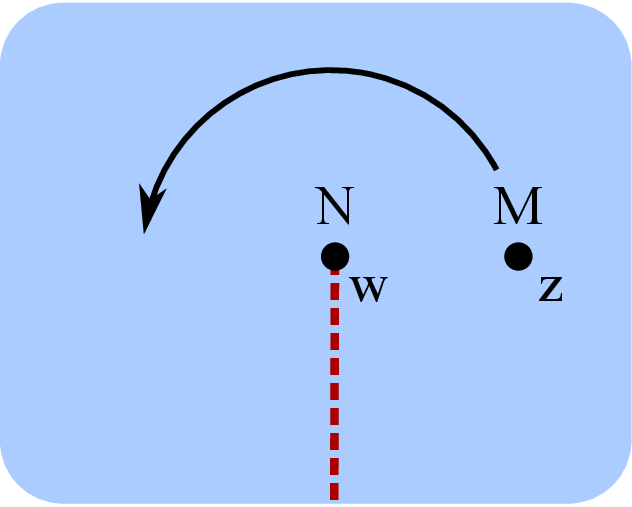} 
& 
\hspace{2em}
\raisebox{2.4em}{b)}~
\ipic{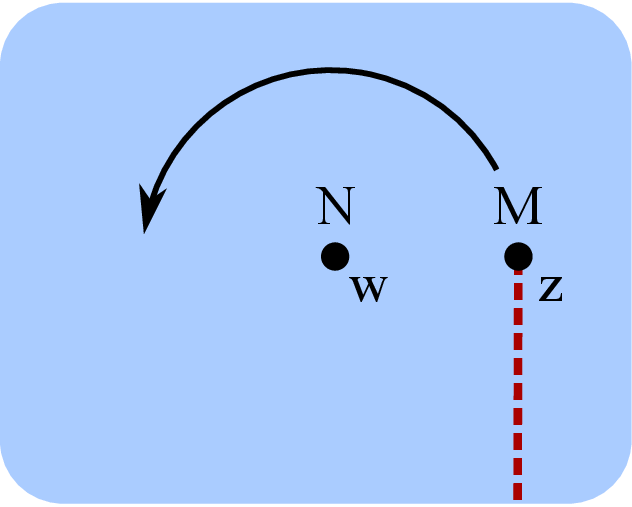} 
&
\hspace{2em}
\raisebox{2.4em}{c)}~
\ipic{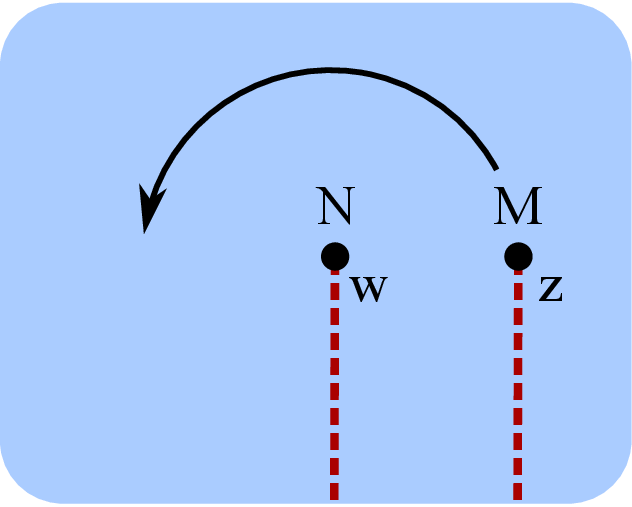} 
\\
\\[-.5em]
$\bm{\downarrow}$\hspace{3.7em}
&
$\bm{\downarrow}$\hspace{3.7em}
&
$\bm{\downarrow}$\hspace{3.7em}
\\[.5em]
\ipic{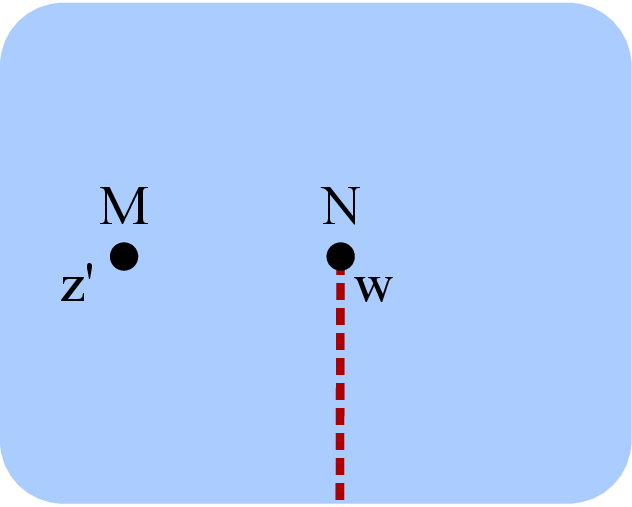} 
& 
\ipic{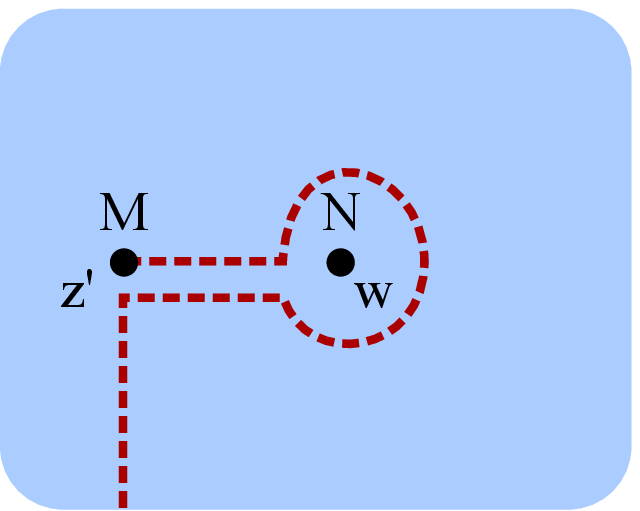} 
&
\ipic{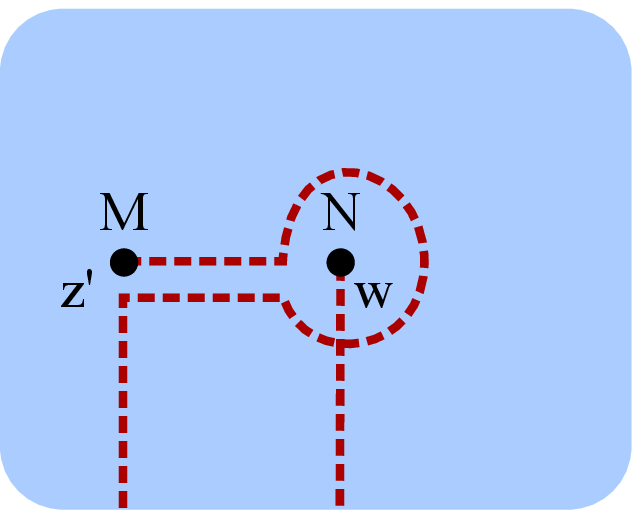} 
\end{tabular}
\end{center}
\caption{Analytic continuation of insertions in a conformal block for the three cases involving twisted representations. In cases b) and c) the action of the currents on the insertion of $N$ at point $w$ is twisted by the $\Zb/2\Zb$-automorphism $\varphi$.}
\label{fig:braid-branch}
\end{figure}

\begin{remark}
Suppose we were to consider conformal blocks which are not restricted to ground states and which may have current insertions. For the currents we need to fix a convention for the branch cuts close to a twist field insertion (i.e.\ an insertion labelled by a representation from $\Rep(\hat\h_\mathrm{tw})$). Let us agree that the branch cut is along the negative imaginary axis, shifted to the insertion point of the twist field, see Figure \ref{fig:branch-cut-twisted}. 
From this point of view, the $\Zb/2\Zb$-action appears in the braiding if during the analytic continuation of insertion points, one of the points moves through a branch cut. This is illustrated in Figure \ref{fig:braid-branch}. The combined process of analytic continuation of insertions in a conformal block and of rearranging the branch cuts turns the pair of representations $M, N$ labelling the insertions into $(\varphi^{|M|})^* N$ and $M$. 
\end{remark}

The outcome of the analytic continuation procedure is summarised in the following proposition (recall the map $\tau_{A,B} : A \otimes B \to B \otimes A$ from \eqref{eq:svect-tau}):

\begin{proposition} \label{prop:braiding-iso-from-3pt}
Let $A,B,C \in \Cc$. For each $f : A \ast B \to C$ there exists a unique $f' : {}^{|A|}\hspace{-1pt}B \ast A \to C$ such that, for all $x,y \in \Rb$ with $x>y$,
\be
  \nu(f;y,x) = \nu(f';x,y) \circ \tau_{A,B} 
  ~ : ~ A \otimes B \to C \ .
\ee
Furthermore, $f$ and $f'$ are related by $f = f'  \circ \tilde c_{A,B}$, 
where $\tilde c_{A,B} : A \ast B \to {}^{|A|}\hspace{-1pt}B \ast A = B \ast A$ is the isomorphism
\be\label{eq:braiding-explicit}
\tilde c_{A,B} ~=~  
\left\{\rule{0pt}{2.8em}\right.
\hspace{-.5em}\raisebox{.7em}{
\begin{tabular}{cccl}
  & $A$ & $B$ & 
\\
\rm a) & $\Cc_0$ & $\Cc_0$ & $\tau_{A,B} \circ \exp\!\big(- i \pi \cdot \Omega^{(12)}\big) $
\\
\rm b) & $\Cc_0$ & $\Cc_1$ & $\tau_{A,B} \circ \exp\!\big( \tfrac{i \pi}{2} \cdot \Omega^{(11)}\big)$
\\
\rm c) & $\Cc_1$ & $\Cc_0$ & $\tau_{A,B} \circ \exp\!\big( \tfrac{i \pi}{2} \cdot \Omega^{(22)}\big)$
\\
\rm d) & $\Cc_1$ & $\Cc_1$ & $\exp\!\big(\tfrac{i \pi \, \sdim(\h)}8 \big) \, \cdot \, 
  \big(\id_{U(\h)} \otimes \tau_{A,B}\big) \circ \exp\!\big( -\tfrac{i \pi}{2} \cdot \Omega^{(11)}\big)$ .
\end{tabular}}
\ee
Case d) requires $\h$ to be purely odd.
\end{proposition}

\begin{proof}
If $f'$ exists, it is unique by Theorem \ref{thm:G-bijective}. It remains to show that $f' = f \circ \tilde c_{A,B}^{-1}$ does the job. 

To compute $\nu(f;y,x)$ one can proceed as follows. By assumption, $x,y$ are real with $x>y$. Define $y' := e^{i \theta}(x-y)+x$. Setting $\theta=0$ gives $y'>x$ so that $\nu(f;y',x) = \Pgs \circ V(f;y'-x)|_{A \otimes B}$, and setting $\theta=\pi$ results in $y'=y$. Along the path $\theta \in [0,\pi]$ the pair $(y',x)$ lies in $C_\mathrm{slit}$, as required. Finally, substituting the explicit expressions from \eqref{eq:V(f)-on-ground-states} and taking $\theta$ from $0$ to $\pi$ in $V(f;y'-x) = V(f;e^{i \theta}(x-y))$ gives
\ba
& \nu(f;y,x) = \Big[ \Pgs \circ V(f;e^{i \theta}(x-y))\big|_{A \otimes B} \Big]_{\theta = \pi} 
\nonumber\\
&
=~  
\left\{\rule{0pt}{2.8em}\right.
\hspace{-.5em}\raisebox{.7em}{
\begin{tabular}{cccl}
  & $A$ & $B$ & 
\\
a) & $\Cc_0$ & $\Cc_0$ & 
$\displaystyle
f \circ e^{i \pi  \Omega^{(12)}} \circ e^{(\ln(x-y)+\ln(4))\cdot \Omega^{(12)}}$ ,
\\
b) & $\Cc_0$ & $\Cc_1$ & 
$\displaystyle
f \circ e^{-\frac12 i \pi \Omega^{(11)}} \circ e^{-(\frac12\ln(x-y)+\ln(4))\cdot \Omega^{(11)}}$ ,
\\
c) & $\Cc_1$ & $\Cc_0$ & 
$\displaystyle
f \circ e^{-\frac12 i \pi \Omega^{(22)}} \circ e^{-(\frac12\ln(x-y)+\ln(4))\cdot \Omega^{(22)}}$ ,
\\
d) & $\Cc_1$ & $\Cc_1$ & 
$\displaystyle
e^{-\frac18 i \pi \sdim \h}
e^{\frac12 i \pi \Omega^{(11)}} \circ e^{\frac12\ln(x-y)\,(-\frac{\sdim \h}4 + \Omega^{(11)})} \circ f \circ (1 \otimes \id)$ .
\end{tabular}}
\label{eq:braiding-proof-aux1}
\ea
We now have to check that in all cases,  \eqref{eq:braiding-proof-aux1} agrees with $\Pgs \circ V(f \circ \tilde c_{A,B}^{-1};x-y) \circ \tau_{A,B}$. For a)--c) this is immediate, for d) one uses in addition that $f$ is an $\h$-intertwiner.
\end{proof}

We will see in Section \ref{sec:C0-braidedmon} that the isomorphism $\tilde c_{A,B}$ defines a braiding on $\Cc_0$. If $\h$ is purely odd the tensor product closes on $\Cc$ and we can consider $\tilde c_{A,B}$ on all of $\Cc$. In this case, $\tilde c$ would give a ``braiding in a crossed $\Zb/2\Zb$-category'' in the sense of \cite{Turaev:2000ug}. To obtain a braiding on $\Cc$ we will precompose $\tilde c_{A,B}$ with an additional isomorphism formulated in terms of the parity involution. The further discussion of this has been postponed to Section \ref{sec:C-hexagon} since by then the associators will have been introduced and we can go on to verify the hexagon. 

\section{Four-point blocks and the associator (${}^\star$)} \label{sec:4pt-assoc}

The `$\star$' in the section heading means that in this section we will not state definitions and prove theorems, but instead carry out a series of conformal field theory calculations which will lead to the compositions of vertex operators listed in Table \ref{tab:4pt-blocks} and to the associativity isomorphisms collected in Section \ref{sec:C-pentagon}.
From the point of view of the mathematical setup in Sections \ref{sec:mode+rep} and \ref{sec:VO+tensor}, the results in Table \ref{tab:4pt-blocks} (at least most of them) are conjectures. The resulting associators are taken as input in Section \ref{sec:braided-mon-cat} and are proven there to define a monoidal structure on $\Cc$.

\subsection{Four-point blocks}\label{sec:4pt-block-overview}

Recall the definition of the vertex operator $V(f;x)$ in terms of a morphism $f : A \ast B \to C$ in \eqref{eq:V(f;x)-def}. Given representations $A,B,C \in \Cc$ we will also use the following more illustrative notation for vertex operators:
\be
  V(f;x) ~\equiv~ \bL{\blockA{C}{A}{f;x}\blockB{B}} \quad : ~~ A \otimes \Indtwutw(B) \to \overline\Indtwutw(C)  \ .
\ee
Given representations $A,B,C,D,P \in \Cc$ and morphisms $f : A \ast P \to D$, $g : B \ast C \to P$, consider the composition
\be
     \bL{\cbB DABCP{f;z}{\rule{0pt}{.68em}g;w}} 
  ~~ \equiv~~
  \Pgs \circ V(f;z) \circ (\id_A \otimes V(g;w))  \big|_{A \otimes B \otimes C}
  ~~:~~ A \otimes B \otimes C \to D 
\ee
of two vertex operators restricted to ground states.
Here $z>w>0$ and the restriction from $\Indtwutw(C)$ to $C$ as well as the final $\Pgs$ are omitted in the pictorial notation. 

\medskip

\begin{table}[p!]
\footnotesize
\begin{align*}
A\,B\,C \qquad &  \hspace{5em}    \bL{\cbB DABCP{f;1}{\rule{0pt}{.68em}g;x}} 
\nonumber\\
0\,~0\,~0\, \qquad & 
f \circ (\id_A \otimes g) \circ 
  \exp\!\Big( \ln(1{-}x) \cdot \Omega^{(12)} + \ln(x)\cdot \Omega^{(23)} 
\nonumber\\
 & \hspace{11em} + \ln(4) \cdot \big( \Omega^{(12)} + \Omega^{(13)} + \Omega^{(23)} \big) \Big) 
\\[.5em]
0\,~0\,~1\, \qquad & 
f \circ (\id_A \otimes g) \circ 
  \exp\!\Big(\!
  - \tfrac12 \ln(x)  \cdot \Omega^{(22)}
  +\ln\tfrac{1-\sqrt{x}}{1+\sqrt{x}} \cdot \Omega^{(12)} 
\nonumber\\
 & \hspace{11em} - \ln(4) \cdot \big( \Omega^{(11)} + \Omega^{(22)}  \big) \Big) 
\\[.5em]
0\,~1\,~0\, \qquad & 
f \circ (\id_A \otimes g) 
  \circ  \exp\!\Big(\!
  - \tfrac12 \ln(1{-}x) \cdot \Omega^{(11)}
  - \tfrac12 \ln(x)  \cdot \Omega^{(33)}
  +\ln\!\Big(\tfrac{\sqrt{1-x}-i\sqrt{x}}{\sqrt{1-x}+i\sqrt{x}}\Big) \cdot \Omega^{(13)} 
\nonumber\\
 & \hspace{11em} - \ln(4) \cdot \big( \Omega^{(11)} + \Omega^{(33)}  \big) \Big) 
\\[.5em]
1\,~0\,~0\, \qquad & 
f \circ (\id_A \otimes g) \circ  \exp\!\Big(\!
  - \tfrac12 \ln(1{-}x) \cdot \Omega^{(22)}
  +\ln\!\Big(4\,\tfrac{1-\sqrt{1-x}}{1+\sqrt{1-x}}\,\Big) \cdot \Omega^{(23)} 
\nonumber\\
 & \hspace{11em} - \ln(4) \cdot \big( \Omega^{(22)} + \Omega^{(23)} + \Omega^{(33)} \big) \Big) 
\\[.5em]
0\,~1\,~1\, \qquad & x^{- \frac{\sdim(\h)}8 } \cdot 
f \circ \exp\!\Big( 
\ln(x) \cdot \big( \tfrac12 \Omega^{(11)} +  \Omega^{(12)} + \tfrac12 \Omega^{(22)} \big) 
+  \ln \tfrac{x}{1-x} \cdot \tfrac12 \Omega^{(11)}
\nonumber\\
& \hspace{8em} 
-\ln\!\Big(4\,\tfrac{1-\sqrt{1-x}}{1+\sqrt{1-x}}\,\Big) \cdot \big( \Omega^{(11)} + \Omega^{(12)} \big) 
\nonumber\\
& \hspace{11em} 
  + \ln(4) \cdot \Omega^{(12)} \Big)  
\circ \big\{ \id_A \otimes \big( g \circ (1 \otimes \id_{B \otimes C})\big) \big\}
\\[.5em]
1\,~0\,~1\, \qquad & (\id_D \otimes \ev_B) \circ
\exp\!\Big( 
\ln \tfrac{\sqrt{1{-}x}-i\sqrt{x}}{\sqrt{1{-}x}+i\sqrt{x}} \cdot \Omega^{(13)}
- \ln(x(1{-}x)) \cdot \tfrac12 \Omega^{(33)} 
\nonumber\\
& \hspace{11em} 
 - \ln(4) \cdot \Omega^{(33)} \Big)
\circ \big\{  (f \circ (1 \otimes \id_A \otimes g))  \otimes \id_{B^* \otimes B} \big\} 
\nonumber\\
& \hspace{12em}  
\circ 
\big\{ \id_A \otimes \big[ (\tau_{C,B} \otimes \id_{B^* \otimes B}) \circ (\id_C \otimes \coev_B \otimes \id_B)
\circ \tau_{B,C} \big] \big\}
\\[.5em]
1\,~1\,~0\, \qquad & (1{-}x)^{- \frac{\sdim(\h)}8} \cdot (\id_D \otimes \ev_C) \circ 
\exp\!\Big(  \ln(1{-}x) \cdot \tfrac12 \Omega^{(11)} +  \ln \tfrac{1-x}x \cdot \tfrac12\Omega^{(33)} 
- \ln \tfrac{1-\sqrt{x}}{1+\sqrt{x}} \cdot \Omega^{(13)} \Big)
\nonumber\\
& \hspace{11em}  
- \ln(4) \cdot \Omega^{(33)} \Big)
\circ \big\{  (f \circ (1 \otimes \id_A \otimes g))  \otimes \id_{C^* \otimes C} \big\}
\nonumber\\
& \hspace{20em}  
\circ (\id_A \otimes \id_B \otimes \coev_C \otimes \id_C)
\\[.5em]
1\,~1\,~1\, \qquad & 
\big( x(1{-}x) \big)^{- \frac{\sdim(\h)}8 }
\big( \tfrac{2}{\pi} K(x) \big)^{- \frac{\sdim(\h)}2 }
\cdot  f
\circ 
\exp\!\Big(\! \big(\! - \tfrac\pi2 \, \tfrac{K(1{-}x)}{K(x)} + \ln(4)\big)  \Omega^{(22)} 
\nonumber\\
& \hspace{11em} 
  - \ln(4) \cdot \Omega^{(22)} \Big)  
\circ \big(\id_A \otimes g\big) \,
\end{align*}

\caption{
Composition of two vertex operators, restricted to ground states, for all eight combinations of choosing $A,B,C$ from $\Cc_0$ or $\Cc_1$. 
In cases 101 and 110 the copairing $\Omega$ acts on source and target of $f \circ (1 \otimes \id_A \otimes g)$ simultaneously; this is implemented using the (co)evaluation maps.
In case 111, 
$K(x) = \int_0^1 \big[ (1-t^2)(1-x t^2) \big]^{-\frac12} dt$  
is the complete elliptic integral of the first kind.
The contribution of the additional $\ln(4)$-factors in \eqref{eq:V(f)-on-ground-states} arising from the choice for the natural isomorphism made in Theorem \ref{thm:*-tensor} is written out separately in a new line in each case. This is the reason not to carry out the obvious cancellation in case 111.
}
\label{tab:4pt-blocks}
\end{table}

Table \ref{tab:4pt-blocks} lists these compositions specialised to $z=1$ and $w=x \in (0,1)$ for the eight possible ways of choosing $A,B,C$ from $\Cc_0$ or $\Cc_1$. Subsequently, the $x \to 1$ limit of these expressions will be used to determine the associativity isomorphism on $\Cc$. 

I do only have proofs for cases 000 and 001 in Table \ref{tab:4pt-blocks}. The remaining six cases are based on the usual contour integral arguments in calculating with conformal blocks and have to be treated as conjectures in the present setup. In any case, the calculations for all eight cases are given in Appendix \ref{app:4pt-calc}. 

\subsection{Definition of associators via asymptotic behaviour}\label{sec:assoc-asymp}

The associator of $\Cc$ is a family of isomorphism
\be
  \alpha_{A,B,C} : A \ast (B \ast C) \to (A \ast B) \ast C \ ,
\ee
natural in $A,B,C$. Recall from Remark \ref{rem:tensor-not-finite} that the tensor product of two objects in $\Cc$ may not lie in $\Cc$.  For the associator to have well-defined source and target we require $A \ast B$ and $B \ast C$ to lie in $\Cc$. If $\h$ is purely odd this poses no restriction, otherwise we are not allowed to choose both $A$ and $B$ from $\Cc_1$ nor to choose both $B$ and $C$ from $\Cc_1$. 

The associator is determined by the following pictorial condition, which will be explained below:
\be
     \bL{\cbB{\hspace{-2.2em}\text{\footnotesize$(A {\ast} B) {\ast} C$}}{\text{\footnotesize$A$}}{\text{\footnotesize$B$}}{\text{\footnotesize$C$}}{\text{\footnotesize$B {\ast} C$}}{\hspace{-.7em}\alpha;1}{\hspace{+.7em}\id;1{-}\eps}} \hspace{1.2em} \sim \hspace{2em}
     \bL{\cbC{\hspace{-2.2em}\text{\footnotesize$(A {\ast} B) {\ast} C$}}{\text{\footnotesize$A$}}{\text{\footnotesize$B$}}{\text{\footnotesize$C$}}{\text{\footnotesize$A{\ast}B$}}{\id;1{-}\eps}{\id;\eps}} \qquad .
\ee
As an equation, this pictorial condition stands for $D(\eps) = C(\eps) + \text{(lower order in $\eps$)}$, where
\ba
  D(\eps)
  &= \Pgs \circ V(\alpha;1) \circ \big(\id_A \otimes V(\id_{B \ast C};1{-}\eps)\big) \ ,
  \nonumber \\
  C(\eps) &= 
  \Pgs \circ V(\id_{(A \ast B) \ast C};1{-}\eps) \circ \big((\Pgs \circ V(\id_{A \ast B};\eps))|_{A \otimes B} \otimes \id_C\big) \ .
  \label{eq:assoc-asymp-def}
\ea
Here, $D(\eps)$ (the ``direct channel'') is of the form discussed in Section \ref{sec:4pt-block-overview}. It depends on a morphism $\alpha : A \ast (B \ast C) \to (A \ast B) \ast C$ which is to be determined. In $C(\eps)$ (the ``crossed channel'') we first go from $A \otimes B$ to $A \ast B$ by restricting the vertex operator $V(\id_{A \ast B};\eps)$ to ground states. The result serves as input for the vertex operator $V(\id_{(A \ast B) \ast C};1{-}\eps)$ which, once restricted to ground states, is a map $(A \ast B) \otimes C \to (A \ast B) \ast C$. 

The morphism $\alpha$ has to be fixed such that $C(\eps)$ describes the leading behaviour in $\eps$ of $D(\eps)$ for $\eps\to0$. This turns out to fix $\alpha$ uniquely except if $A,B,C$ are all taken from $\Cc_1$. This latter case is treated separately in Section \ref{sec:assoc111} below.

\begin{remark}
Using the appropriate exchange conditions for modes one can try to extend vertex operators to maps $V : \Ind(A) \otimes \Ind(B) \to \overline\Ind(C)$ instead of just allowing ground states from $A$. The asymptotic condition \eqref{eq:assoc-asymp-def} can then be strengthened to the identity
\be\label{eq:assoc-VO-id}
  V(\alpha;1) \circ \big(\id \otimes V(\id_{B \ast C};x)\big)
  = 
  V(\id_{(A \ast B) \ast C};x) \circ \big(V(\id_{A \ast B};1{-}x)) \otimes \id \big) 
\ee
for all $x \in (0,1)$. In vertex operator algebra language, this describes the associativity of intertwining operators, see \cite{Huang:2010}. In the present context, conjecturally all vertex operators can be extended and the maps $\alpha$ determined below by the asymptotic condition also satisfy \eqref{eq:assoc-VO-id}.
\end{remark}

\subsection{Computation of associators}\label{sec:assoc-compute}

We now turn to determining $\alpha$ in each case. 



\subsubsection*{Associator for $\Cc_0 \times \Cc_0 \times \Cc_0 \to \Cc_0$}

In this case, both $A \ast (B \ast C)$ and $(A \ast B) \ast C$ result in the tensor product $A \otimes B \otimes C$ in $\Rep(\h)$.
Inserting the expressions in \eqref{eq:V(f)-on-ground-states} into the definition of $C(\eps)$ from \eqref{eq:assoc-asymp-def} gives the endomorphism
\ba
  C(\eps) ~&=~ \exp\!\big( (\ln(1{-}\eps)+\ln(4)) \cdot (\Omega^{(13)}+\Omega^{(23)}) + (\ln(\eps)+\ln(4)) \cdot \Omega^{(12)} \big)
  \nonumber\\
  &\sim~ \exp\!\big( \ln(\eps) \cdot\Omega^{(12)} + \ln(4) \cdot \big( \Omega^{(12)} + \Omega^{(13)} + \Omega^{(23)} \big) \big) 
\ea
of $A \otimes B \otimes C$. In the direct channel the map $\alpha$ appears, which in the present case is an $\h$-endomorphism of $A \otimes B \otimes C$. Case 000 in Table \ref{tab:4pt-blocks} has the asymptotic behaviour
\ba
  D(\eps) ~&=~ \alpha \circ 
  \exp\!\big( \ln(\eps) \cdot \Omega^{(12)} + \ln(1{-}\eps)\cdot \Omega^{(23)} 
  + \ln(4) \cdot \big( \Omega^{(12)} + \Omega^{(13)} + \Omega^{(23)} \big) \big) 
  \nonumber\\
  &\sim~ \alpha \circ \exp\!\big( \ln(\eps) \cdot\Omega^{(12)} + \ln(4) \cdot \big( \Omega^{(12)} + \Omega^{(13)} + \Omega^{(23)} \big) \big) \ .
\ea
For the asymptotics of $D(\eps)$ and $C(\eps)$ to agree we must take
\be \label{eq:000-assoc}
  \alpha = \id_{A \otimes B \otimes C} \ .
\ee
As an aside, this particular associator would still just be the identity if we had not included the $\ln(4)$-factors in \eqref{eq:*-tensor-beta}; it is the associators involving objects from $\Cc_1$ which become simpler this way.

\subsubsection*{Associator for $\Cc_0 \times \Cc_0 \times \Cc_1 \to \Cc_1$}

Both $A \ast (B \ast C)$ and $(A \ast B) \ast C$ result in the tensor product $A \otimes B \otimes C$ in $\svect$. For $C(\eps)$ one finds
\be
  C(\eps) \sim \exp\!\big( (\ln(\eps)+\ln(4)) \cdot \Omega^{(12)} - \ln(4) \cdot (\Omega^{(11)} + 2\Omega^{(12)} + \Omega^{(22)}) \big)
\ee
Here it was used that the action of $\Omega^{(11)}$ on $A \ast B = A \otimes_{\Rep(\h)} B$ is given by
$\Delta(\Omega^{(11)}) = \Omega^{(11)} + 2\Omega^{(12)} + \Omega^{(22)}$. The asymptotic behaviour of $D(\eps)$ reads
\be
  D(\eps) \sim \alpha \circ \exp\!\big( \ln \tfrac{\eps}4 \cdot \Omega^{(12)} - \ln(4) \cdot (\Omega^{(11)}+\Omega^{(22)}) \big)
\ee
This shows that 
\be
  \alpha = \id_{A \otimes B \otimes C} \ .
\ee
Without the $\ln(4)$-factors in \eqref{eq:*-tensor-beta}, the condition on $\alpha$ would have read
$\exp( \ln(\eps) \cdot \Omega^{(12)} ) \sim \alpha \circ \exp( \ln \tfrac{\eps}4 \cdot \Omega^{(12)} )$, forcing $\alpha$ to be $e^{\ln(4) \cdot \Omega^{(12)}}$ instead of just the identity.

\subsubsection*{Associator for $\Cc_0 \times \Cc_1 \times \Cc_0 \to \Cc_1$}

Here again $\alpha$ is an endomorphism of $A \otimes B \otimes C$. For $C(\eps)$ one obtains
\be
C(\eps) \sim  \exp\!\big( - \tfrac12 \ln(\eps)  \cdot \Omega^{(11)} - \ln(4) \cdot (\Omega^{(11)} + \Omega^{(33)}) \big) \ .
\ee
For $D(\eps)$ one has to be careful to choose the correct branch of the logarithm in case 010 of Table \ref{tab:4pt-blocks}. Namely,
\be \label{eq:log-case-010-branch}
  l(x) = \ln\!\Big(\tfrac{\sqrt{1-x}-i\sqrt{x}}{\sqrt{1-x}+i\sqrt{x}}\Big)
  ~~,\quad
  0 = l(0) \xrightarrow{\text{ anal. cnt. $x$ from $0$ to $1$ }} l(1) = - \pi i \ .
\ee
Using this, one finds
\be
D(\eps) \sim \alpha \circ \exp\!\big( - \tfrac12 \ln(\eps) \cdot \Omega^{(11)} - i \pi \cdot \Omega^{(13)} - \ln(4) \cdot (\Omega^{(11)} + \Omega^{(33)})\big) \ .
\ee
The result is
\be
  \alpha = \exp\!\big( i \pi \cdot \Omega^{(13)} \big) \ .
\ee

\subsubsection*{Associator for $\Cc_1 \times \Cc_0 \times \Cc_0 \to \Cc_1$}

We have
\ba
C(\eps) &\sim \exp\!\big( - \tfrac12 \ln(\eps)  \cdot \Omega^{(22)} - \ln(4) \cdot (\Omega^{(22)} + \Omega^{(33)}) \big) \ ,
\nonumber\\
D(\eps) &\sim \alpha \circ  \exp\!\big( - \tfrac12 \ln(\eps)  \cdot \Omega^{(22)} + \ln(4) \cdot \Omega^{(23)}  - \ln(4) \cdot (\Omega^{(22)} + \Omega^{(23)} + \Omega^{(33)}) \big) \ ,
\ea
so that
\be
  \alpha = \id_{A \otimes B \otimes C} \ .
\ee

\subsubsection*{Intermezzo on the Hopf algebra structure of $U(\h)$}

To treat the cases with two objects chosen from $\Cc_1$ we need to use the fact that $U(\h)$ is a Hopf algebra.
Since $\h$ is abelian, $U(\h)$ is equal to the symmetric algebra of $\h$ in $\svect$. In particular, $\h$ is graded by the number of generators. The coproduct is, for all $a \in \h$,
\be
  \Delta(a) = a \otimes 1 + 1 \otimes a \ .
\ee
For all other elements of $U(\h)$ the coproduct is determined by the fact that it is an algebra map. Some care must be taken with the parity signs:
\be
  \Delta(ab) = \Delta(a)\Delta(b) 
  = (ab) \otimes 1 + a \otimes b + (-1)^{|a||b|} b \otimes a + 1 \otimes (ab) \ .
\ee
The antipode is fixed by declaring that $S(a) = -a$ for all $a \in \h$ and extending it as an algebra anti-automorphism. This results in $(-1)^{\text{grade}}$, independent of parity. For example,
\be
   S(ab) = (-1)^{|a||b|} S(b)S(a) = (-1)^{|a||b|} \, ba = ab
\ee
for all $a,b \in \h$. The counit $\eps$ satisfies $\eps(1) = 1$ and vanishes outside of grade $0$.

The Hopf algebra $U(\h)$ is commutative and cocommutative in $\svect$.

\subsubsection*{Associator for $\Cc_0 \times \Cc_1 \times \Cc_1 \to \Cc_0$}

Evaluating source and target of the associator results in
\be
  A \ast (B \ast C) = \underline{A \otimes U(\h)} \otimes B \otimes C \quad , \qquad
  (A \ast B) \ast C = \underline{U(\h)} \otimes A \otimes B \otimes C \ .
\ee
The underline indicates which objects contribute to the $\h$-action. So in the first case, the $\h$ action is on $A$ and $U(\h)$ while in the second case $\h$ acts only on $U(\h)$ (since $A \ast B = F(A) \otimes B$, forgetting the $\h$-action on $A$).
The crossed channel $C(\eps)$ is a map from $A \otimes B \otimes C$ to $U(\h) \otimes A \otimes B \otimes C$. Explicitly, 
\be
  C(\eps) \sim 
\exp\!\big(- ( \tfrac12\ln(\eps)+\ln(4) ) \cdot \Omega^{(22)}\big) \circ (1 \otimes \id_{A \otimes B \otimes C}) \ .
\ee
The direct channel $D(\eps)$ is a map from $A \otimes B \otimes C$ to $A \otimes U(\h) \otimes B \otimes C$ and has the asymptotic behaviour
\be
  D(\eps) \sim \alpha \circ
\exp\!\big( - ( \tfrac12\ln(\eps)+\ln(4) ) \cdot \Omega^{(11)} \big) \circ (\id_A \otimes 1 \otimes \id_{B \otimes C}) \ .
\ee
Composing from the right with the endomorphism $\exp\!\big(( \tfrac12\ln(\eps)+\ln(4) ) \cdot \Omega^{(11)}\big)$ shows that
the asymptotic behaviour of $C(\eps)$ and $D(\eps)$ agrees if and only if
\be \label{eq:alpha-011-aux1}
  \alpha \circ (\id_A \otimes 1 \otimes \id_{B \otimes C}) 
  = 1 \otimes \id_A \otimes \id_{B \otimes C} \ .
\ee
Since $\alpha$ is an $\h$-module map, we also know that 
\be
  a^{(1)} \circ \alpha = \alpha \circ (\Delta(a))^{(12)} \qquad \text{for all}~~  a \in \h \ . 
\ee  
As before, in $a^{(1)}$, $a$ acts on the first tensor factor of $U(\h) \otimes A \otimes B \otimes C$, and $(\Delta(a))^{(12)}$ means that $\Delta(a) \in \h \otimes \h$ acts on the first two factors of $A \otimes U(\h) \otimes B \otimes C$. This should not be confused with Sweedler's notation $\Delta(h) = \sum_{(h)} h_{(1)} \otimes h_{(2)}$ for the coproduct. We need the following lemma about Hopf algebras in $\svect$:

\begin{lemma}\label{lem:Hopf-aux2}
Let $H$ be a Hopf algebra in $\svect$. Let $M,N$ be $H$-modules in $\svect$ and let $V \in \svect$. Suppose an even linear map $f : M \otimes H \otimes V \to N$ satisfies
$h^{(1)} \circ f = f \circ (\Delta(h))^{(12)}$ for all $h \in H$. Then
$$
  f(m \otimes h \otimes v)
  ~=~ \sum_{(h)} (-1)^{|m||h|+|h_{(1)}||h_{(2)}|} \, (h_{(2)})^{(1)} \circ f\big( \, S^{-1}(h_{(1)}).m \, \otimes 1 \otimes v \,\big) \ .
$$
\end{lemma}

\begin{proof}
Precomposing both sides of the identity $h^{(1)} \circ f(m \otimes k \otimes v) = f \circ (\Delta(h))^{(12)}(m \otimes k \otimes v)$ (thought of as a map $H \otimes M \otimes H \otimes V \to N$ with the first argument being $h$) with the map
\be
  h \otimes m \otimes k \otimes v ~\mapsto~
  \sum_{(h)} (-1)^{|h_{(1)}||h_{(2)}|} \, h_{(2)} \otimes S^{-1}(h_{(1)}).m \otimes k \otimes v
\ee
and specialising to $k=1$ results in 
\ba
 &\sum_{(h)} (-1)^{|h_{(1)}||h_{(2)}|} \, (h_{(2)})^{(1)} \circ f\big( \, S^{-1}(h_{(1)}).m \, \otimes 1 \otimes v \,\big) 
 \nonumber\\
 &=  \sum_{(h)} (-1)^{|h_{(1)}|(|h_{(2)}|+|h_{(3)}|) + |h_{(3)}|(|h_{(1)}|+|m|)} \, f\big( h_{(2)} S^{-1}(h_{(1)}).m \otimes h_{(3)} \otimes v \big)
 \nonumber\\
 &\overset{(*)}=  (-1)^{|m||h|} \, f\big( m \otimes h \otimes v \big) \ .
\ea
In (*) the property $\sum_{(h)} (-1)^{|h_{(1)}||h_{(2)}|} \, h_{(2)} S^{-1}(h_{(1)}) = \eps(h) \cdot 1$ of the antipode was used.
\end{proof}

Specialising this lemma to $H = U(\h)$, $M=A$, $N=U(\h) \otimes A \otimes B \otimes C$ and $V=B \otimes C$ shows
\ba
\alpha(a \otimes h \otimes b \otimes c)
&\overset{\text{Lem.\,\ref{lem:Hopf-aux2}}}= \sum_{(h)} (-1)^{|a||h|+|h_{(1)}||h_{(2)}|} \, (h_{(2)})^{(1)} \circ \alpha\big( \, S^{-1}(h_{(1)}).a \, \otimes 1 \otimes b \otimes c \big)
\nonumber\\
&~\overset{\eqref{eq:alpha-011-aux1}}= \sum_{(h)} (-1)^{|a||h|+|h_{(1)}||h_{(2)}|} \, h_{(2)} \otimes  S^{-1}(h_{(1)}).a  \otimes b \otimes c \ .
\ea
We can write this without parity signs by using the symmetric braiding $\tau$ of $\svect$. Denote by $\rho^A : H \otimes A \to A$ the action of $H$ on $A$. Then
\be
  \alpha = \Big[\big\{ \id_{U(\h)} \otimes (\rho^A \circ (S^{-1} \otimes \id_A))  \big\} \circ \big\{ (\tau_{U(\h),U(\h)} \circ \Delta) \otimes \id_A \big\} \circ \tau_{A,U(\h)} \Big] \otimes \id_{B \otimes C} \ .
\ee
Of course, for $U(\h)$ we have $S^{-1}= S$ and $\tau_{U(\h),U(\h)} \circ \Delta = \Delta$ and this could be used to simplify the above expression. However, the above form allows for a more direct comparison with the results of \cite{Davydov:2012xg}, where such associators are considered for not necessarily (co)commutative Hopf algebras.

\subsubsection*{Associator for $\Cc_1 \times \Cc_0 \times \Cc_1 \to \Cc_0$}

The source and target of the associator are
\be
  A \ast (B \ast C) = \underline{U(\h)} \otimes A \otimes B \otimes C \quad , \qquad
  (A \ast B) \ast C = \underline{U(\h)} \otimes A \otimes B \otimes C \ .
\ee
As in the previous case, the underline indicates which objects contribute to the $\h$-action. The asymptotic behaviour of the crossed and direct channel are
\ba
  C(\eps) &\sim 
\exp\!\big(- ( \tfrac12\ln(\eps)+\ln(4) ) \cdot \Omega^{(33)}\big) \circ (1 \otimes \id_{A \otimes B \otimes C}) \ ,
\nonumber\\
 D(\eps) &\sim
(\id_{U(\h) \otimes A \otimes B \otimes C} \otimes \ev_B) \circ
\exp\!\Big( - i \pi \cdot  \Omega^{(16)}
- \big( \tfrac12 \ln(\eps) + \ln(4) \big) \cdot  \Omega^{(66)} \Big)
\nonumber\\ & \hspace{4em}
\circ \big\{  (\alpha \circ (1 \otimes \id_{A \otimes B \otimes C}))  \otimes \id_{B^* \otimes B} \big\} 
\nonumber\\ & \hspace{4em}
\circ 
\big\{ \id_A \otimes \big[ (\tau_{C,B} \otimes \id_{B^* \otimes B}) \circ (\id_C \otimes \coev_B \otimes \id_B)
\circ \tau_{B,C} \big] \big\} \ .
\ea
The exponential in $D(\eps)$ is an endomorphism of $U(\h) \otimes A  \otimes B  \otimes C  \otimes B^*  \otimes B$, so that for example $\Omega^{(16)}$ acts on $U(\h)$ and $B$. The coefficient $- i \pi$ of $\Omega^{(16)}$ arises via \eqref{eq:log-case-010-branch}.
We can simplify the condition $C(\eps) \sim D(\eps)$ by cancelling the $\Omega^{(33)}$ contribution in $C(\eps)$ against the $\Omega^{(66)}$  contribution in $D(\eps)$, and by acting with $\exp(\pi i \cdot \Omega)$ on both sides such that the first leg acts on the target factor $U(\h)$ and the second leg on the source factor $B$. The result is $\alpha \circ (1 \otimes \id_{A \otimes B \otimes C}) = \exp\!\big( \pi i \cdot \Omega^{(13)}\big) \circ (1 \otimes \id_{A \otimes B \otimes C})$. Since $\alpha$ is an $\h$-intertwiner, this condition determines $\alpha$ uniquely to be
\be
  \alpha = \exp\!\big( \pi i \cdot \Omega^{(13)}\big) \ .
\ee

\subsubsection*{Associator for $\Cc_1 \times \Cc_1 \times \Cc_0 \to \Cc_0$}

Source and target are
\be
  A \ast (B \ast C) = \underline{U(\h)} \otimes A \otimes B \otimes C \quad , \qquad
  (A \ast B) \ast C = \underline{U(\h)} \otimes A \otimes B \otimes \underline{C\raisebox{-2pt}{\rule{0em}{.9em}}} \ ,
\ee
with underlines giving the $\h$-action.
The asymptotic behaviour in the crossed and direct channel is
\ba
  C(\eps) &\sim \eps^{- \frac{\sdim(\h)}8} \cdot 
\exp\!\big( \ln(\eps) \cdot \tfrac12 \Omega^{(11)} + \ln(4) \cdot \Omega^{(14)} \big) \circ (1 \otimes \id_{A \otimes B \otimes C}) \ ,
\nonumber\\
 D(\eps) &\sim 
\eps^{- \frac{\sdim(\h)}8} \cdot (\id_{U(\h) \otimes A \otimes B \otimes C} \otimes \ev_C) \circ 
\exp\!\Big(  \ln(\eps) \cdot \big( \tfrac12 \Omega^{(11)} + \Omega^{(14)} 
+ \tfrac12 \Omega^{(44)} + \tfrac12 \Omega^{(66)} \big)
 \nonumber\\ 
& \hspace{19em}  
- \ln\tfrac\eps4 \cdot \big( \Omega^{(16)} + \Omega^{(46)} \big)
- \ln(4) \cdot  \Omega^{(66)} \Big)
\nonumber\\
& \hspace{5em}  
\circ \big\{  (\alpha \circ (1 \otimes \id_{A \otimes B \otimes C}))  \otimes \id_{C^* \otimes C} \big\}
\circ (\id_A \otimes \id_B \otimes \coev_C \otimes \id_C) \ .
\ea
The exponential in $D(\eps)$ is an endomorphism of $U(\h) \otimes A  \otimes B  \otimes C  \otimes C^*  \otimes C$. 
To arrive at the expression for $D(\eps)$, fix the representation $D$ in case 110 of Table \ref{tab:4pt-blocks} to be $ (A \ast B) \ast C = U(\h) \otimes A \otimes B \otimes C$ with $\h$ acting on $U(\h)$ and $C$. Because of this, $\Omega^{(11)}$ in case 110 -- where the `$1$' refers to $D$ -- turns into $\Omega^{(11)} + 2 \Omega^{(14)} + \Omega^{(44)}$ when acting on $U(\h) \otimes A \otimes B \otimes C$. Similarly, $\Omega^{(13)}$ turns into $\Omega^{(16)} +  \Omega^{(46)}$.

To simplify the condition $C(\eps) \sim D(\eps)$, move all exponential factors to the left hand side. After a short calculation one finds that they all cancel, resulting simply in 
\be\label{eq:assoc-case110-aux1}
  \alpha \circ (1 \otimes \id_{A \otimes B \otimes C}) = 1 \otimes \id_{A \otimes B \otimes C} \ .
\ee
Since $\alpha$ is an $\h$-intertwiner, it satisfies $(\Delta(u))^{(14)} \circ \alpha = \alpha \circ u^{(1)}$ for all $h \in U(\h)$. Then
\ba
  \alpha(u \otimes a \otimes b \otimes c)
  &~\,=\,~ (\Delta(u))^{(14)} \circ \alpha(1 \otimes a \otimes b \otimes c)
  \nonumber \\
  &\overset{\eqref{eq:assoc-case110-aux1}}= \textstyle \sum_{(u)} (-1)^{|u_{(2)}|(|a|+|b|)} \cdot u_{(1)} \otimes a \otimes b \otimes u_{(2)}.c \ .
\ea
As in case 011 one can hide the parity signs by using the braiding of $\svect$:
\be
  \alpha = (\id_{U(\h) \otimes A \otimes B} \otimes \rho^C)
  \circ 
  (\id_{U(\h)} \otimes \tau_{U(\h),A \otimes B} \otimes \id_C)
  \circ
  (\Delta \otimes \id_{A \otimes B \otimes C})
  \ .
\ee

\subsection{Associator for $\Cc_1 \times \Cc_1 \times \Cc_1 \to \Cc_1$}\label{sec:assoc111}

This case is set apart from the previous seven since it turns out that the associator cannot be determined from the restriction of the four-point block to ground states. To stress this point, let us briefly go through the ground state calculation before we turn to the evaluation of four-point blocks on descendent states.

\medskip

Source and target of $\alpha$ are
\be
  A \ast (B \ast C) = A \otimes U(\h) \otimes B \otimes C \quad , \qquad
  (A \ast B) \ast C = U(\h) \otimes A \otimes B \otimes C \ ,
\ee
both of which lie in $\Cc_1$. The asymptotic behaviour in the crossed channel is
\be
  C(\eps) \sim \eps^{- \frac{\sdim(\h)}8} \cdot 
\exp\!\big( \tfrac12  \ln \tfrac{\eps}{16}  \cdot \Omega^{(11)} \big) \circ (1 \otimes \id_{A \otimes B \otimes C}) \ .
\ee
This has to be reproduced by the expression in the direct channel (case 111 in Table \ref{tab:4pt-blocks}), \ba
  D(\eps) ~=~ & \big( \eps(1{-}\eps) \big)^{- \frac{\sdim(\h)}8 } \big( \tfrac{2}{\pi}\, K(1{-}\eps) \big)^{- \frac{\sdim(\h)}2 }
\nonumber \\
&\cdot ~ \alpha
\circ 
\exp\!\big(\!  - \tfrac\pi2 \, \tfrac{K(\eps)}{K(1{-}\eps)} \cdot  \Omega^{(22)} \big) 
\circ (\id_A \otimes 1 \otimes \id_{B \otimes C}) \ .
\ea
Recall from the beginning of Section \ref{sec:assoc-asymp} that in the present case -- as for all associators where either $A,B$ or $B,C$ are from $\Cc_1$ -- we have to restrict $\h$ to be purely odd. Define
\be 
  N =  - \tfrac12 \sdim(\h) \in \Zb_{> 0}  \ ,
\ee
so that $(\Omega^{(11)})^{N+1}$ acts as zero on all $\h$-modules. The direct channel can be rewritten as
\ba
  D(\eps) &= \big( \eps(1{-}\eps) \big)^{\frac{N}4 } 
\sum_{k=0}^N  \frac{(-1)^k}{k!}  \big(\tfrac{2}{\pi}\big)^{N-k} K(\eps)^k K(1{-}\eps)^{N-k}  
\cdot \alpha \circ (\Omega^{(22)})^k \circ (\id_A \otimes 1 \otimes \id_{B \otimes C})
\nonumber \\
& \sim ~ \eps^{\frac{N}4 } 
\sum_{k=0}^N  \frac{(-1)^{N} \big(\tfrac{\pi}{2}\big)^{2k-N}}{k!}  \big(\tfrac12 \ln\tfrac{\eps}{16} \big)^{N-k}  
\cdot \alpha \circ (\Omega^{(22)})^k \circ (\id_A \otimes 1 \otimes \id_{B \otimes C}) \ ,
\label{eq:111-D-asymp}
\ea
where we used the leading order of the expansions
\be \label{eq:EllipticK-asymp}
  K(x) = \tfrac{\pi}2 \big( 1 + \tfrac{x}{4} + O(x^2) \big) \quad , \quad
  K(1-x) = - \tfrac12 \ln\tfrac{x}{16} \cdot \big( 1 + \tfrac14 x + O(x^2) \big) ~-~ \tfrac14 x + O(x^2) \ .
\ee
Comparing the crossed and direct channel gives $N$ conditions on $\alpha$,
\be
\alpha \circ (\Omega^{(22)})^k \circ (\id_A \otimes 1 \otimes \id_{B \otimes C})
=  \frac{ (-1)^{N}  \big(\tfrac{\pi}{2}\big)^{N-2k}  \,k!  }{(N{-}k)!} \cdot 
   (\Omega^{(11)})^{N-k} \circ (1 \otimes \id_{A \otimes B \otimes C})  \ .
\ee
These conditions only determine $\alpha$ on elements of the form $(\Omega^{(11)})^k 1 \in U(\h)$, that is, only on an $(N{+}1)$-dimensional subspace of the $2^{2N}$ dimensional $U(\h)$. In the cases 011, 101, 110, the map $\alpha$ was a morphism in $\Cc_0$ and the $\h$-intertwiner property was then enough to fix $\alpha$ uniquely from its action on $1 \in U(\h)$. In the present case, however, $\alpha$ is a morphism in $\Cc_1$ and we need a different argument to determine $\alpha$ completely.

\subsubsection*{Direct and crossed channel for descendent states}

Let $p = a^1_{1/2} \cdots a^n_{1/2}$ and $\overline p = a^1 \cdots a^n$ for some $a^1,\dots,a^n \in \h$. We will compute
\be
  D(p;\eps)
  = \Pgs \circ p \circ V(\alpha;1) \circ \big(\id_A \otimes V(\id_{B \ast C};1{-}\eps)\big)  \ ,
\ee
where $p$ acts on $\overline\Indtw(U(\h) \otimes A \otimes B \otimes C)$, 
and compare its $\eps \to 0$ behaviour to
\be
  C(p;\eps) = 
  \Pgs \circ p \circ V(\id_{(A \ast B) \ast C};1{-}\eps) \circ \big((\Pgs \circ V(\id_{A \ast B};\eps))|_{A \otimes B} \otimes \id_C\big) \ .
\ee
The expression for $C(p,\eps)$ is immediate from the commutation relation in case b) of Definition \ref{def:vertex-op} and the explicit form of the vertex operators restricted to ground states in \eqref{eq:V(f)-on-ground-states}:
\ba
C(p;\eps) &= 
  \Pgs \circ V(\id_{(A \ast B) \ast C};1{-}\eps) \circ \big((\overline p \circ \Pgs \circ V(\id_{A \ast B};\eps))|_{A \otimes B} \otimes \id_C\big)  
  \nonumber\\
  & \sim \eps^{- \frac{\sdim(\h)}8} \cdot 
\exp\!\big(  \tfrac12  \ln \tfrac{\eps}{16}  \cdot \Omega^{(11)} \big) \circ (\overline p \otimes \id_{A \otimes B \otimes C}) \ .
\label{eq:111-Cp-asymp}
\ea
The calculation for $D(p;\eps)$ is more involved. As an auxiliary quantity we introduce, for $q \in U(\h)$,
\be
  D(p,q;\eps)
  = \Pgs \circ p \circ V(\alpha;1) \circ q^{(2)} \circ \big(\id_A \otimes V(\id_{B \ast C};1{-}\eps)\big)  \ .
\ee
In Appendix \ref{app:4ptblock-111} the following relation is given (see \eqref{eq:case111desc-J-rec}), which allows one to compute $D(p,q;\eps)$ recursively:
\be \label{eq:D(p,q,eps)-recurs}
  \tfrac{i \pi}{2} \cdot D(p,a q;\eps)
  = 
  K(1{-}\eps) \cdot D(p a_{\frac12},q;\eps)
  +
   \big(\tfrac{\eps}2  K(1{-}\eps) -  E(1{-}\eps) \big) \cdot D(p a_{-\frac12},q;\eps)    \ .
\ee
Using the asymptotic behaviour of $K(1{-}\eps)$ in \eqref{eq:EllipticK-asymp} together with 
$E(1-x) = 1  + O( x \ln(x) )$
we see that in the leading behaviour of $D(p,q;\eps)$ satisfies
\be \label{eq:D(p,q,eps)-recurs-asymp}
  \tfrac{i \pi}{2} \cdot D(p,a q;\eps)
  ~\sim~  
  - \tfrac12 \ln\tfrac{\eps}{16}  \cdot D(p a_{\frac12},q;\eps) 
  - D(p a_{-\frac12},q;\eps) 
  \ .
\ee
This allows one to express $D(1,q;\eps)$ as a sum of terms $D(p;\eps)$. For the latter, the asymptotic behaviour has to match that of $C(p;\eps)$ and for the former, the asymptotic behaviour is given by
\be
  D(1,q;\eps) \sim  \eps^{\frac{N}4 } 
\sum_{k=0}^N  \frac{  (-1)^N \big(\tfrac{\pi}{2}\big)^{2k-N}}{k!}  \big( \tfrac12 \ln\tfrac{\eps}{16} \big)^{N-k}  
\cdot \alpha \circ (\Omega^{(22)})^k \circ (\id_A \otimes q \otimes \id_{B \otimes C}) 
  \ ,
  \label{eq:111-Dq-asymp} 
\ee
as follows from repeating the calculation leading to \eqref{eq:111-D-asymp}. Since $q \in U(\h)$ is arbitrary, \eqref{eq:111-Cp-asymp}, \eqref{eq:D(p,q,eps)-recurs-asymp} and \eqref{eq:111-Dq-asymp} determine $\alpha$ uniquely. It remains to find an $\alpha$ which solves these conditions. 

\medskip

We make the ansatz
\be \label{eq:111-alpha-ansatz}
  \alpha = (\phi \otimes \id_{A \otimes B \otimes C}) \circ (\tau_{A,U(\h)} \otimes \id_{B \otimes C}) \ ,
\ee
where $\phi$ is an even linear endomorphism of $U(\h)$. To define $\phi$ we introduce the linear form $\lambda : U(\h) \to \Cb$ which is non-vanishing only in the top degree (i.e.\ in degree $2N$) and satisfies
\be \label{eq:111-lambda-def}
  \lambda\big(  \big(\Omega^{(11)}\big)^N \big) ~=~  N! \, \big(\tfrac{-2}\pi\big)^N \ ,
\ee
where $\Omega^{(11)}$ is understood as an element of $U(\h)$. Since the degree $2N$-component of $U(\h)$ is one-dimensional, this fixes $\lambda$ uniquely. Let furthermore $\mu_{U(\h)} : U(\h) \otimes U(\h) \to U(\h)$ be the multiplication of elements. Then
\be \label{eq:111-phi-def}
  \phi ~=~ \big( \id_{U(\h)} \otimes ( \lambda \circ \mu_{U(\h)}) \big) \circ \big(\exp( - \pi i \Omega ) \otimes \id_{U(\h)}\big) \ ,
\ee
where $\exp(- \pi i \Omega ) \in U(\h) \otimes U(\h)$. Note that $\phi$ maps an element of degree $n$ in $U(\h)$ to an element of degree $2N{-}n$.

\medskip

To check that this ansatz for $\alpha$ and $\phi$ solves the condition $D(p;\eps) \sim C(p;\eps)$ we will work in a symplectic basis. That is, for the rest of this section we will assume that the basis $\{ \alpha^i \}$ of $\h$ fixed in Section \ref{sec:lie-super} satisfies $(\alpha^{2m-1},\alpha^{2m}) = 1 = - (\alpha^{2m},\alpha^{2m-1})$ for all $m$ and is zero else. In this case, the dual basis satisfies $\beta^1 = \alpha^2$, $\beta^2 = - \alpha^1$, etc. The copairing reads
\be
  \Omega = \sum_{m=1}^N \Big( \alpha^{2m} \otimes \alpha^{2m-1} -  \alpha^{2m-1} \otimes \alpha^{2m} \Big) \ .
\ee
Then $\Omega^{(11)} = -2 \sum_{m=1}^N \alpha^{2m-1} \alpha^{2m}$ and $\big(\Omega^{(11)}\big)^N = (-2)^N N! \cdot \alpha^1 \alpha^2 \cdots \alpha^d$. Comparing this to \eqref{eq:111-lambda-def} shows that the one-form $\lambda$ obeys
\be \label{eq:111-lambda-sympbas}
  \lambda\big( \alpha^1 \alpha^2 \cdots \alpha^d \big) = \pi^{-N} \ .
\ee

The condition $D(p;\eps) \sim C(p;\eps)$ can now be factorised into $N$ two-dimensional problems. Namely, a monomial in $U(\h)$ can be written as $x_1 \cdots x_N$, where the individual factors $x_m$ are taken from $\{1 , \alpha^{2m-1}, \alpha^{2m} , \alpha^{2m-1} \alpha^{2m} \}$. 

Let $p$ be such that $\overline p = x_1 \cdots x_N$. Abbreviate $L \equiv  \tfrac12  \ln \tfrac{\eps}{16}$. According to \eqref{eq:111-Cp-asymp}, each factor $x_m$ contributes to the singular behaviour as $C(p,\eps) \sim \eps^{\frac N4} \prod_{m=1}^N C_m \otimes \id_{A \otimes B \otimes C}$ with
\be
{\renewcommand{\arraystretch}{1.4}
\begin{array}{l|l|l|l|l}
  x_m 
  & 1
  &  \alpha^{2m-1}
  &   \alpha^{2m}
  &  \alpha^{2m-1} \alpha^{2m}
\\
\hline
  C_m
  &  1 - 2L \alpha^{2m-1} \alpha^{2m}
  &  \alpha^{2m-1}
  &   \alpha^{2m}
  &  \alpha^{2m-1} \alpha^{2m}
\end{array}
} \qquad .
\label{eq:C_m-table}
\ee
For $D$, consider the relation \eqref{eq:D(p,q,eps)-recurs-asymp} in the case that $p,q$ do not contain elements form the $m$'th factor:
\ba
D(p \alpha^{2m-1}_{\frac12} \,, q;\eps) &\sim  
- i \tfrac{\pi}{2L}  \cdot D(p ,\, \alpha^{2m-1} q;\eps) 
~ , ~~~
D(p \alpha^{2m}_{\frac12} \,,q;\eps) \sim
- i \tfrac{\pi}{2L}  \cdot D(p ,\, \alpha^{2m} q;\eps)  \ ,
\nonumber \\
D(p \alpha^{2m-1}_{\frac12} \alpha^{2m}_{\frac12} \,,q;\eps) &\sim
- \big(\tfrac{\pi}{2L}\big)^2 \cdot D(p \, ,\alpha^{2m-1} \alpha^{2m} q;\eps) - \tfrac{1}{2L} \cdot D(p,q;\eps) \ .
\label{eq:111-Dpq-mth-factor}
\ea
Next, rewrite \eqref{eq:111-Dq-asymp} as
\be
  D(1,q;\eps) \sim  \eps^{\frac{N}4 } \, \big(\tfrac{-2L}{\pi}\big)^{N}
\cdot \alpha \circ 
\exp\!\Big(\!  - \tfrac{\pi^2}{2L} \sum_{m=1}^N \alpha^{2m-1} \alpha^{2m} \Big) \circ (\id_A \otimes q \otimes \id_{B \otimes C}) 
\ee
Combining this with \eqref{eq:111-Dpq-mth-factor} gives $D(p,\eps) \sim \eps^{\frac N4} 
\cdot \alpha \circ (\id_A \otimes \prod_{m=1}^N D_m \otimes \id_{B \otimes C})$, where
\be
{\renewcommand{\arraystretch}{1.4}
\begin{array}{l|l|l|l|l}
  x_m 
  & 1
  &  \alpha^{2m-1}
  &   \alpha^{2m}
  &  \alpha^{2m-1} \alpha^{2m}
\\
\hline
  D_m
  & - \tfrac{2L}{\pi} \, 1 + \pi \, \alpha^{2m-1} \alpha^{2m}
  & i \,  \alpha^{2m-1} 
  & i \,  \alpha^{2m} 
  & \tfrac{1}{\pi} \, 1
  \end{array}
}\qquad .
\label{eq:D_m-table}
\ee
Comparing the tables \eqref{eq:C_m-table} and \eqref{eq:D_m-table} shows that $\phi(x_1 \cdots x_N) = \prod_{m=1}^N \phi_m(x_m)$ with
\be
{\renewcommand{\arraystretch}{1.4}
\begin{array}{l|l|l|l|l}
  x_m 
  & 1
  &  \alpha^{2m-1}
  &   \alpha^{2m}
  &  \alpha^{2m-1} \alpha^{2m}
\\
\hline
  \phi_m
  & \pi\,\alpha^{2m-1} \alpha^{2m}
  & -i\, \alpha^{2m-1}
  & -i\, \alpha^{2m}
  & \pi^{-1}\,1
  \end{array}
}\qquad .
\ee
This is indeed the same as \eqref{eq:111-phi-def}: each factor is given by (recall \eqref{eq:d=0|2-example-Omega})
\ba
\phi_m(x_m) 
~=~ & 1 \cdot \lambda_m(1 \cdot x_m) - \pi i \, \alpha^{2m} \cdot \lambda_m( \alpha^{2m-1} x_m)  
+  \pi i \, \alpha^{2m-1} \cdot  \lambda_m( \alpha^{2m} x_m) 
\nonumber \\
&+ \pi^2 \, \alpha^{2m-1} \alpha^{2m} \cdot \lambda_m( \alpha^{2m-1} \alpha^{2m} x_m )
\ea
where $\lambda_m(\alpha^{2m-1} \alpha^{2m}) = \pi^{-1}$. We have now completed the verification of the ansatz \eqref{eq:111-alpha-ansatz}.

\section{Braided monoidal category} \label{sec:braided-mon-cat}

In this section we will verify that the braiding and associativity isomorphism calculated in Sections \ref{sec:3pt-braid} and \ref{sec:4pt-assoc} do indeed define a braided monoidal category. We will start in Section \ref{sec:C0-braidedmon} by restricting to the sector of untwisted $\hat\h$-representations, $\Cc_0$. There, the verification of pentagon and hexagon is immediate and it is not necessary to assume that $\h$ is purely odd. 

On the other hand, only for $\h$ purely odd do the results of Sections \ref{sec:VO+tensor}--\ref{sec:4pt-assoc} give a well-defined tensor product, associator and braiding on all of $\Cc$. In that case, the pentagon and hexagon identity on all of $\Cc$ are verified in Sections \ref{sec:C-pentagon}--\ref{sec:C-hexagon} via the results in \cite{Davydov:2012xg}.

\begin{figure}
\begin{tabular}{cc}
(P) &
$\raisebox{-8.8em}{
\xygraph{ !{0;/r4.5pc/:;/u4.5pc/::}[]*+{A\otimes_\Cc(B\otimes_\Cc(C\otimes_\Cc D))} (
  :[u(1.0)r(1.9)]*+{(A\otimes_\Cc B)\otimes_\Cc(C\otimes_\Cc D)} ^{\alpha_{A,B,C D}}
  :[d(1.0)r(1.9)]*+{((A\otimes_\Cc B)\otimes_\Cc C)\otimes_\Cc D}="r" ^{\alpha_{A B, C,D}}
  ,
  :[r(.7)d(1.0)]*+!R(.3){A\otimes_\Cc((B\otimes_\Cc C)\otimes_\Cc D)} _{id_A\otimes_\Cc \alpha_{B,C,D}}
  :[r(2.4)]*+!L(.3){(A\otimes_\Cc(B\otimes_\Cc C))\otimes_\Cc D} ^{\alpha_{A,B C,D}}
  : "r" _{\alpha_{A,B,C}\otimes_\Cc id_D}
)
}}
$
\\
\\[2em]
(H1) &
$
\xymatrix{& A\otimes_\Cc(B\otimes_\Cc C) \ar[dl]_{\alpha_{A,B,C}} \ar[rr]^{c_{A,B C}} && (B\otimes_\Cc C)\otimes_\Cc A \\
(A\otimes_\Cc B)\otimes_\Cc C \ar[rd]_{c_{A,B}\otimes_\Cc id_C~~~~} &&&& B\otimes_\Cc(C\otimes_\Cc A) \ar[ul]_{\alpha_{B,C,A}} \\
& (B\otimes_\Cc A)\otimes_\Cc C && B\otimes_\Cc(A\otimes_\Cc C) \ar[ll]^{\alpha_{B,A,C}}  \ar[ur]_{~~~~id_B\otimes_\Cc c_{A,C}} }
$
\\
\\[2em]
(H2) &
$
\xymatrix{& (A\otimes_\Cc B)\otimes_\Cc C\ar[rr]^{c_{A B,C}} && C\otimes_\Cc(A\otimes_\Cc B) \ar[dr]^{\alpha_{C,A,B}}  \\
A\otimes_\Cc(B\otimes_\Cc C) \ar[rd]_{id_A\otimes_\Cc c_{B,C}~~~} \ar[ru]^{\alpha_{A,B,C}} &&&& (C\otimes_\Cc A)\otimes_\Cc B \\
& A\otimes_\Cc(C\otimes_\Cc B) \ar[rr]^{\alpha_{A,C,B}} && (A\otimes_\Cc C)\otimes_\Cc B \ar[ur]_{~~c_{A,C}\otimes_\Cc id_B} }
$
\end{tabular}

\caption{(P): The pentagon condition for the associator of a monoidal category $\Cc$. (H1,H2): The hexagon conditions for the braiding isomorphisms in a braided monoidal category $\Cc$. In the indices of $\alpha$ and $c$ we have omitted the $\otimes_{\Cc}$ for better readability.} \label{fig:pent-hex}
\end{figure}

\subsection{$\Cc_0$ is a braided monoidal category}\label{sec:C0-braidedmon}

Recall that $\Cc_0 = \Rep\lf(\h)$ and that according to Theorem \ref{thm:*-tensor} the tensor product is given by $A \ast B = A \otimes_{\Rep(\h)} B$. The associativity isomorphism was found to be the identity in \eqref{eq:000-assoc}  (more precisely, the associator of the underlying category of super-vector spaces, but we will not write these out) and the braiding was given in Proposition \ref{prop:braiding-iso-from-3pt}:
\be
  \alpha_{A,B,C} = \id_{A \otimes B \otimes C} \quad , \qquad
  c_{A,B} = \tau_{A,B} \circ \exp\!\big(- i \pi \cdot \Omega^{(12)}\big) \ .
\ee

\begin{proposition}\label{prop:C0-monoidal}
$\Cc_0$ with $\ast$ as tensor product, $\alpha$ as associator and $c$ as braiding, is a braided monoidal category.
\end{proposition}

\begin{proof}
The pentagon and hexagon identities are listed in Figure \ref{fig:pent-hex}. The pentagon is trivial (i.e.\ it is that of $\svect$), and the hexagon (H1) reads $c_{A,B \otimes C} = (\id_B \otimes c_{A,C}) \circ (c_{A,B} \otimes \id_C)$. This identity follows since $(\id_{U(\h)} \otimes \Delta) \circ \Omega^{(12)} = \Omega^{(12)} + \Omega^{(13)}$ and since $\Delta$ is an algebra map. The hexagon (H2) is checked analogously.
\end{proof}

The category $\Cc_0$ can be understood as a version of Drinfeld's category for metric Lie algebras specialised to the abelian case \cite{Drinfeld:1990,Davydov:2010od}.

\subsection{Pentagon identity for $\Cc$}\label{sec:C-pentagon}

Assume from here on that $\h$ is purely odd. Recall that $\Cc$ is the $\Zb/2\Zb$-graded category with $\Cc_0 = \Rep\lf(\h)$ and $\Cc_1 = \svect$ with tensor product functor `$\ast$' given in Theorem \ref{thm:*-tensor}. (Of course, for $\h$ purely odd every $\h$-module is locally finite.) The associators $\alpha_{A,B,C}$ found in Section \ref{sec:assoc-compute} and \ref{sec:assoc111} were:
\begin{align*}
A\,B\,C \qquad & \alpha_{A,B,C} : A \ast (B \ast C) \to (A \ast B) \ast C
\nonumber\\[.3em]
0\,~0\,~0\, \qquad & \id_{A \otimes B \otimes C}
\nonumber\\[.3em]
0\,~0\,~1\, \qquad & \id_{A \otimes B \otimes C} 
\nonumber\\[.3em]
0\,~1\,~0\, \qquad & \exp\!\big( i \pi \cdot \Omega^{(13)} \big)
\nonumber\\[.3em]
1\,~0\,~0\, \qquad &  \id_{A \otimes B \otimes C}
\nonumber\\[.3em]
0\,~1\,~1\, \qquad & \Big[\big\{ \id_{U(\h)} \otimes (\rho^A \circ (S^{-1} \otimes \id_A))  \big\} \circ \big\{ (\tau_{U(\h),U(\h)} \circ \Delta) \otimes \id_A \big\} \circ \tau_{A,U(\h)} \Big] \otimes \id_{B \otimes C}
\nonumber\\[.3em]
1\,~0\,~1\, \qquad &  \exp\!\big( i \pi \cdot \Omega^{(13)} \big)
\nonumber\\[.3em]
1\,~1\,~0\, \qquad &  \big\{ \id_{U(\h) \otimes A \otimes B} \otimes \rho^C \big\}
  \circ 
  \big\{ \id_{U(\h)} \otimes \tau_{U(\h),A \otimes B} \otimes \id_C  \big\}
  \circ
  \big\{ \Delta \otimes \id_{A \otimes B \otimes C} \big\}
\nonumber\\[.3em]
1\,~1\,~1\, \qquad & \big\{ \phi \otimes \id_{A \otimes B \otimes C} \big\} \circ \big\{\tau_{A,U(\h)} \otimes \id_{B \otimes C}\big\}
\end{align*}
In case 010, $\alpha$ is an endomorphism of $A \otimes B \otimes C$ while in case 101, it is an endomorphism of $U(\h) \otimes A \otimes B \otimes C$. 
Let $N = - \sdim(\h)/2$.
The endomorphism $\phi$ in the last case was defined in \eqref{eq:111-lambda-def} and \eqref{eq:111-phi-def} to be  
\be \label{eq:111-phi-def-repeat}
  \phi ~=~ \big( \id_{U(\h)} \otimes ( \lambda \circ \mu_{U(\h)}) \big) \circ \big(\exp( - \pi i \Omega ) \otimes \id_{U(\h)}\big) \ ,
\ee
with $\lambda$ the top form on $U(\h)$ which satisfies $\lambda\big(  (i \pi \cdot \Omega^{(11)})^N \big) =  (-1)^N \, N! \, (2 i )^N$. The work for the following theorem was already done in \cite{Davydov:2012xg}.

\begin{theorem} \label{thm:C-associator}
$\Cc$ with $\ast$ as tensor product and $\alpha$ as associator is a monoidal category.
\end{theorem}

\begin{proof}
The associator is precisely of the form treated in \cite{Davydov:2012xg}. Namely, in \cite[Sect.\,3.8.2]{Davydov:2012xg} choose $C = i \pi \Omega$ and $\zeta = (-1)^N$. Then $\lambda$ agrees with the cointegral in \cite[Sect.\,3.8.2]{Davydov:2012xg} and the associativity isomorphisms above are precisely those in \cite[Fig.\,3]{Davydov:2012xg}. The pentagon holds due to \cite[Cor.\,3.17]{Davydov:2012xg}; that the conditions of that corollary are met is shown in \cite[Sect.\,3.8.2]{Davydov:2012xg}. 
\end{proof}

\begin{remark}
If $\h$ is not purely odd, the tensor products and associators with only one object taken from $\Cc_1$ are still all well-defined. The pentagon conditions which only involve such associators are satisfied, so that $\Cc_1$ becomes a bimodule category over $\Cc_0$.
\end{remark}

\subsection{Hexagon identity for $\Cc$}\label{sec:C-hexagon}

Let $\h$ be purely odd. 

Recall from Section \ref{sec:3pt-braid} that $G(\varphi)$ denotes the pair of functors $(\varphi^*|_{\h},\Id)$ on $\Cc_0 + \Cc_1$, and that we abbreviated ${}^{|A|}\hspace{-1pt}B = G(\varphi^{|A|})(B)$. In Proposition \ref{prop:braiding-iso-from-3pt} we defined the family of natural isomorphism $\tilde c_{A,B} : A \ast B \to {}^{|A|}\hspace{-1pt}B \ast A$ for $A,B \in \Cc$. As already mentioned in Section \ref{sec:3pt-braid}, we have ${}^{|A|}\hspace{-1pt}B \ast A = B \ast A$ for all $A,B \in \Cc$. Nonetheless, it would be misleading to write $\tilde c_{A,B} : A \ast B \to B \ast A$ because $\tilde c$ is not a braiding. For example, (H2) in Figure \ref{fig:pent-hex} fails for $A \in \Cc_1$ and $B,C \in \Cc_0$. To turn $\tilde c$ into a braiding, we will construct an isomorphism ${}^{|A|}\hspace{-1pt}B \to B$.

\medskip

Since $\varphi$ leaves $\Omega$ and $\lambda$ (from \eqref{eq:111-lambda-def}) invariant, via \cite[Prop.\,3.19]{Davydov:2012xg}, $G(\varphi)$ becomes a monoidal equivalence. The only non-trivial part of the monoidal structure on the functor $G(\varphi)$ is that for $X,Y \in \Cc_1$, $G(\varphi)_{X,Y} = \varphi \otimes \id_{X \otimes Y}$.

For a super-vector space $V$ denote by $\omega_V : V \to V$ the parity involution $\omega(v) = (-1)^{|v|} v$. The family of isomorphisms $\{ \omega_V \}_{V \in \svect}$ is a natural monoidal isomorphism of the identity functor on $\svect$.
Write $\omega^0_V = \id_V$ and $\omega^1_V = \omega_V$. Since $\h$ is purely odd, the action of $\varphi$ on $U(\h)$ is just the parity involution $\omega_{U(\h)}$. Given a morphism $f : R \to S$ in $\Cc_0$, the diagram
\be
\raisebox{2em}{\xymatrix{
  R \ar[rr]^f \ar[d]^{\omega_R} && S \ar[d]^{\omega_S}
  \\
  \varphi^*R \ar[rr]^{\varphi^*(f)} && \varphi^* S
}}
\ee
commutes. In other words, $\omega : \Id \Rightarrow \varphi^*$ is a natural isomorphism on $\Cc_0$. One verifies that this isomorphism is in addition monoidal. On all of $\Cc$, $G(\varphi)$ is naturally isomorphic to $\Id_{\Cc}$ via the pair $(\omega, \id)$; however, the natural isomorphism $(\omega, \id)$ is not monoidal (e.g.\ for $X = \Cb^{1|0}$ and $Y =  \Cb^{0|1}$, $G(\varphi)_{X,Y} \circ \omega_{X \ast Y} \neq \id_X \ast \id_Y$). Instead, we should take the pair $(\omega, \omega)$. Altogether,
\be
  \omega ~:~ \Id \, \Longrightarrow \, G(\varphi)
\ee
is a natural monoidal isomorphism of functors on $\Cc$. Combining this with $\tilde c$, we obtain our candidate for the braiding on $\Cc$:
\be
  A \ast B \xrightarrow{ ~~\tilde c_{A,B}~~ } {}^{|A|}\hspace{-1pt}B \ast A \xrightarrow{ ~~\omega_B^{|A|} \ast \id_A ~~ } B \ast A \ .
\ee
Using naturality of $\tilde c$ we can alternatively place $\omega$ before $\tilde c$ and define for $A,B \in \Cc$ and $N = - \sdim(\h)/2$:
\be \label{eq:braiding-def-h-odd}
\raisebox{.7em}{
\begin{tabular}{ccl}
 $A$ & $B$ & $c_{A,B} ~:~ A \ast B \to B \ast A$
\\[.3em]
 $\Cc_0$ & $\Cc_0$ & $ \tau_{A,B} \circ \exp\!\big(- i \pi \cdot \Omega^{(12)}\big) $
\\[.3em]
 $\Cc_0$ & $\Cc_1$ & $\tau_{A,B} \circ \exp\!\big( \tfrac{i \pi}{2} \cdot \Omega^{(11)}\big)$
\\[.3em]
 $\Cc_1$ & $\Cc_0$ & $\tau_{A,B} \circ \exp\!\big( \tfrac{i \pi}{2} \cdot \Omega^{(22)}\big) \circ \{ \id_A \otimes \omega_B \}$
\\[.3em]
 $\Cc_1$ & $\Cc_1$ & $e^{-i \pi  \frac{N}4} \cdot
  \big(\id_{U(\h)} \otimes \tau_{A,B}\big) \circ \exp\!\big(\! -\tfrac{i \pi}{2} \cdot \Omega^{(11)}\big)\circ \{ \id_{U(\h) \otimes A} \otimes \omega_B \}$ \ .
\end{tabular}}
\ee
To say that $c_{A,B}$ defines a braiding on $\Cc$ means that the hexagon identities (H1) and (H2) in Figure \ref{fig:pent-hex} hold. This is the content of the next theorem, which is also contained in \cite{Davydov:2012xg}.

\begin{theorem}\label{thm:C-braiding}
The natural family of isomorphisms $c_{A,B}$ defines a braiding on the monoidal category $(\Cc , \ast , \alpha)$ from Theorem \ref{thm:C-associator}.
\end{theorem}

\begin{proof}
The braiding defined in \eqref{eq:braiding-def-h-odd} is of the form \cite[Eqn.\,(4.3)]{Davydov:2012xg}; the defining morphisms are stated in \cite[Eqn.\,(4.108)]{Davydov:2012xg} with $C = i \pi \Omega$ as in the proof of Theorem \ref{thm:C-associator}, as well as $\beta = e^{-i \pi N/4}$. We have $\beta^2 = \zeta \cdot i^N$ where $\zeta = (-1)^N$ from Theorem \ref{thm:C-associator} as required by \cite[Eqn.\,(4.107)]{Davydov:2012xg}. By \cite[Thm.\,1.3\,\&\,Rem.\,4.11\,(ii)]{Davydov:2012xg}, $c$ defines a braiding on $\Cc$; that the conditions of the theorem are satisfied is checked in \cite[Sect.\,4.7.2]{Davydov:2012xg}.
\end{proof}

\begin{remark} \label{rem:C-braiding}
(i)
It is proved in \cite[Prop.\,4.20]{Davydov:2012xg} that the braiding is non-degenerate in the sense that $c_{B,A} \circ c_{A,B}$ for all $B \in \Cc$ implies that $A$ lies in $\Cc_0$, is purely even, and has trivial $\h$-action. If $A$ is finite-dimensional, it is thus a finite direct sum of copies of the trivial representation $\Cb^{1|0}$. 

\smallskip\noindent
(ii) The full subcategory $\Cc\fd$ of all objects in $\Cc_0$ and $\Cc_1$ which have finite-dimensional underlying super-vector spaces carries the structure of a ribbon category. That is, it has right duals and twist automorphisms \cite[Sect.\,4.6]{Davydov:2012xg}. For $R \in \Cc_0\fd$ and $X \in \Cc_1\fd$, the twist automorphisms are given by \cite[Prop.\,4.18]{Davydov:2012xg}
\be
  \theta_R = \exp\!\big(\!- \pi i \cdot \Omega^{(11)}\big) \quad , \qquad
  \theta_X = e^{i \pi N/4} \cdot \omega_X \ .
\ee
In $\Cc_0$, this is of the form $\exp( - 2 \pi i L_0 )$ as expected, cf.\ \eqref{eq:virasoro-untwisted}. In $\Cc_1$, on the other hand, it is not, due to the parity involution $\omega_X$, cf.\ \eqref{eq:Virasoro-modes-twisted}. This may seem worry-some at first, but note that $\exp( - 2 \pi i L_0 )$ is not actually a morphism in $\Cc_1$ because it does not commute with the action of $a_m$ ($m \in \Zb + \frac12$). Instead, the ribbon structure presented here is geared towards the ``even sub-theory of symplectic fermions'': $\theta_X$ is the eigenvalue of $\exp( - 2 \pi i L_0 )$ on the even subspace of $\Indtw(X)$. This will be elaborated in more detail in \cite{DR-holo}.
\end{remark}

\section{Two applications to symplectic fermions}

\subsection{Torus blocks}\label{sec:torusblock}


In rational conformal field theory, the characters of the irreducible representations of the underlying vertex operator algebra transform into each other under the action of the modular group \cite{Zhu1996}. In logarithmic conformal field theories this is typically not the case, as the space of conformal blocks on the torus is no longer spanned by characters \cite{Miyamoto:2002od,Flohr:2005cm}. One may attempt to span the space of torus blocks by one-point blocks instead of zero-point blocks. In doing so one must ensure that the field inserted on the torus does not contribute a conformal factor under modular transformations. In this section we will investigate this approach for the even sub-theory of symplectic fermions.

\medskip

Consider $N$ pairs of symplectic fermions, $\h = \Cb^{0|2N}$. To each object in $\Cc\fd$ we can assign the character of its induced module with or without parity involution. Let $A \in \Cc\fd$ and abbreviate $\hat A = \Indtwutw(A)$. We set
\be
  \chi_A^+(\tau) = \mathrm{tr}_{\hat A} \big(\, q^{L_0 - c/24} \, \big)
  \quad , \quad
  \chi_A^-(\tau) = \mathrm{tr}_{\hat A}\big(\,  \omega_{\hat A} \, q^{L_0 - c/24} \, \big) \ ,
\ee
where $q = \exp(2\pi i \tau)$ and $\tau$ a complex number in the upper half plane. From these we can obtain the character of the even subspace of $\hat A$ via
\be
  \chi_A^\mathrm{ev}(\tau) = \tfrac12\big( \chi_A^+(\tau) + \chi_A^-(\tau) \big) \ .
\ee
Denote by $\Pi$ the parity shift functor on $\svect$. We have $\chi_A^\pm(\tau) = \pm \chi_{\Pi A}^\pm(\tau)$. Recall from \eqref{eq:sympferm-simple} the simple objects $\one, \Pi\one, T, \Pi T$ of $\Cc$. Each character $\chi_A^+(q)$ is a $\Zb_{\ge 0}$ linear combination of the characters of the four irreducibles. The latter characters are \cite{Kausch:1995py}, \cite[Sect.\,5.2]{Abe:2005}:
\ba
\chi_\one^+(\tau)
&=  \Big( q^{\frac1{24}} \prod_{n=1}^\infty (1+q^n) \Big)^{2N}
\quad , \quad &
\chi_\one^-(\tau) &=  \Big( q^{\frac1{24}} \prod_{n=1}^\infty (1-q^n) \Big)^{2N} \quad ,
\nonumber \\
\chi_T^+(\tau) &=  \Big( q^{-\frac1{48}} \prod_{n=1}^\infty (1+q^{n-\frac12}) \Big)^{2N}
\quad , \quad &
\chi_T^-(\tau)  &=  \Big( q^{-\frac1{48}} \prod_{n=1}^\infty (1-q^{n-\frac12}) \Big)^{2N} \ .
\label{eq:symp-ferm-irrep-char}
\ea
We are interested in the modular transformation properties of the corresponding characters $\chi^\mathrm{ev}$. The $T$-transformation $\tau \mapsto \tau+1$ produces a phase, while the $S$-transformation $\tau \mapsto -1/\tau$ gives  \cite{Kausch:1995py}, \cite[Sect.\,5.2]{Abe:2005}
\ba
\chi_{\one}^\mathrm{ev}(-1/\tau)
&= 
2^{-N-1} 
\big( \chi_{T}^\mathrm{ev}(\tau) - \chi_{\Pi T}^\mathrm{ev}(\tau) \big)
\,+\,
\tfrac12(-i \tau)^{N}
\big( \chi_{\one}^\mathrm{ev}(\tau) - \chi_{\Pi\one}^\mathrm{ev}(\tau) \big) \ ,
\nonumber \\  
\chi_{\Pi\one}^\mathrm{ev}(-1/\tau)
&=
2^{-N-1} 
\big( \chi_{T}^\mathrm{ev}(\tau) - \chi_{\Pi T}^\mathrm{ev}(\tau) \big)
\,-\,
\tfrac12(-i \tau)^{N}
\big( \chi_{\one}^\mathrm{ev}(\tau) - \chi_{\Pi\one}^\mathrm{ev}(\tau) \big) \ ,
\nonumber \\  
\chi_{T}^\mathrm{ev}(-1/\tau)
&=
\tfrac12 
\big( \chi_{T}^\mathrm{ev}(\tau) + \chi_{\Pi T}^\mathrm{ev}(\tau) \big)
\,+\,
2^{N-1} 
\big( \chi_{\one}^\mathrm{ev}(\tau) + \chi_{\Pi\one}^\mathrm{ev}(\tau) \big) \ ,
\nonumber \\  
\chi_{\Pi T}^\mathrm{ev}(-1/\tau)
&=
\tfrac12 
\big( \chi_{T}^\mathrm{ev}(\tau) + \chi_{\Pi T}^\mathrm{ev}(\tau) \big)
\,-\,
2^{N-1} 
\big( \chi_{\one}^\mathrm{ev}(\tau) + \chi_{\Pi\one}^\mathrm{ev}(\tau) \big) \ .
\label{eq:symp-ferm-Strans}
\ea
We see that the even characters do not close onto themselves under the $S$-transformation but instead produce terms multiplying a power of $\tau$. We will now produce these terms via one-point functions on the torus.

\medskip

More specifically, we will look at torus one-point blocks with insertions from $U(\h)$. For $A \in \Cc\fd$, $f : U(\h) \ast A \to A$, $u \in U(\h)$, $\tau$ in the upper half plane and $z \in \Cb$, define
\be \label{eq:Z-1pt-torus-def}
  Z^\pm_A(f,u;\tau,z) = \mathrm{tr}_{\hat A} \Big( \omega_{\hat A}^{\pm} \, q^{L_0-c/24} \, V(f;e^{2 \pi i z}) \circ \big\{ (e^{2 \pi i z L_0} \tilde u) \otimes \id \big\} \Big) \ ,
\ee
where $\omega^+ = \id$, $\omega^- = \omega$, and $\tilde u = e^{\ln(4) \, L_0} u$ for $A \in \Cc_0$ while $\tilde u = e^{2\ln(4) \, L_0} u$ for $A \in \Cc_1$. This redefinition of $u$ makes the final answer look simpler.

We can think of \eqref{eq:Z-1pt-torus-def} as a conformal one-point block on the torus $\Cb / (\Zb + \tau \Zb)$ with an insertion of $u$ at the point $z$. The next proposition expresses $Z^\pm_A(f,u;\tau,z)$ in terms of the characters \eqref{eq:symp-ferm-irrep-char} and finite dimensional traces; it also shows that $Z$ is actually independent of $z$ as one would expect by translation invariance on the torus.

\begin{proposition}\label{prop:torus-1-pt}
Let $A \in \Cc\fd$ and fix a morphism $f : U(\h) \ast A \to A$. Then
$$
Z^\eps_A(f,u;\tau,z) = 
\begin{cases}
A \in \Cc_0 : &
  \mathrm{tr}_A\big( \omega_{A}^{\eps} \circ f \circ (u \otimes e^{2 \pi i \tau L_0}) \big) 
  \cdot \chi_\one^{\nu}(\tau) 
\\
A \in \Cc_1 : &
  \mathrm{tr}_A\big( \omega_{A}^{\eps} \circ f \circ (u \otimes \id_A) \big) 
  \cdot \chi_T^{\nu}(\tau) 
\end{cases}
$$
where $\nu = \eps (-1)^{|u|}$.
\end{proposition}

\begin{proof}
We compute the trace over $\hat A$ in two steps: first over the $U(\h_{<0})$ factor of the induced module and then over the space of ground states. Let $\{ e_j \}$ be a basis of the space ground states $A$.
Abbreviate $V \equiv V(f;e^{2 \pi i z}) \circ \big\{ (e^{2 \pi i z L_0} \tilde u) \otimes \id \big\}$. 
Consider a basis vector of $\hat A$ of the form $\alpha^{i_1}_{-m_1} \cdots \alpha^{i_n}_{-m_n} e_j$ and abbreviate the dual basis element by $\psi \equiv \psi^{i_1,\dots,i_n}_{m_1,\dots,m_n;j}$. Then
\ba
  &\big\langle \psi , \omega_{\hat A}^{\eps} \, q^{L_0-c/24} \, V
  \alpha^{i_1}_{-m_1} \cdots \alpha^{i_n}_{-m_n} e_j \big\rangle
  \nonumber \\ &
  \overset{(1)}=  (-1)^{|u|(|\alpha^{i_1}| + \cdots + |\alpha^{i_n}|)}
  \big\langle \psi , \omega_{\hat A}^{\eps} \, q^{L_0-c/24} \, 
  \alpha^{i_1}_{-m_1} \cdots \alpha^{i_n}_{-m_n} V e_j \big\rangle \ ,
  \nonumber \\ &
  \overset{(2)}= \big(\eps(-1)^{|u|}\big)^{|\alpha^{i_1}| + \cdots + |\alpha^{i_n}|} q^{m_1+\cdots+m_n-c/24}
  \big\langle e_j^* , \omega_{A}^{\eps} \, q^{L_0} \, 
   V e_j \big\rangle \ ,
\ea
Step 1 follows from the mode exchange relation in cases a) and b) of Definition \ref{def:vertex-op}\,(iii). The $x^m \cdot a \otimes \id$ summand in the exchange relation removes one mode $a^k_{-m_k}$ from the vector in $\hat A$ and so gives zero when contracted with $\psi$. In step 2 the negative modes are moved past $\omega_{\hat A}^{\eps} \, q^{L_0-c/24}$ and contracted with $\psi$, leaving only the ground state part. This shows that
\be
  Z^\pm_A(f,u;\tau,z) = 
  \begin{cases}
A \in \Cc_0 : &
  \mathrm{tr}_A\big( \omega_{A}^{\eps} \, e^{2 \pi i \tau L_0} \, V \big) 
  \cdot \chi_\one^{\nu}(\tau) 
\\
A \in \Cc_1 : &
  \mathrm{tr}_A\big( \omega_{A}^{\eps} \, V \big) 
  \cdot \chi_T^{\nu}(\tau) 
\end{cases}
\ee
where the explicit action of $L_0$ on ground states has been inserted. For $A \in \Cc_0$ this was given in \eqref{eq:virasoro-untwisted}.
For $A \in \Cc_1$, according to \eqref{eq:Virasoro-modes-twisted} the Virasoro mode $L_0$ acts on $A$ as $-\frac{N}8 \cdot \id_A$, and this factor now forms part of the character $\chi_T$. Hence the $e^{2 \pi i \tau L_0}$ is absent for $A \in \Cc_1$.

In the above expression, $V$ is evaluated on ground states and we can insert the explicit form \eqref{eq:V(f)-on-ground-states}. 

\nxt $A \in \Cc_0$:
\be
\mathrm{tr}_A\big( \omega_{A}^{\eps} \circ e^{2 \pi i \tau L_0} \circ V \big) 
=
\mathrm{tr}_A\big( \omega_{A}^{\eps} \circ e^{2 \pi i \tau L_0} \circ F \circ (\tilde u \otimes \id_A) \big) 
\ee
where, for $\ell := 2 \pi i z+\ln(4)$,
\ba
  F &= 
f \circ \exp\!\big(
\ell \cdot \Omega^{(12)}+ 2 \pi i z \cdot \tfrac12 \Omega^{(11)} \big)
 \nonumber \\
&=
f \circ \exp\!\big(
\ell \cdot \tfrac12 (\Omega^{(11)} + 2\Omega^{(12)} + \Omega^{(22)})
- \ell \cdot \tfrac12  \Omega^{(22)}
- \ln(4) \cdot  \tfrac12 \Omega^{(11)}
\big)
 \nonumber \\
&=
\exp\!\big( \ell \cdot \tfrac12 \Omega^{(11)}
\big)
\circ
f \circ \exp\!\big(
-  \ell \cdot \tfrac12  \Omega^{(22)}
- \ln(4) \cdot \tfrac12 \Omega^{(11)}
\big)
\ea
The factor $\exp\!\big( \ell \cdot \tfrac12 \Omega^{(11)}\big)$ to the left of $f$ can be taken around the trace and will cancel with the corresponding exponential to the right of $f$. This results in the expression claimed in the statement for $A \in \Cc_0$.

\nxt $A \in \Cc_1$: We have
\be
  \mathrm{tr}_A\big( \omega_{A}^{\eps} \, V \big) 
  = 
  \mathrm{tr}_A\big( \omega_{A}^{\eps} \,
  f \circ \exp\!\big(- ( \pi i z +\ln(4) ) \cdot \Omega^{(11)} \big)
  \circ \big\{ (e^{2 \pi i z L_0} \tilde u) \otimes \id_A \big\} \big) \ .
\ee
Since $(e^{2 \pi i z L_0} \tilde u) \otimes \id_A = e^{(\pi i z + \ln(4)) \Omega^{(11)}} \circ ( u \otimes \id_A )$, 
this gives the statement of the proposition for $A \in \Cc_1$.
\end{proof}

Below we will drop the the argument $z$ from the notation $Z^\eps_A(f,u;\tau,z)$.

\medskip

Using $\chi_{\one}^\mathrm{ev}(\tau) - \chi_{\Pi\one}^\mathrm{ev}(\tau)  = \chi_\one^-(\tau)$, from \eqref{eq:symp-ferm-Strans} one sees that the vector space of functions on the upper half plane spanned by the irreducible characters and their modular transformations is the $(N{+}4)$-dimensional space
\be\label{eq:modularclosure}
  \mathrm{span}_{\Cb} \big\{ \, \chi_{\one}^\mathrm{ev}(\tau) , \chi_{\Pi\one}^\mathrm{ev}(\tau) , \chi_{T}^\mathrm{ev}(\tau) , \chi_{\Pi T}^\mathrm{ev}(\tau) \,\big\}
  ~\oplus~
  \mathrm{span}_{\Cb} \big\{ \tau^k \cdot \chi_\one^-(\tau) \, \big| \, k = 1, \dots,N \big\} \ .
\ee
In \cite{Feigin:2005zx}, an analogous structure for the modular closure of the characters was found for the $W_{1,p}$-models.

Proposition \ref{prop:torus-1-pt} shows that one can produce terms of the form $\tau^k \cdot \chi_\one^\pm(\tau)$ if $A$ is in $\Cc_0$ and has a non-trivial (necessarily nilpotent) $L_0$-action. 
An argument similar to Lemma \ref{lem:Hopf-aux2} shows that the most general $\h$-module map $f : U(\h) \otimes A \to A$ is of the form
$f_g(u \otimes a) = \sum_{(u)} u_{(1)}.g\big(S(u_{(2)}).a\big)$ where $g : A \to A$ is an arbitrary even linear map. 

\begin{lemma}\label{lem:odd-trace}
For $A \in \Cc_0$, $g \in \End(A)$, $u \in U(\h)$ and $f_g$ as above,
\be
  Z_{A}^-(f_g,u;\tau) = \eps(u) \cdot \mathrm{str}_A\big( g \circ e^{2 \pi i \tau L_0} \big) 
  \cdot \chi_\one^{-}(\tau) \ .
\ee
\end{lemma}

\begin{proof}
By Proposition \ref{prop:torus-1-pt}, $Z_{A}^-(f_g,u;\tau) =  \mathrm{str}_A\big( f_g \circ (u \otimes e^{2 \pi i \tau L_0}) \big)  \cdot \chi_\one^{-}(\tau)$.
If $u$ is odd, the trace is over an odd endomorphism and hence zero. For $u$ even we have
\ba
&  \mathrm{str}_A\big( f_g \circ (u \otimes e^{2 \pi i \tau L_0}) \big) 
=
\sum_{(u)} 
  \mathrm{str}_A\big( u_{(1)}g(S(u_{(2)}).-) \circ e^{2 \pi i \tau L_0}) \big) 
  \nonumber \\
&\overset{(1)}=
\sum_{(u)} (-1)^{|u_{(1)}|}
  \mathrm{str}_A\big( g(S(u_{(2)}) u_{(1)}.-) \circ e^{2 \pi i \tau L_0}) \big) 
\overset{(2)}=
 \eps(u) \cdot \mathrm{str}_A\big( g \circ e^{2 \pi i \tau L_0} \big) 
\ea
In (1) the endomorphism $u_{(1)}$ is taken through the super-trace. In (2) the identities 
$\sum_{(u)} (-1)^{|u_{(1)}|} S(u_{(2)}) u_{(1)} = \sum_{(u)} u_{(1)}  S(u_{(2)})  = \eps(u) \, 1$ were used.
\end{proof}

Consider the case $A = U(\h)$ and $g(x) = 1 \cdot \alpha(L_0^{\,N-k}x)$ where $\alpha$ is a top-form on $U(\h)$, say the one satisfying $\alpha(L_0^{\,N}) = 1$. Since the image of $g$ is $\Cb\,1$, the trace in Lemma \ref{lem:odd-trace} can be taken just over $\Cb\,1$. The result is
\be\label{eq:vacu-pseudo-trace}
  Z_{U(\h)}^-(f_g,1;\tau) = 
  \alpha\big((L_0)^{N-k} e^{2 \pi i \tau L_0} \big) \cdot  \chi_\one^{-}(\tau) 
  = \tfrac{(2 \pi i \tau)^k}{k!} \cdot  \chi_\one^{-}(\tau) \ ,
\ee
i.e.\ precisely the ``non-character directions'' in the modular closure \eqref{eq:modularclosure}.
The idea to obtain torus blocks which are not in the span of characters as torus one-point functions of $\omega$ is mentioned in \cite{Gainutdinov:2007tc} (for the $W_{1,p}$-models) but without computing the corresponding traces. The above calculation is the first explicit evaluation of such a trace.

The functions \eqref{eq:vacu-pseudo-trace} are examples of so-called pseudo-traces \cite{Miyamoto:2002od} (which are defined differently from $Z^\pm_A$, i.e.\ not as traces over vertex operators). Pseudo-traces where studied in \cite{Adamovic:2009} for the triplet $\mathcal{W}(p)$-models and in \cite{Arike:2011ab} for (the even part of) symplectic fermions. Indeed, the pseudo-traces with vacuum insertion computed in \cite[Thm.\,6.3.2]{Arike:2011ab} agree with \eqref{eq:vacu-pseudo-trace} up to normalisation.

\medskip

As in Section \ref{sec:examples-VO} (in the case $N=1$), let us write $\omega \equiv 1 \in U(\h)$ to avoid confusion with the `identity field'. The function $Z_{U(\h)}^-(f,\omega;\tau)$ should be interpreted as a conformal one-point block on the torus with an insertion of $\omega$. One may now wonder why this can transform in the same way as a zero-point block on the torus, given that an $S$-transformation introduces an action of rotation and rescaling of the form $\omega \mapsto e^{\ln(-1/\tau) L_0} \omega$ (for $z=0$).
The resolution is that 
\be \label{eq:S-trans-aux2}
  Z_{U(\h)}^-(f_g,L_0^{\,k}\omega;\tau) = 0
  \qquad \text{for} ~~ k>0 \ ,
\ee
as is immediate from Lemma \ref{lem:odd-trace} and the fact that $\eps(L_0^{\,k}\omega)=0$ for all $k>0$.

\subsection{Algebra structure on $U(\h)$}\label{sec:Uh-algebra}

In this application we compute the algebra structure on $U(\h)$ which lies in the Morita class of $\one$. In the case $\h = \Cb^{0|2}$, this algebra structure can be compared to the operator product expansion of boundary fields in the triplet model calculated in \cite{Gaberdiel:2006pp}.

\medskip

The rigid structure on $\Cc\fd$ defined in \cite[Sect.\,3.7]{Davydov:2012xg} assignes to a finite dimensional $X \in \Cc_1$ the dual vector space $X^*$ of all (even and odd) linear maps to $\Cb$. Let $T = \Cb^{1|0} \in \Cc_1$ as above. Then $T^* = (\Cb^{1|0})^*$ and $T \ast T^* = U(\h) \otimes \Cb^{1|0} \otimes (\Cb^{1|0})^* \cong U(\h)$.

The object $T \ast T^*$ carries a canonical structure of an associative algebra (namely that of the internal End of $T$ with respect to the action of $\Cc$ on itself, hence it is in the same Morita class as $\one$). The product $\mu : U(\h) \ast U(\h) \to U(\h)$ on $U(\h) \cong T \ast T^*$ is computed as follows.
Denote by $\ev_T : T^* \ast T \to \one$ the evaluation map of the right duality on $\Cc\fd$. This is given in  \cite[Sect.\,3.7]{Davydov:2012xg} to be $\ev_T(u \otimes \varphi \otimes z) = \eps(u) \cdot \varphi(z)$, where $u \otimes \varphi \otimes z \in U(\h) \otimes (\Cb^{1|0})^* \otimes \Cb^{1|0}$ and $\eps$ is the counit on $U(\h)$. Write $m : U(\h) \otimes U(\h) \to U(\h)$ the Hopf algebra product on $U(\h)$; note that $m$ cannot be the sought-for product on $T \ast T^*$ since it is not an intertwiner of $\h$-modules and thus not a morphism in $\Cc_0$. Then
\ba
  \mu 
  &= \Big[
  U(\h) \ast U(\h)
  \cong
  (T \ast T^*) \ast (T \ast T^*)
  \xrightarrow{\alpha^{-1}_{T,T^*,TT^*}}
  T \ast (T^* \ast (T \ast T^*))
\nonumber \\
& \hspace{3em}  
  \xrightarrow{\id \ast \alpha_{T^*,T,T^*}}
  T \ast ((T^* \ast T) \ast T^*)
  \xrightarrow{\id \ast (\ev_T \ast \id)}
  T \ast (\one \ast T^*)
  \cong T \ast T^* \cong U(\h)
  \Big]  
\nonumber \\
  &= \big\{ \id \otimes \eps \big\} \circ \big\{ \id \otimes \phi \big\} \circ \big\{ (\id \otimes m) \circ (\id \otimes S \otimes \id) \circ (\Delta \otimes \id) \big\}
\nonumber \\
  &= (\id \otimes (\lambda \circ m)) \circ (\id \otimes S \otimes \id) \circ (\Delta \otimes \id)  \ .
\ea
The three sets of curly brackets in the second line indicate how the morphisms arise from those in the previous line. In the last line we simplified $\eps \circ \phi = \lambda$, with $\phi$ from \eqref{eq:111-phi-def-repeat} and $\lambda$ from \eqref{eq:111-lambda-def}.

The unit for this algebra is given by the Hopf algebra integral $\Lambda = (-\pi)^N / N! \cdot L_0^{\,N}1$. Indeed, since $\Lambda$ is an integral on $U(\h)$ we have $\mu \circ (\id \otimes \Lambda) = \id \cdot \lambda(\Lambda) = \id$. That also $\mu \circ (\Lambda \otimes \id) = \id$ follows for example from the identity in \cite[Lem.\,2.3\,b)]{Davydov:2012xg}.

As in Section \ref{sec:examples-VO} we set $\omega \equiv 1 \in U(\h)$. For comparison to \cite{Gaberdiel:2006pp} it is convenient to normalise the algebra such that the unit is given by $\hat\Omega := L_0^{\,N}\omega$. Namely, we pass to the isomorphic algebra with
\be
\text{mult. :}~~
  \hat\mu =\tfrac{ (- \pi)^N }{N!} \cdot (\id \otimes (\lambda \circ m)) \circ (\id \otimes S \otimes \id) \circ (\Delta\otimes \id)
  ~~,\quad
\text{unit :} ~~\hat\Omega \ .
\ee 
Via Theorem \ref{thm:*-tensor}, the morphism $\hat\mu : U(\h) \ast U(\h) \to U(\h)$ determines a vertex operator $V(\hat\mu;x) : U(\h) \otimes \Ind(U(\h)) \to \overline\Ind(U(\h))$. Let us look at this vertex operator explicitly in the case that $\h = \Cb^{0|2}$. The ground state contribution of $V(\hat\mu;x)(\omega \otimes \omega)$ can be read off from \eqref{eq:V(f)-on-ground-states} and was already worked out in \eqref{eq:case-a-omom} (up to the normalisation $\ln(4)$ from \eqref{eq:V(f)-on-ground-states} which amounts to replacing $\ln(x)$ by $\ln(4x)$):
\ba
V(\hat\mu;x)(\omega \otimes \omega) &= \hat\mu(\omega \otimes \omega) + \ln(4x) \cdot \hat\mu( \chi^- \omega \otimes \chi^+ \omega - \chi^+ \omega\otimes \chi^-\omega)
  \nonumber \\ 
  & \hspace{7em}- (\ln(4x))^2 \cdot \hat\mu(\hat\Omega \otimes \hat\Omega) ~+~\text{(higher terms)} \ ,
  \nonumber \\ 
  &= - (\ln(4x))^2 \cdot \hat\Omega - 2\ln(4x) \cdot \omega
  ~+~\text{(higher terms)}  \ .
\label{eq:om-om-OPE}
\ea
Apart from the fact that $\hat\Omega$ is the unit for $\hat\mu$ we used
$\hat\mu(\omega \otimes \omega) = -\pi\, \lambda(1 \cdot 1) \cdot 1  = 0$, 
$\hat\mu( \chi^- \omega \otimes \chi^+ \omega ) = - \pi\,\lambda(-\chi^- \chi^+) \cdot 1  = - \omega$ and
$\hat\mu( \chi^+ \omega \otimes \chi^- \omega ) = \omega$.

\medskip

We will now compare this result to \cite{Gaberdiel:2006pp}. There, the triplet model of central charge $c=-2$ was investigated on the upper half plane. The chiral symmetry of the triplet model is the even sub-vertex operator algebra of a single pair of symplectic fermions. In the ``Cardy case'' one expects a boundary condition whose space of boundary fields is just the vacuum representation.\footnote{
  This is the case for rational conformal field theory but may fail for logarithmic models \cite{Gaberdiel:2009ug,Runkel:2012rp}. However, for the triplet model
  there is such a boundary condition \cite{Gaberdiel:2006pp}.}
The other boundary conditions should then be labelled by representations $M$ of the chiral algebra, and the spaces of boundary fields should be given by the fusion product of $M$ with the dual $M^*$.\footnote{ 
  For logarithmic models not all $M$ are allowed. One needs to require the existence of a non-degenerate pairing on
  the resulting space of boundary fields, see \cite{Gaberdiel:2009ug,Runkel:2012rp}. }
This was tested in \cite{Gaberdiel:2006pp} for the irreducible representations. Consider the boundary condition labelled by the representation $\mathcal{V}_{-1/8}$ of the triplet algebra of lowest conformal weight $h = -\frac18$. This representation is the even subspace of $\Indtw(T)$, cf.\ Remark \ref{rem:C-braiding}\,(ii). The space of boundary fields is the representation $\mathcal{R}_0$ of the triplet algebra \cite[Sect.\,2.3]{Gaberdiel:2006pp}, which is the even subspace of $\Ind(U(\h)) = \Ind(T \ast T^*)$. The boundary OPE found in \cite[Eqn.\,(2.23)]{Gaberdiel:2006pp} reads
\be
  \omega^\text{GR}(x) \omega^\text{GR}(0) ~=~ - (\ln(x))^2 \cdot \Omega^\text{GR}(0) - 2 \ln(x) \cdot \omega^\text{GR}(0) ~+~ \cdots \ ,
  \label{eq:om-om-OPE-GR}
\ee
where $\Omega^\text{GR} = L_0 \, \omega^\text{GR}$. The expansions \eqref{eq:om-om-OPE} and \eqref{eq:om-om-OPE-GR} agree upon setting $\omega^\text{GR} = \omega + \ln(4) \hat\Omega$ which then forces $\Omega^\text{GR} = \hat\Omega$.

\appendix

\section{Proofs omitted in Sections \ref{sec:mode+rep} and \ref{sec:VO+tensor}}

\subsection{Proof of Theorem \ref{thm:ind-equiv-untw}} \label{app:ind-equiv}

The proof follows \cite[Sect.\,1.7]{Frenkel:1987} and \cite[Sect.\,3.5]{Kac:1998}. It proceeds by first investigating induced modules of a Lie subalgebra $\tilde\h$ which is defined as ``$\hat\h$ without zero modes''. Since the essential difference between $\hat\h$ and $\hat\h_\mathrm{tw}$ is the presence of zero modes in $\hat\h$, analogous results hold for $\tilde\h_\mathrm{tw}$ and imply Theorem \ref{thm:ind-equiv-tw} (in fact we even have $\hat\h_\mathrm{tw} = \tilde\h_\mathrm{tw}$). Here, only the proof of Theorem \ref{thm:ind-equiv-untw} will be given explicitly.

\medskip

Let
\be
  \tilde\h = \mathrm{span}\{ a_m \,|\, m \neq 0 , a \in \h \} \oplus \Cb K ~\subset~ \hat\h \ .
\ee
This is a Lie-subalgebra of $\hat\h$. For all $\tilde\h$-modules below we will assume that $K$ acts as the identity. Given an $\tilde\h$-module $M$, we define the {\em space of ground states} as
\be
  M_\mathrm{gs} = \{ v \in M \,|\, a_m v = 0 \text{ for all $m > 0$ and $a \in \h$}\,\} \ .
\ee
A $\tilde\h$-module is {\em bounded below} if the condition in Definition \ref{def:bounded-below-rep} holds (with all $m_i \neq 0$).

\begin{lemma} \label{lem:M_gs=0_implies_M=0}
Let the $\tilde\h$-module $M$ be bounded below. If $M_\mathrm{gs} = \{0\}$ then also $M = \{0\}$.
\end{lemma}

\begin{proof}
Suppose $v \in M$ with $v \neq 0$. By boundedness there is $N$ such that $a^1_{m_1} a^2_{m_2} \cdots a^L_{m_L} v = 0$ whenever $m_1+\dots+m_L > N$. By iteration we can find $m_1,\dots,m_n>0$ such that $w := a^1_{m_1} \cdots a^n_{m_n} v \neq 0$ but $a_m w = 0$ for all $m>0$. But then $w \in M_\mathrm{gs}$, in contradiction to $M_\mathrm{gs} = \{0\}$.
\end{proof}

The $\tilde\h$-module $U(\tilde\h)$ shows that the boundedness condition in the above lemma is necessary. 

A super-vector space $V$ will be understood as a $\tilde\h_{>0} \oplus \Cb K$ module by letting $\tilde\h_{>0}$ act as $0$ and $K$ as the identity. Denote the induced $\tilde\h$-module as
\be
  J(V) := U(\tilde\h) \otimes_{U(\tilde\h_{>0} \oplus \Cb K)} V \ .
\ee
The induced modules $J(V)$ are bounded below. Recall from \eqref{eq:H-op-def} the operator $H$ which acts on $\tilde\h$-modules that are bounded below.

\begin{lemma}\label{lem:H-eval-on-induced}
$H$ is diagonalisable on $J(V)$ with eigenvalues in $\Zb_{\ge 0}$. The $H$-eigenspace of eigenvalue $0$ is $V$.
\end{lemma}

\begin{proof}
Let $\{v_j\}$ be a basis of $V$ and let $\{ \alpha^i \}$ be a basis of $\h$.
By the PBW-theorem, $\{ \alpha^{i_1}_{-m_1} \cdots \alpha^{i_n}_{-m_n} v_j \}$ (where $m_k>0$ and some ordering prescription on the $\alpha^i_{-m}$ is implicit) is a basis of $J(V)$. Since $[H,a_m] = -m a_m$, this basis consists of eigenvectors of $H$ with eigenvalue in $\Zb_{\ge 0}$. The only basis vectors of eigenvalue $0$ are the $v_j$.
\end{proof}

\begin{theorem}\label{thm:tilde-h-all-induced}
Let the $\tilde\h$-module $M$ be bounded below. Then the embedding $M_\mathrm{gs} \hookrightarrow M$ lifts to a unique isomorphism $J(M_\mathrm{gs}) \cong M$ of $\tilde\h$-modules.
\end{theorem}

\begin{proof}
Denote the isomorphism $J(M_\mathrm{gs}) \to M$ by $j$. Existence and uniqueness of the $\tilde\h$-module map $j$ follow from the universal property of induced modules. It remains to show that $j$ is an isomorphism.

\smallskip\noindent
\nxt $j$ is injective: Let $K \subset J(M_\mathrm{gs})$ be the kernel of $j$. Since $J(M_\mathrm{gs})$ is bounded below, so is $K$. From the definition of $H$ we see that $H$ acts as zero on $K_\mathrm{gs}$. By Lemma \ref{lem:H-eval-on-induced}, the $H$-eigenspace of eigenvalue $0$ in $J(M_\mathrm{gs})$ is $M_\mathrm{gs}$. Thus $K_\mathrm{gs} \subset M_\mathrm{gs}$. Since the embedding $M_\mathrm{gs} \hookrightarrow M$ is injective, also the map $j$ is injective on $M_\mathrm{gs}$. Hence $K_\mathrm{gs}=\{0\}$. By Lemma \ref{lem:M_gs=0_implies_M=0} then also $K=\{0\}$.

\smallskip\noindent
\nxt $j$ is surjective: Let $M'$ be the image of $J(M_\mathrm{gs})$ in $M$. The quotient module $M/M'$ is bounded below. Suppose that $M/M' \neq \{0\}$. By Lemma \ref{lem:M_gs=0_implies_M=0} then also $(M/M')_\mathrm{gs} \neq \{0\}$. Pick a nonzero $u \in (M/M')_\mathrm{gs}$ and write $u = v + M'$ for some $v \in M$. Since $Hu=0$ we have $Hv = m'$ with $m' \in M'$. By Lemma \ref{lem:H-eval-on-induced}, $J(M_\mathrm{gs})$ is a sum of $H$-eigenspaces with non-negative integral eigenvalues. Hence so is $M'$ and we can write $m' = m_0' + \sum_{k>0} m_k'$ with $H m_k' = k m_k'$. Set
\be
  w = v - \sum_{k>0} \frac{1}{k} m_k' \ .
\ee
The vector $w$ satisfies $w + M' = u$ and $Hw = Hv - \sum_{k>0} m_k' = m_0'$. Since $0$ is the lowest $H$-eigenvalue on $M'$, we have $a_k m_0' = 0$ for all $k>0$. Suppose there is a $k>0$ such that $a_k w \neq 0$. Then $a_k w \in M'$ (since $a_k u = 0$ in $M/M'$) and $H a_k w = [H,a_k]w + a_k H w = - k a_k w + a_k m_0' = -k a_kw$. But by Lemma \ref{lem:H-eval-on-induced} there are no negative $H$-eigenvalues on $M'$ and so $a_k w = 0$ for all $k>0$. This shows that $w \in M_\mathrm{gs}$. But $M_\mathrm{gs}$ is contained in the image of $j$ by construction and so $u = w + M' = 0$ in $M/M'$, in contradiction to our assumption $u \neq 0$.
It follows that $(M/M')_\mathrm{gs} = \{0\}$ and by Lemma \ref{lem:M_gs=0_implies_M=0} thereby $M/M' = \{0\}$.
\end{proof}

To make contact with Theorem \ref{thm:ind-equiv-untw}, observe that if for a given $\h$-module $R$ one restricts the $\hat\h$-action on the induced $\hat\h$-module $\Ind(R)$ to an $\tilde\h$-action, one obtains the induced module $J(R)$ (which is independent of the $\h$-action). This follows since both have the same PBW basis $\{ \alpha^{i_1}_{-m_1} \cdots \alpha^{i_n}_{-m_n} v_j \}$ already used above.

\bigskip

\begin{proof}[Proof of Theorem \ref{thm:ind-equiv-untw}]
Denote by $\mathrm{Gs} : \Rep_{\flat,1}(\hat\h) \to \Rep(\h)$ the functor that acts on objects as $\mathrm{Gs}(M) = M_\mathrm{gs}$ and on morphisms $f: M \to N$ as $\mathrm{Gs}(f) = f|_{M_\mathrm{gs}}$ (since $f$ is an $\hat\h$-intertwiner, $f$ maps ground states to ground states).

\smallskip\noindent
\nxt $\mathrm{Gs} \circ \Ind \cong \Id$: The $\h$-module $R$ is embedded in $\Ind(R)$, and this embedding restricts to a (natural) isomorphism $R \to (\Ind(R))_\mathrm{gs}$.

\smallskip\noindent
\nxt $\Ind \circ \mathrm{Gs} \cong \Id$: For an $\hat\h$-module $M$, the embedding $M_\mathrm{gs} \hookrightarrow M$ induces an $\hat\h$-module map $j_M : \Ind(M_\mathrm{gs}) \to M$. Restricting $\Ind(M_\mathrm{gs})$ to an $\tilde\h$-module gives $J(M_\mathrm{gs})$ and Theorem \ref{thm:tilde-h-all-induced} then states that the map $j_M$ is an isomorphism. The universal property of induced modules shows that $j_M$ is natural in $M$.
\end{proof}

\subsection{Proof of Lemma \ref{lem:Qf-exists}}\label{app:lemma-Qf-exists}

It remains to show that a map $Q_f$ with the properties stated in Lemma \ref{lem:Qf-exists} exists. We will make use of the subalgebra $\tilde\h \subset \hat\h$ and the induced $\tilde\h$-modules $J(V)$ introduced in Appendix \ref{app:ind-equiv}. Recall that restricting $\Ind(S)$ to an $\tilde\h$-module results in $J(S)$. 

For $r \in R$ let $F_r : S \to T \subset \Ind(T)$ be the map $s \mapsto f(r \otimes s)$. By the universal property of induced $\tilde\h$-modules, this lifts to a unique map $\bar F_r : J(S) \to \Ind(T)$ of $\tilde\h$-modules. Define 
\be
  Q_f(r \otimes v) := \bar F_r(v) \ .
\ee
By construction, $Q_f(r\otimes s) = F_r(s) = f(r \otimes s)$, so that $Q_f$ satisfies property 1 in Lemma \ref{lem:Qf-exists}. Since $Q_f(r \otimes -)$ is an $\tilde\h$-intertwiner from $\Ind(S)$ to $\Ind(T)$ it satisfies property 2 for all $m \neq 0$. For $m=0$ we verify property 2 on the spanning set $r \otimes (a^1_{-m_1} \cdots a^k_{-m_k} s)$ of $R \otimes \Ind(S)$ (here $a^i \in \h$ and $m_i >0$). Let $b \in \h$ and let $\eps$ be the parity sign $(-1)^{|b|(|a^1|+\cdots+|a^k|)}$. We have
\ba
  &b_0 \, Q_f\big( (\id \otimes (a^1_{-m_1} \cdots a^k_{-m_k})(r \otimes s) \big)
  \overset{(1)}=
  b_0 a^1_{-m_1} \cdots a^k_{-m_k} \, Q_f(r \otimes s) 
  \nonumber \\
  &\qquad   \overset{(2)}=
  \eps \cdot
  a^1_{-m_1} \cdots a^k_{-m_k} \, b_0 \, f(r \otimes s) 
  \overset{(3)}=
  \eps \cdot
  a^1_{-m_1} \cdots a^k_{-m_k} \, f\big( (b \otimes \id + \id \otimes b)(r\otimes s) \big)
  \nonumber \\
  &\qquad   \overset{(4)}= \eps \cdot
  Q_f\big( (\id \otimes (a^1_{-m_1} \cdots a^k_{-m_k}) (b \otimes \id + \id \otimes b_0)(r\otimes s) \big)
  \nonumber \\
  &\qquad   \overset{(5)}=
  Q_f\big((b \otimes \id + \id \otimes b_0)\, (\id \otimes (a^1_{-m_1} \cdots a^k_{-m_k})(r\otimes s) \big) \ .
\ea
Step (1) is the $\tilde\h$-intertwiner property of $Q_f(r \otimes -)$, step (2) uses $[b_0,a_m]=0$ for all $a,b,m$ and that $Q_f$ equals to $f$ on ground states (property 1), equality (3) holds since $f$ is an $\h$-intertwiner, step (4) and (5) are again the $\tilde\h$-intertwiner property and $[b_0,a_m]=0$.

\subsection{Proof of Theorem \ref{thm:G-bijective} in cases c) and d)} \label{app:G-surjective-cd}

Since we already have a construction of $V_f$ in case b), the most elegant and conceptual way to proceed would be to use M\"obius covariance to relate cases c) and d) to b), see \cite[Sect.\,4]{Dolan:1989vr} (for irreducible $\hat\h$-representations and $\h$ purely even). However, this requires to extend the source of the vertex operators to $\Indtwutw(A) \otimes \Indtwutw(B)$ (cf.\ Remark \ref{rem:VO-def}\,(ii)), which is not treated here. Instead, a direct proof of existence is given for each case separately. 

\medskip

To start with, recall the expression for the half-infinite sums of modes $T(a)^{x,\eps}_m$ in Definition \ref{def:vertex-op}\,(iii) and restrict to $x=1$,  
\be
  T(a)^{1,\eps}_m = \sum_{k=0}^\infty {{1/2} \choose k} (-1)^{k} \cdot a_{m + \eps k} \ .
\ee
In a representation of $\hat\h_{(\mathrm{tw})}$ on which $T(a)^{1,\eps}_m$ is well-defined, we obtain the commutator ($a,b \in \h$, $m,n \in \Zb+\alpha$, $\alpha \in \{0,\frac12\}$)
\ba
\big[\,T(a)^{1,\eps}_m \,,\, T(b)^{1,\eps}_n \,\big] 
&= \sum_{k,l=0}^\infty
 {{1/2} \choose k} {{1/2} \choose l}
  (-1)^{k+l} [ a_{m + \eps k}, b_{n + \eps l} ]
\nonumber\\
&= (a,b) (-1)^{m+n} \sum_{k=0}^{-\eps(m+n)}
 {{1/2} \choose k} {{1/2} \choose {-k-\eps(m+n)}} (m + \eps k)
\nonumber\\
&= (a,b)  \big( m \delta_{m+n,0} - (m + \tfrac12 \eps) \delta_{m+n+\eps,0} \big) \ .
\label{eq:TT-comm}
\ea

As an auxiliary device, we introduce a Lie algebra $\tilde \h_\alpha$ for $\alpha \in \{0,\frac12\}$. Namely, $\tilde \h_\alpha = \h \otimes t^\alpha \Cb[t,t^{-1}] \oplus \Cb K$ with Lie bracket, for $a,b \in \h$ and $m,n \in \Zb+\alpha$,
\be\label{eq:tilde-h-bracket}
  [\tilde a_m,\tilde b_n] = (a,b) \cdot \big\{ m \delta_{m+n,0} - \big(m+\tfrac12\big) \delta_{m+n+1,0}  \big\} \cdot K \ ,
\ee
and where $\tilde a_m = a \otimes t^m$. The element $K$ is central. As a vector space, $\tilde \h_\alpha$ is equal to $\hat\h_{(\mathrm{tw})}$, but the bracket is different. Comparing the bracket \eqref{eq:tilde-h-bracket} to the commutator \eqref{eq:TT-comm} shows that for a representation $Q$ in $\Rep_{\flat,1}(\hat\h)$ (resp.\ $\Rep_{\flat,1}(\hat\h_\mathrm{tw})$), the assignment
\be
  \rho(\tilde a_m) = T(a)^{1,+}_m 
  \quad , \qquad
  \rho(K) = \id 
\ee
defines a representation of $\tilde \h_0$ (resp.\ $\tilde \h_{\frac12}$) on $Q$. The assumption that $Q$ is bounded below guarantees that the infinite sum in $T(a)^{1,+}_m$ truncates to a finite sum when acting on an arbitrary element of $Q$.

\medskip

Let us concentrate on case c) in Definition \ref{def:vertex-op} for concreteness. Case d) can be proved along the same lines and will not be treated explicitly. 

Pick $X,Y \in \svect$ and $R \in \Rep\lf(\h)$ and fix an even linear map $f : X \otimes R \to Y$. We will show that there is a vertex operator $V : \Rb_{\ge 0} \times (X \otimes \Ind(R)) \to \overline\Indtw(Y)$ which satisfies $G(V)=f$. This will be done in two steps:
\begin{enumerate}
\item Construct an even linear map $V(1) : X \otimes \Ind(R) \to \overline\Indtw(Y)$ which satisfies condition (iii) in Definition \ref{def:vertex-op}, restricted to $x=1$. Explicitly,
\be
  T(a)^{1,-}_{m+\frac12} \circ V(1) = i \, V(1) \circ \big( \id \otimes T(a)^{1,+}_m \big) \ .
\ee
The map $V(1)$ must in addition obey $\Pgs \circ V(1)|_{X \otimes R} = f$.

\item Extend $V(1)$ to $V(x)$ for all $x \in \Rb_{> 0}$ and show that conditions (i)--(iii) in Definition \ref{def:vertex-op} hold.
\end{enumerate}
Step 1.\ is equivalent to the following problem: give a bilinear pairing $P :  \Indtw(Y)' \times (X \otimes \Ind(R)) \to \Cb$ subject to the condition, for all $\hat\gamma \in  \Indtw(Y)'$, $v \in X$, $\hat r \in R$ and $a \in \h$, $m \in \Zb$,
\be\label{eq:P-pairing-mode-moving}
  P\big( \,\rho(\tilde a_{-m-\frac12}) \hat\gamma \,,\, v \otimes \hat r \,\big)
  = (-1)^{|a||\hat\gamma|} \,(-i) \cdot 
  P\big( \,\hat\gamma \,,\, (\id_X \otimes \rho(\tilde a_m))(v \otimes \hat r) \,\big) \ .
\ee
To see this, note that the canonical pairing $\Indtw(Y)' \times \overline\Indtw(Y) \to \Cb$ satisfies
\ba
  \langle T(a)^{1,+}_m \varphi, w \rangle
  &=
  \sum_{k=0}^\infty {{1/2} \choose k} (-1)^k \cdot  
  \langle a_{m + k} \varphi, w \rangle
\nonumber\\
  &\overset{\eqref{eq:am-dual-action}}= - (-1)^{|\varphi||a|}
  \sum_{k=0}^\infty {{1/2} \choose k} (-1)^k \cdot  
  \langle  \varphi,  a_{-m - k} w \rangle
\nonumber\\
  &= - (-1)^{|\varphi||a|}
  \langle \varphi, T(a)^{1,-}_{-m} w \rangle \ ,
\ea
and so
\ba
&P(T(a)^{1,+}_{-m-\frac12} \hat\gamma,a \otimes \hat r)
= \langle T(a)^{1,+}_{-m-\frac12} \hat\gamma , V(1)(a \otimes \hat r) \rangle
\nonumber\\
&=  - (-1)^{|a||\hat\gamma|} \langle  \hat\gamma , T(a)^{1,-}_{m+\frac12} V(1)(a \otimes \hat r) \rangle
= - (-1)^{|a||\hat\gamma|} i  \langle  \hat\gamma , V(1) (\id \otimes T(a)^{1,+}_{m})(a \otimes \hat r) \rangle
\nonumber\\
&= - (-1)^{|a||\hat\gamma|} i  P(\hat\gamma , (\id \otimes T(a)^{1,+}_{m})(a \otimes \hat r) \rangle \ .
\ea

The pairing $P$ will be defined on a pair of PBW-bases for $\Indtw(Y)'$ and $\Ind(R)$. Recall the basis $\{ \alpha^i \}_{i \in \Ic}$ fixed in Section \ref{sec:lie-super}. Pick any ordering on $\Ic$ and define an ordering on $\Ic \times \Zb$ via $(i,m)<(i',m')$ if $m<m'$ or if $m=m'$ and $i<i'$. The set of vectors
\be\label{eq:PBW-basis-IndY*}
  \rho((\widetilde{\alpha^{i_1}})_{-m_1-\frac12}) \cdots
  \rho((\widetilde{\alpha^{i_n}})_{-m_n-\frac12}) \gamma_j \ ,
\ee
where $(i_1,m_1) \ge \cdots \ge (i_n,m_n)$ (and `$>$' of two consecutive odd basis elements), $m_k\ge 0$ for all $k$, and $\gamma_j$ is an element of some basis of $Y^*$, forms a basis of $\Indtw(Y)'$. To see this, one can use that $\rho(\tilde a_m) = a_m + (\text{higher modes})$ and relate \eqref{eq:PBW-basis-IndY*} to a PBW-basis of the induced $\hat\h_\mathrm{tw}$-module $\Indtw(Y)' \cong \Indtw(Y^*)$. Or one can show directly that $\rho$ turns an induced $\hat\h_\mathrm{tw}$-module into an induced $\tilde \h_{\frac12}$-module. 
Similarly, a basis of $\Ind(R)$ is given by 
\be\label{eq:PBW-basis-IndR}
  \rho((\widetilde{\alpha^{i_1}})_{-m_1}) \cdots
  \rho((\widetilde{\alpha^{i_n}})_{-m_n}) r_j \ ,
\ee
where $(i_1,m_1) \ge \cdots \ge (i_n,m_n)$ (or `$>$' for consecutive odd elements), $m_k> 0$ for all $k$, and $r_j$ is an element of some basis of $R$. The case $m_k=0$ is excluded because the zero modes are part of the subalgebra of $\hat\h$ used in the induction.

Fix $v \in X$. The pairing $P$ will be uniquely defined on a pair of basis vectors $(e_I,f_J)$ with $e_I$ of the form \eqref{eq:PBW-basis-IndY*} with $n_I$ modes acting on the ground state and $f_J$ of the form \eqref{eq:PBW-basis-IndR} with $n_J$ modes:
\begin{itemize}
\item $n_I=0$ and $n_J=0$: $P(\gamma_i,v \otimes r_j) = \langle \gamma_i, f(v \otimes r_j) \rangle$, where $f$ is the even linear map $X \otimes R \to Y$ fixed at the start of the argument.
\item $n_I=0$ and $n_J>0$: $P(\gamma_i,v \otimes f_J) = 0$.
\item $n_I>0$: Use condition \eqref{eq:P-pairing-mode-moving} to move all modes from the left argument of $P$ to the right one, then express the resulting product of modes in the right argument of $P$ in the basis \eqref{eq:PBW-basis-IndR}; this leads to one of the previous two cases.
\end{itemize}
This prescription defines the bilinear pairing $P$. However, it is not a priori clear that $P$ satisfies  \eqref{eq:P-pairing-mode-moving} for all choices of $a,m,\hat\gamma,\hat r$. This is what we will establish next.

\medskip

By expressing $\hat r \in \Ind(R)$ as a sum of PBW-basis vectors, the definition of $P$ gives
\ba
  &P\big( \, \rho((\widetilde{\alpha^{i_1}})_{-m_1-\frac12}) \cdots
  \rho((\widetilde{\alpha^{i_n}})_{-m_n-\frac12}) \gamma \, , \, v \otimes \hat r \, \big)
  \nonumber \\
  & \qquad = (-1)^{\text{parity signs}}(-i)^n \cdot 
  P\big( \, \gamma \, , \, \rho((\widetilde{\alpha^{i_n}})_{m_n}) \cdots
  \rho((\widetilde{\alpha^{i_1}})_{m_1}) (v \otimes \hat r) \, \big)
   \label{eq:P-modemove-aux1}
\ea
whenever $\rho((\widetilde{\alpha^{i_1}})_{-m_1-\frac12}) \cdots
  \rho((\widetilde{\alpha^{i_n}})_{-m_n-\frac12})$ is ordered as in the PBW-basis.
From \eqref{eq:tilde-h-bracket} we see that for $k,l \le -\tfrac12$ or for $k,l \ge 0$ we have $[\tilde a_k,\tilde a_l]=0$. Thus \eqref{eq:P-modemove-aux1} continues to hold independent of the ordering, as long as $m_k \ge 0$ for all $k$ (the parity signs in the reordering are the same on both sides). It follows that \eqref{eq:P-pairing-mode-moving} holds for all $a,\hat\gamma,\hat r$ and all $m \ge 0$.

For $m<0$ we proceed by induction on the number $n$ of $\tilde a_{-m-\frac12}$-modes in left argument of $P$:

\medskip\noindent
\nxt If $n=0$, the left hand side of \eqref{eq:P-pairing-mode-moving} is zero as $\rho(\tilde a_{-m-\frac12})\gamma = 0$ for all $m<0$. The right hand side is zero, too, as for any $\hat r$ in the basis \eqref{eq:PBW-basis-IndR}, $\rho(\tilde a_{m}) \hat r$ can be reordered to a basis vector of the form \eqref{eq:PBW-basis-IndR} as well, and so $P$ vanishes by the second point in the list of defining properties.

\medskip\noindent
\nxt Let $n>0$ and suppose \eqref{eq:P-pairing-mode-moving} holds for all $\hat\gamma = e_I$ with $n_I \le n$. Let now $e_I$ be a basis element with $n_I=n$. Pick $b \in \h$ and $k \ge 0$ such that $\tilde b_{-k-\frac12} e_I$ is again an element of the basis \eqref{eq:PBW-basis-IndY*}. Let $a \in \h$ and $m<0$ be arbitrary. In the following calculation, all $\rho$'s are omitted and on the right hand side $f_J$ is written instead of $v \otimes f_J$, and $\tilde a_m$ instead of $\id \otimes \rho(\tilde a_m)$. We have:
\ba
&P( \tilde a_{-m-\frac12} \tilde b_{-k-\frac12} e_I ,  f_J )
=
(-1)^{|a||b|}
P( \tilde b_{-k-\frac12} \tilde a_{-m-\frac12}  e_I , f_J )
+
P( [\tilde a_{-m-\frac12},\tilde b_{-k-\frac12}] e_I , f_J )
\nonumber\\
&\overset{(1)}=
(-1)^{|a||b|}(-1)^{|b|(|a|+|e_I|)}
(-i)\,P(  \tilde a_{-m-\frac12}  e_I , \tilde b_{k} f_J )
+
P( [\tilde a_{-m-\frac12},\tilde b_{-k-\frac12}]  e_I ,  f_J )
\nonumber\\
&\overset{(2)}=
-(-1)^{(|a|+|b|)|e_I|}
P(  e_I , \tilde a_{m}  \tilde b_{k} f_J )
-(-1)^{|a||b|} P(  e_I , [\tilde b_{k},\tilde a_{m}] f_J )
\nonumber\\
&\overset{(3)}=
-(-1)^{(|a|+|b|)|e_I|}(-1)^{|a||b|}
P(  e_I , \tilde b_{k} \tilde a_{m}   f_J )
\nonumber\\
&\overset{(4)}=
-(-1)^{|a|(|e_I|+|b|)}(-i)^{-1}
P(  \tilde b_{-k-\frac12} e_I , \tilde a_{m}   f_J ) \ .
\label{eq:surj-proof-aux1}
\ea
Equalities 1 and 4 are the already established statement that \eqref{eq:P-pairing-mode-moving} holds for all $m \ge 0$. Equality 2 is, firstly, the induction assumption, and secondly, the observation that $[\tilde a_{-m-\frac12},\tilde b_{-k-\frac12}] = [\tilde a_m,\tilde b_k]$, together with the fact that $K$ acts as 1 on all modules. To see that the parity signs work out in equality 3 note that $[\tilde b_{k},\tilde a_{m}]$ can only be non-zero for $|a|=|b|$.

\void{
Details: By definition
\be
  [\tilde a_m,\tilde b_n] = (a,b) \cdot \big\{ m \delta_{m+n,0} - \big(m+\tfrac12\big) \delta_{m+n+1,0}  \big\} \cdot K \ ,
\ee
and so
\ba
  [\tilde a_{-m-\frac12},\tilde b_{-k-\frac12}] 
  &= 
  (a,b) \cdot \big\{ (-m-\tfrac12) \delta_{-m-\frac12-k-\frac12,0} - \big(-m-\tfrac12+\tfrac12\big) \delta_{-m-\frac12-k-\frac12+1,0} ]  \big\} \cdot K
  \nonumber\\
  &= 
  (a,b) \cdot \big\{ - (m+\tfrac12) \delta_{-m-k-1,0} + m \delta_{-m-k,0}   \big\} \cdot K
  \nonumber\\
  &= 
  (a,b) \cdot \big\{  m \delta_{m+k,0} - (m+\tfrac12) \delta_{m+k+1,0}    \big\} \cdot K
  \nonumber\\
  &= [\tilde a_m,\tilde b_k] \ .
\ea
}

\medskip

This completes step 1. Next we turn to step 2, the extension of $V(1)$ to $V(x)$. It will improve readability if we abuse notation and write $[L_m,V(x)]$ to mean $L_m \circ V(x) - V(x) \circ (\id \otimes L_m)$. We start by reducing the extension problem to the condition 
\be\label{eq:L0L1V(1)-comm-condition}
[L_0-L_{-1},V(1)] = \tfrac{\sdim(\h)}{16} \, V(1) \ .
\ee 
Namely, define
\be \label{eq:V(x)-from-V(1)}
  V(x) = x^{-\frac{\sdim(\h)}{16}} \, \exp\!\big( \ln(x) L_0 \big) \circ V(1) \circ \big\{ \id \otimes \exp\!\big(-\ln(x) L_0 \big) \big\} \ .
\ee 
By local finiteness, the exponentials converge and one obtains a well-defined linear map $X \otimes \Ind(R) \to \overline\Indtw(Y)$. The $x$-derivative of a matrix element $\langle \hat\gamma , V(x)(v \otimes \hat r) \rangle$ is (omitting the $\langle \hat\gamma , \dots (v \otimes \hat r) \rangle$ from the notation):
\be
\tfrac{d}{dx} V(x) = -\tfrac{\sdim(\h)}{16} x^{-1} V(x) + x^{-1} [L_0 , V(x)] \ .
\ee
To compute $[L_{-1},V(x)]$ we use
$e^{-\ln(x)L_0} L_{-1} e^{\ln(x)L_0} = x^{-1} L_{-1}$ to find
\ba
[L_{-1},V(x)]
&= 
x^{-1} x^{-\frac{\sdim(\h)}{16}} \, e^{ \ln(x) L_0 } \circ  [L_{-1},  V(1)] \circ (\id \otimes e^{-\ln(x) L_0 })
\nonumber\\
&= 
x^{-1} [L_{0},  V(x)]
-
\tfrac{\sdim(\h)}{16} \,x^{-1} V(x) \ ,
\ea
where in the last step condition \eqref{eq:L0L1V(1)-comm-condition} was substituted in the form $[L_{-1},V(1)] =[L_0,V(1)] - \frac{\sdim(\h)}{16} V(1)$. This shows that inside matrix elements we have $\tfrac{d}{dx} V(x) = [L_{-1},V(x)]$ as required, so that condition (ii) in Definition \ref{def:vertex-op} holds. The explicit form of the $x$-dependence in \eqref{eq:V(x)-from-V(1)} also shows that the smoothness requirement in (i) is satisfied.

\medskip

It remains to prove \eqref{eq:L0L1V(1)-comm-condition}. In terms of the pairing $P$, we need to verify that for all $\hat\gamma \in  \Indtw(Y)'$, $v \in X$ and $\hat r \in \Ind(R)$,
\be \label{eq:P-L0L1-rule}
  P\big(\, (L_0-L_1) \hat\gamma \,,\, v \otimes \hat r \,\big)
  =
  P\big(\, \hat\gamma \,,\, v \otimes (L_0-L_{-1})  \hat r \,\big)
  +
  \tfrac{\sdim(\h)}{16}  \cdot
  P\big(\, \hat\gamma \,,\, v \otimes \hat r \,\big) \ .  
\ee
To start with, suppose that $\hat\gamma = \gamma \in Y^*$ is a ground state. Then by \eqref{eq:Virasoro-modes-twisted} the left hand side of \eqref{eq:P-L0L1-rule} is equal to $\tfrac{\sdim(\h)}{16} P(\gamma,v \otimes \hat r)$. Thus we need to verify that $P(\gamma,v \otimes (L_0-L_{-1})\hat r) = 0$. 
If also $\hat r = r \in R$ is a ground state, the same calculation as in \eqref{eq:x-dep-gs-case_c} (with $x=1$) shows that
\be
P(\gamma , v \otimes L_{-1}  r ) 
\overset{\eqref{eq:x-dep-gs-case_c}}= 
\tfrac12 \,P(\gamma , \Omega^{(22)} v \otimes  r )
\overset{\eqref{eq:virasoro-untwisted}}= P(\gamma ,  v \otimes  L_0 r ) \ ,
\ee
as required. For $\hat r$ arbitrary, we will need the identity
\be \label{eq:L0L1-T-commute}
  \big[\,L_0-L_{-1} \,,\, \rho(\tilde a_m) \,\big] = -(m+\tfrac12)\, \rho(\tilde a_m) + m \, \rho(\tilde a_{m-1}) \ .
\ee
To see this, first note that $[L_m,a_k] = - k a_{m+k}$, as is immediate from \eqref{eq:virasoro-untwisted} and \eqref{eq:Virasoro-modes-twisted}. We have, for all $a\in\h$ and all $m \in \Zb$ or $m \in \Zb+\frac12$,
\ba
&[L_0-L_{-1},T(a)^{1,+}_{m}] + (m+\tfrac12) T(a)^{1,+}_{m} - m T(a)^{1,+}_{m-1}
\nonumber\\
&=
\sum_{k=0}^\infty 
{{1/2} \choose k} (-1)^k \big\{
-(m+k) (a_{m+k}-a_{m+k-1}) + m (a_{m+k}-a_{m+k-1}) + \tfrac12 a_{m+k}
\big\}
\nonumber\\
&= 
\sum_{k=0}^\infty 
{{1/2} \choose k} (-1)^k k a_{m+k-1}
+ 
\sum_{k=0}^\infty 
{{1/2} \choose k} (-1)^k (\tfrac12-k) a_{m+k}
\nonumber\\
&= 
\sum_{k=0}^\infty \Bigg\{ - {{1/2} \choose k+1}(k+1) + {{1/2} \choose k} (\tfrac12-k) \Bigg\} 
(-1)^k a_{m+k} = 0 \ .
\ea
An analogous calculation gives the adjoint relation
\be \label{eq:L0L1-T-commute2}
  \big[\,\rho(\tilde a_{-m-\frac12}) \,,\, L_0-L_1 \,\big] = -(m+\tfrac12)\, \rho(\tilde a_{-m-\frac12}) + m \, \rho(\tilde a_{-(m-1)-\frac12} ) \ .
\ee

\begin{remark}
In terms of fields, one would write $[L_m,a(z)] = (m+1)z^m a(z) + z^{m+1} \tfrac{\partial}{\partial z} a(z)$, so that $[L_0-L_{-1},a(z)] = a(z) + (z-1)\tfrac{\partial}{\partial z} a(z)$. It is then easy to evaluate $[L_0-L_{-1},T(a)^{1,+}_m] = \oint   z^m(z-1)^{\frac12} [L_0-L_{-1} , a(z) ] \tfrac{dz}{2 \pi i} $.
\end{remark}

Since we already know \eqref{eq:P-L0L1-rule} on ground states, we may assume that $\hat r = \rho(\tilde b_{-m}) f_J$ for some $b \in \h$, $m >0$ such that we obtain an element in the basis \eqref{eq:PBW-basis-IndR}. Then 
\be
  (L_0 - L_{-1})\rho(\tilde b_{-m}) f_J
  =
  [ L_0 - L_{-1} , \rho(\tilde b_{-m})]  f_J
  +
  \rho(\tilde b_{-m}) (L_0 - L_{-1}) f_J \ .
\ee
When inserting this into $P(\gamma,v \otimes - )$ and using \eqref{eq:L0L1-T-commute}, both summands on the right hand side give zero separately, since by \eqref{eq:P-pairing-mode-moving} for a ground state $\gamma$ and for all $a \in \h$, $k > 0$, we have $P(\gamma,(\id \otimes \rho(\tilde a_{-k})(v \otimes \hat r)) = 0$.

The general case of \eqref{eq:P-L0L1-rule} now follows by induction on the number of $\tilde a_m$-modes in the left argument of $P$ via the equalities (the same abbreviations as in \eqref{eq:surj-proof-aux1} apply),
\ba
&P\big(\, (L_0-L_1) \tilde a_{-m-\frac12} \hat\gamma \,,\, \hat r \,\big)
\nonumber\\
&= 
P\big(\, [L_0-L_1 , \tilde a_{-m-\frac12}] \hat\gamma \,,\, \hat r \,\big)
+
P\big(\, \tilde a_{-m-\frac12} (L_0-L_1)  \hat\gamma \,,\, \hat r \,\big)
\nonumber\\
&\overset{(2)}=
(-1)^{|a||\hat\gamma|}\,(-i)\Big( 
P\big(\, \hat\gamma \,,\, [\tilde a_{m} , L_0-L_{-1} ] \hat r \,\big)
+
P\big(\, (L_0-L_1)  \hat\gamma \,,\, \tilde a_{m} \hat r \,\big) 
\Big)
\nonumber\\
&\overset{(3)}= 
(-1)^{|a||\hat\gamma|}\,(-i)\Big( 
P\big(\, \hat\gamma \,,\, [\tilde a_{m} , L_0-L_{-1} ] \hat r \,\big)
+
P\big(\,  \hat\gamma \,,\, (L_0-L_{-1}) \tilde a_{m} \hat r \,\big)
  + 
  \tfrac{\sdim(\h)}{16}  \,
  P\big(\, \hat\gamma \,,\, \tilde a_{m} \hat r \,\big) 
  \Big)
\nonumber\\
&=
(-1)^{|a||\hat\gamma|}\,(-i)\Big( 
P\big(\, \hat\gamma \,,\, \tilde a_{m} ( L_0-L_{-1})  \hat r \,\big)
  + 
  \tfrac{\sdim(\h)}{16}  \,
  P\big(\, \hat\gamma \,,\, \tilde a_{m} \hat r \,\big) 
  \Big)
\nonumber\\
&= 
P\big(\, \tilde a_{-m-\frac12}  \hat\gamma \,,\, ( L_0-L_{-1})  \hat r \,\big)
  + 
  \tfrac{\sdim(\h)}{16}  \,
  P\big(\, \tilde a_{-m-\frac12} \hat\gamma \,,\, \hat r \,\big) \ .  
\ea
In step 2, expression \eqref{eq:L0L1-T-commute2} is used to convert the commutator involving $L$'s into $\tilde a$'s which are then moved to the other argument using \eqref{eq:P-pairing-mode-moving} and converted back into an $L$-commutator via \eqref{eq:L0L1-T-commute}. In step 3 the induction assumption is used to move $L_0-L_1$ to the other argument via \eqref{eq:P-L0L1-rule}.
The completes the proof of the commutator \eqref{eq:L0L1V(1)-comm-condition}.

\medskip

We have now proved Theorem \ref{thm:G-bijective} in case c). Case d) works along the same lines and will not be treated explicitly.

\section{Computation of four-point blocks in Table \ref{tab:4pt-blocks}}\label{app:4pt-calc}

\subsection{Cases 000 and 001 from normal ordered products}

\subsubsection*{Untwisted vertex operators}

If all representations are taken from the untwisted sector $\Cc_0$, it is straightforward to obtain the Coulomb gas expression for products of several vertex operators -- or rather a generalisation thereof to the non-semisimple and super setting used here. Let us briefly go through the calculation.

\medskip

Recall the notation $\Omega^{(is)}_m$ and $E^{(is)}_{\pm,0}$ introduced in \eqref{eq:Omega-ij-v1} and \eqref{eq:E-ij-v1}. The exchange relation \eqref{eq:h-Em-comm} implies that for $m>0$
\be\label{eq:vo-untwisted-product-aux1}
  \Omega^{(is)}_m  \, E^{(js)}_-(y) 
  = E^{(js)}_-(y) \, \big( \, \Omega^{(is)}_m + y^m \cdot \Omega^{(ij)} \, \big)
\ee
To compute the product of several normal ordered exponentials we first repeat the standard mode commutation exercise to obtain, for $x>y>0$,
\ba
  E^{(is)}_+(x) \, E^{(js)}_-(y) &\overset{\rule{1.7em}{0pt}}= \textstyle
  \exp\!\Big( \sum_{m>0}  \tfrac{x^{-m}}{-m} \Omega_m^{(is)} \Big) \, E^{(js)}_-(y)
  \nonumber \\
  &\overset{\eqref{eq:vo-untwisted-product-aux1}}= \textstyle
  E^{(js)}_-(y) \, \exp\!\Big(\sum_{m>0} \tfrac{-1}{m} (\tfrac{y}{x})^m \cdot  \Omega^{(ij)} + \sum_{m>0}  \tfrac{x^{-m}}{-m} \Omega_m^{(is)} \Big) 
  \nonumber \\
  &\overset{\rule{1.7em}{0pt}}= \exp\!\big( \ln(1{-}\tfrac yx) \cdot  \Omega^{(ij)} \big) \, E^{(js)}_-(y) \, E^{(is)}_+(x)  
  \ .
 \label{eq:E+isE-js-com}
\ea
The condition $x>y$ is needed for convergence. 

Now pick $R_0,\dots,R_s,P_1,\dots,P_{s-2} \in \Rep\lf(\h)$. Let $f_i : R_{i} \otimes P_i \to P_{i-1}$ with $1 \le i \le s{-}1$ be $\h$-intertwiners, where it is understood that $P_0 \equiv R_0$ and $P_{s-1} \equiv R_s$.
The product of vertex operators
\be
  \bL{\blockA{\scriptstyle R_0}{\scriptstyle R_1}{\scriptstyle \hspace{-.5em}f_1;x_1}\blockC{\scriptstyle P_1}{\scriptstyle R_2}{\scriptstyle \hspace{.5em} f_{2};x_2}\blockB{\scriptstyle P_2} \raisebox{1em} 
  {~\dots~} 
\blockA{\scriptstyle P_{s-2}}{\scriptstyle R_{s-1}}{\scriptstyle f_{s-1};x_{s-1}}\blockB{\scriptstyle R_s}}  \quad ,
\ee
with $x_1>x_2>\cdots>x_{s-1}>0$, is given by
\ba
&
\Pgs \circ V(f_1;x_1) 
\circ \big[\id_{R_1} \otimes V(f_2;x_2) \big]
\circ \cdots
\circ \big[\id_{R_1 \otimes \cdots \otimes R_{s-2}} \otimes V_{f_{s-1}}(x_{s-1})\big]
\nonumber\\
&
= ~
f_1
\circ (\id_{R_1} \otimes f_2)
\circ \cdots
\circ (\id_{R_1 \otimes \cdots \otimes R_{s-2}} \otimes f_{s-1})
\nonumber\\
& \hspace{5em}
\circ
\exp\hspace{-2.5pt}\Big( \sum_{1 \le i<j \le s} \hspace{-.6em} \ln(x_i-x_j) \cdot \Omega^{(ij)} ~+~  \ln(4)  \hspace{-.6em} \sum_{1 \le i<j \le s} \hspace{-.6em} \Omega^{(ij)} \Big) \ ,
\label{eq:untw-n-point-block}
\ea
where it is understood that $x_s\equiv 0$. This product of vertex operators, restricted to ground states, is an even linear map $R_1 \otimes R_2 \otimes \cdots \otimes R_s \to R_0$. 

To arrive at \eqref{eq:untw-n-point-block}, first move all the $E_0^{(12)}(x_i)$, which are endomorphisms of $R_i \otimes \Ind(P_i)$, all the way to the right. By Lemma \ref{lem:Qf-exists}, moving zero modes through $Q_f$ results in the action of the zero mode on the tensor product, this results in $\exp(  \sum_{j=i+1}^s \ln(x_i) \, \Omega^{(ij)})$ for each $E_0$. Moving all $E_+$ through all $E_-$ results in $\exp( \sum_{1 \le i<j \le s } \ln(1-\frac{x_j}{x_i}) \, \Omega^{(ij)})$. Composing these two gives \eqref{eq:untw-n-point-block}. The $\ln(4)$ terms arise due to the normalisation of the vertex operators in \eqref{eq:V(f)-on-ground-states}.

Note that choosing $s=3$ and $x_1=1$, $x_2=x$  produces the four-point block quoted in case 000 of Table \ref{tab:4pt-blocks}.

\begin{example} Let us see how the general expression \eqref{eq:untw-n-point-block} specialises in the two examples treated in Section \ref{sec:examples-VO}, using the notation introduced there.

\smallskip\noindent
(i) Single free boson, $\h = \Cb^{1|0}$: Choose the one-dimensional $\h$-representations $P_i = \Cb_{p_i}$ for some $p_i \in \Cb$. Then $\exp( \ln(x_i-x_j) \cdot \Omega^{(ij)}) = (x_i - x_j)^{p_i p_j}$ and one finds the standard Coulomb gas expression (the $\ln(4)$'s have been absorbed into the constant and it is understood that $x_s=0$)
\be
\eqref{eq:untw-n-point-block} ~=~
\text{(const)} \, \delta_{p_0,p_1+\cdots+p_{s-1}} \prod_{1 \le i<j \le s} (x_i-x_j)^{p_i p_j} \ .
\ee
(ii) Single pair of symplectic fermions, $\h = \Cb^{0|2}$: For $f : A \ast P \to D$ and $g: B \ast C \to P$ we get
\ba
&\Pgs \circ V(f;x) 
\circ \big( \id_{A} \otimes V(g;y) \big)|_{A \otimes B \otimes C}
\nonumber \\
& =
f \circ (\id_{A} \otimes g) \circ 
\exp\!\big( 
\ln(4(x-y)) \cdot \Omega^{(12)} 
+ \ln(4x) \cdot \Omega^{(13)} 
+ \ln(4y) \cdot \Omega^{(23)} 
\big)
\label{eq:example-4pt-sympferm}
\ea 
Since any product of three modes in $\h$ is zero, monomials built out of four or more of $\Omega^{(12)}$, $\Omega^{(13)}$, $\Omega^{(23)}$ are zero. Thus the above exponential terminates with the term of power three. Such products of symplectic fermion vertex operators were first computed in \cite{Gaberdiel:2006pp}.
\end{example}

\subsubsection*{Twisted vertex operators}

Let $R,S \in\Rep\lf(\h)$ and $X,Y,Z\in\svect$ and fix even linear maps of super-vector spaces $f : F(R) \otimes Z \to X$ and $g : F(S) \otimes Y \to Z$. The computation for
\be
     \bL{\cbB XRSYZ{f;x}{\rule{0pt}{.68em}g;y}} 
   ~=~    
  \Pgs \circ V(f;x) \circ (\id_R \otimes V(g;y))  \big|_{R \otimes S \otimes Y} \ ,
\ee
where $x>y>0$, is similar to the one leading to \eqref{eq:untw-n-point-block}. Recall the notation $\tilde E^{(is)}$
from \eqref{eq:tilde-E^is_m-def}. Because 
$\sum_{m \in \Zb_{\ge 0}+\frac12} \frac{-1}m (\frac yx)^m = \ln(1{-}\sqrt{y/x}) -\ln(1{+}\sqrt{y/x})$,
the identity corresponding to \eqref{eq:E+isE-js-com} now reads
\be \label{eq:E+isE-js-com-twist}
  \tilde E^{(is)}_+(x)\, \tilde E^{(js)}_-(y) 
  = \exp\!\big( \ln \tfrac{\sqrt{x}-\sqrt{y}}{\sqrt{x}+\sqrt{y}} \cdot \Omega^{(ij)}\big) \, E^{(js)}_-(y) \, E^{(is)}_+(x)  
  \ .
\ee 
This results in, for $x>y>0$,
\ba
  &\Pgs \circ V(f;x) \circ (\id_R \otimes V(g;y))  \big|_{R \otimes S \otimes Y}
\nonumber\\
  &=~
 f \circ (\id_R \otimes g) \circ 
  \exp\!\Big(
  - \tfrac12 \ln(x) \cdot \Omega^{(11)}
  - \tfrac12 \ln(y)  \cdot \Omega^{(22)}
  +\ln\tfrac{\sqrt{x}-\sqrt{y}}{\sqrt{x}+\sqrt{y}} \cdot \Omega^{(12)} 
\nonumber\\
 & \hspace{23em} - \ln(4) \cdot (\Omega^{(11)} + \Omega^{(22)})
    \Big) \ .
\label{eq:4pt-RNNR}
\ea
Setting $x=1$ gives case 001 in Table \ref{tab:4pt-blocks}.

\subsection{Cases 010 and 100 from translation invariance}\label{app:010-100-transinv}

There are two further four-point blocks with two twisted representations which can be found from \eqref{eq:4pt-RNNR} by translation invariance. 
Write 
\be
  \nu(x,y,z) = \Pgs \circ V(f;x-z) \circ (\id_R \otimes V(g;y-z))  \big|_{R \otimes S \otimes Y} \ ; 
\ee  
this is a four-point block with $R,S,Y$ inserted at $x,y,z$, respectively. Omitting the precise morphism-dependence gives
\be \label{eq:nu-xyz-twotwist}
\nu(x,y,z) = \text{(morph.)} \circ \exp\!\Big(
  - \tfrac12 \ln(x{-}z) \cdot \Omega^{(11)}
  - \tfrac12 \ln(y{-}z)  \cdot \Omega^{(22)}
  +\ln\tfrac{\sqrt{x-z}-\sqrt{y-z}}{\sqrt{x-z}+\sqrt{y-z}} \cdot \Omega^{(12)} \Big) \ ,
\ee
where the exponential is an endomorphism of $R \otimes S \otimes Y$.
The $x$-dependence of the four-point block will be of the form $\nu(1,0,x)$ in case 010 and of the form $\nu(x,0,1)$ in case 100, respectively. To determine the morphism in front of the coordinate dependent part, we compute the $x$-expansion of the product of the two vertex operators explicitly to first non-trivial order.

\subsubsection*{Case 010}

The first two orders in the $x$-expansion are
\ba
\nu(1,0,x)
&= \text{(morph.)} \circ \exp\!\Big(
  - \tfrac12 \ln(1{-}x) \cdot \Omega^{(11)}
  - \tfrac12 \ln(x)  \cdot \Omega^{(22)} 
  \nonumber\\ & \hspace{10em} + (2N+1) \pi i \cdot \Omega^{(22)}
  +\ln\tfrac{\sqrt{1-x}-i \nu \sqrt{x}}{\sqrt{1-x}+ i \nu \sqrt{x}} \cdot \Omega^{(12)} \Big) 
  \nonumber\\ &
= \text{(morph.$'$)} \circ \exp\!\big(\! - \tfrac12 \ln(x)  \cdot \Omega^{(22)} \big) 
\circ \big( \id_{R \otimes S \otimes Y} - 2 i \nu \Omega^{(12)} x^{\frac12} + O(x) \big) \ .
\label{eq:calc-010-aux1}
\ea
In the first equality the ambiguity in choosing a branch of the logarithm and the square root is made explicit via the constants $N \in \Zb$ and $\nu \in \{ \pm 1\}$, respectively. In the second line, $N$ was absorbed into a redefinition of the constant morphism.

We want to use the expansion of the above functional form to conclude the full $x$-dependence of the four-point block in case 010 from its expansion to first non-trivial order. Pick $f: R \otimes Z \to X$ and $g : Y \otimes S \to Z$. We want to compute
\be \label{eq:two-twist-aux1}
\bL{\cbB XRYSZ{f;1}{\rule{0pt}{.68em}g;x}}
= \Pgs \circ V(f;1) \circ \big( \id_R \otimes V(g;x) \big) \ .
 \ee
To obtain the first non-trivial order, we will insert a sum over intermediate states for $\Ind(Z)$, namely
\be \label{eq:id-expansion-twist}
  \id_{\Ind(Z)} = \Pgs  + \sum_{m>0} \sum_{i\in\Ic} \frac1m \beta^i_{-m} \,\Pgs\, \alpha^i_m + \dots \quad ,
\ee  
where the sums run over positive elements in $\Zb + \frac12$ and the omission stands for terms with two or more $\alpha^i$ and $\beta^i$. Using this, we can evaluate
\ba
\eqref{eq:two-twist-aux1} &= 
B_0 + B_1 + \dots
\nonumber\\
&\overset{(*)}{=} 
f \circ (\id_R \otimes g) \circ \exp\!\big(- \ln(4) \cdot (\Omega^{(11)}+\Omega^{(33)}) \big) 
\nonumber\\
&\hspace{4em} 
 \circ \exp\!\big(-\tfrac12 \ln(x)\cdot \Omega^{(33)}\big) 
\circ (\id_{R\otimes Y \otimes S} - 2i x^{\frac12} \Omega^{(13)} + O(x))
~,
 \label{eq:RNRN-expansion}
\ea
where $B_i$ is the $i$'th order in the $x$-expansion and the equality (*) will explained in a moment. 
First note that comparing the small-$x$ expansion of $\nu(1,0,x)$ in \eqref{eq:calc-010-aux1} to \eqref{eq:RNRN-expansion} fixes the full expression for \eqref{eq:two-twist-aux1} to be the one given in Table \ref{tab:4pt-blocks} (that $\Omega^{(33)}$ and $\Omega^{(13)}$ appear instead of $\Omega^{(22)}$ and $\Omega^{(12)}$ is due to the reordering of tensor factors).

To obtain the $B_i$, insert \eqref{eq:id-expansion-twist} between $V(f;1)$ and $V(g;x)$. This gives
$B_0 = \Pgs \circ V(f;1) \circ (\id_R \otimes  \Pgs) \circ (\id_R \otimes V(g;x))$ and 
\ba
B_1 &=  2 \sum_{j\in\Ic}  \Pgs \circ V(f;1) \circ (\beta^j_{-\frac12})^{(2)} \circ (\id_R \otimes  \Pgs ) \circ (\alpha^j_{\frac12})^{(2)} \circ (\id_R \otimes V(g;x))
\nonumber\\
&\overset{(**)}{=} 
 - 2i\,x^{\frac12}  \sum_{j\in\Ic}  \Pgs \circ V(f;1) \circ (\beta^j)^{(1)} \circ (\id_R \otimes  \Pgs ) \circ (\id_R \otimes V(g;x)) \circ (\alpha^j_{0})^{(3)} 
\nonumber\\
&=
- 2i\,x^{\frac12} \cdot B_0 \circ \Omega^{(13)} \ .
\ea
To obtain (**), one uses 
case c) in part (iii) of Definition \ref{def:vertex-op} for $m=0$, together with 
\be
  T(\alpha^j)^{x,-}_{\frac12} = \alpha^j_{\frac12} + (\text{neg.\ modes})
  \quad , \quad
  T(\alpha^j)^{x,+}_{0} = \alpha^j_{0} + (\text{pos.\ modes}) \ ,
\ee  
as well as case b), which gives $\Pgs \circ V(f;1) \circ (\beta^j_{-\frac12})^{(2)} = - \Pgs \circ V(f;1) \circ (\beta^j)^{(1)}$. Finally, to get (*) in \eqref{eq:RNRN-expansion}, substitute the expressions in \eqref{eq:V(f)-on-ground-states} for the action of vertex operators on ground states.

\subsubsection*{Case 100}

The first two orders in the $x$-expansion are
\ba
\nu(x,0,1)
&= \text{(morph.)} \circ \exp\!\Big(
  - \tfrac12 \ln(1{-}x) \cdot \Omega^{(11)}
  +\ln\tfrac{1 - \nu \sqrt{1-x}}{1 + \nu \sqrt{1-x}} \cdot \Omega^{(12)} \Big) 
  \nonumber\\ &
= \text{(morph.)} \circ \exp\!\big( \nu  \ln\tfrac{x}{4} \cdot \Omega^{(12)} \big) 
\circ \big( \id_{R \otimes S \otimes Y} + \tfrac{x}2 \big(\Omega^{(11)} + \nu \Omega^{(12)}\big) + O(x^2) \big) \ .
\label{eq:calc-100-aux1}
\ea
In the first equality the ambiguity in the square root branch cut is captured by $\nu \in \{ \pm 1\}$ while the choice for the logarithm has been absorbed into the morphism.

Pick an even linear map $f : Y \otimes P \to X$ and an $\h$-intertwiner $g: R \otimes S \to P$. In the untwisted sector, the sum over intermediate states for $\id_{\Ind(P)}$ is of the same form as \eqref{eq:id-expansion-twist} except that the sum now runs over $m \in \Zb_{>0}$. The expansion of the four-point block is
\ba
& \bL{\cbB XYRSP{f;1}{\rule{0pt}{.68em}g;x}} 
~=~ \Pgs \circ V(f;1) \circ (\id_Y \otimes V(g;x)) =  B_0 +  B_1 + \dots
\nonumber \\
&\overset{(*)}{=} f \circ (\id_Y \otimes g) \circ \exp\!\big(\ln(x)\cdot \Omega^{(23)} - \ln(4) \cdot (\Omega^{(22)} + \Omega^{(23)}+ \Omega^{(33)}) \big) 
\nonumber \\
&\hspace{7.1em}
\circ \big( \id_{Y \otimes R \otimes S} + \tfrac12 x \cdot ( \Omega^{(22)} + \Omega^{(23)}) + O(x^2) \big)
\label{eq:RRNN-expansion}
\ea
Before looking at the details for equality (*) let us compare the above expansion to \eqref{eq:calc-100-aux1}. Firstly, in \eqref{eq:calc-100-aux1} we have to replace $\Omega^{(11)}$ by $\Omega^{(22)}$ and $\Omega^{(12)}$ by $\Omega^{(23)}$. Next, we see that $\nu=+1$ and that the morphism in \eqref{eq:calc-100-aux1} has to be chose as $f \circ (\id_Y \otimes g) \circ \exp\!\big(\!- \ln(4) \cdot (\Omega^{(22)} + \Omega^{(33)}) \big)$. This then produces the four-point block given in Table \ref{tab:4pt-blocks}.

Equality (*) is found by computing $B_0$ and $B_1$ as
\ba
  B_0 &~\,=\,~ \Pgs \circ V(f;1) \circ (\id_Y \otimes  \Pgs) \circ \big(\id_Y \otimes V(g;x)\big)
  \nonumber \\
  &\overset{\eqref{eq:V(f)-on-ground-states}}=
f \circ \exp\!\big(-\ln(4) \cdot \Omega^{(22)}\big) \circ
(\id_Y \otimes  g) \circ \exp\!\big((\ln(x)+\ln(4))\cdot \Omega^{(23)}\big) 
  \nonumber \\
&~\,=\,~ f \circ (\id_Y \otimes  g) \circ \exp\!\big( \ln(x) \cdot \Omega^{(23)} - \ln(4) \cdot (\Omega^{(22)} + \Omega^{(23)}+ \Omega^{(33)}) \big) \ .
\ea
and 
\ba
B_1 &=  \sum_{j\in\Ic}  \Pgs \circ V(f;1) \circ (\beta^j_{-1})^{(2)} \circ (\id_Y \otimes  \Pgs ) \circ (\alpha^j_{1})^{(2)} \circ (\id_Y \otimes V(g;x))
\nonumber\\
&\overset{(**)}= 
 \tfrac12 x \sum_{j\in\Ic}  \Pgs \circ V(f;1) \circ (\beta^j_{0})^{(2)} \circ (\id_Y \otimes  \Pgs ) \circ (\id_Y \otimes V(g;x)) \circ (\alpha^j_{0})^{(2)} 
\nonumber\\
&=  \tfrac12 x  \cdot B_0 \circ (\Omega^{(22)}+\Omega^{(23)}) \ .
\ea
For (**), firstly, one uses case c) in part (iii) of Definition \ref{def:vertex-op} together with 
\be
  T(\beta^j)^{1,+}_{-1} = \beta^j_{-1} - \tfrac12  \beta^j_{0} + (\text{pos.\ modes}) \ ;
\ee  
since $T(\beta^j)^{1,+}_{-1}$ vanishes on $\Pgs \circ V(f;1)$, we can replace $\beta^j_{-1}$ with $\tfrac12 \beta^j_{0}$. Secondly, case a) in part (iii) states that moving $\alpha^j_{1}$ past $V(g;x)$ gives $x \, \alpha^j \otimes \id$ plus a term that vanishes on ground states. 

\subsection{Contour integrals of currents} \label{sec:contour-int}

In this and the next two sections the language of currents and conformal blocks is used to derive differential equations for the four-point blocks of types 011 and 111 in Table \ref{tab:4pt-blocks}. This is an adaptation of the methods in \cite{Knizhnik:1984nr,Dixon:1986qv,Zamolodchikov:1987ae} to the present setting.

\medskip

Let us agree that $a(z)$, $b(z)$, etc., denotes an insertion of a current for $a,b,\dots \in \h$ at position $z$. 
By $\phi(x)$ we shall mean a state from a representation in $\Rep_{\flat,1}(\hat\h)$ and
by $\sigma(x)$ a state from a representation in $\Rep_{\flat,1}(\hat\h_\mathrm{tw})$. If we do not want to specify if the state lies in a twisted or untwisted representation, we write $\gamma(x)$. 

\medskip

By definition
\be
  (a_m \gamma)(x) = \oint_{|z-x|\ll 1} \hspace{-2em} (z-x)^m \, a(z)\gamma(x) \tfrac{dz}{2 \pi i} \ ,
  \label{eq:contour-1}
\ee
where $m \in \Zb$ if $\gamma$ lies in an untwisted representation and $m \in \Zb +\frac12$ for twisted representations. To avoid sign ambiguities we must decide on a convention for branch cuts of square roots. Define $s(z)$ to be the square root of $z$ with branch cut placed along the negative imaginary axis, i.e.\
\be
  s(r e^{i \theta}) = \sqrt{r} e^{i \theta/2} 
  \quad \text{for} \quad \theta \in [-\tfrac\pi2,\tfrac{3\pi}2) \ .
\ee
Analogously, we fix the convention that close to an insertion of a state from a twisted representation, the branch cut for currents is along the negative imaginary axis, shifted to the insertion point, see Figure \ref{fig:branch-cut-twisted}. In other words, close to $x$ the combination $s(z-x) \cdot a(z)\,\sigma(x)$ is single valued. In particular, in the twisted case the mode integral \eqref{eq:contour-1} stands for 
\be
  (a_{n+\frac12} \sigma)(x) = \oint_{|z-x|\ll 1}  \hspace{-2em} (z-x)^n s(z-x) \, a(z)\gamma(x) \tfrac{dz}{2 \pi i} 
  ~~ , \quad \text{where} \quad n \in \Zb \ ,
  \label{eq:contour-2}
\ee
so that the integrand is single valued close enough to $x$. 

\begin{figure}[bt]
\begin{center}
\raisebox{2.4em}{a)}~
\ipic{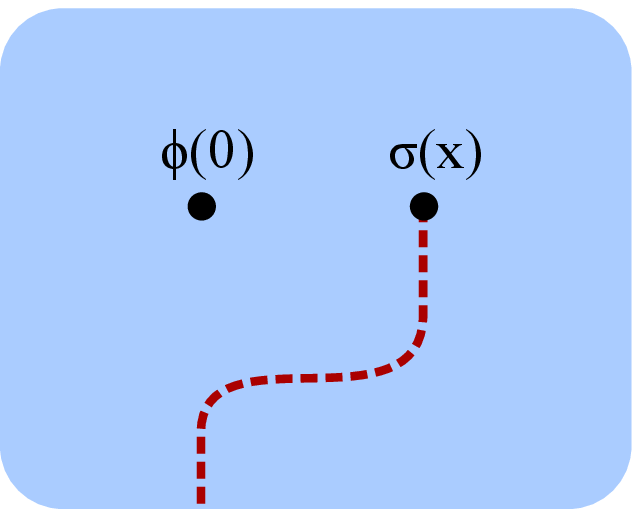}
\hspace{3em}
\raisebox{2.4em}{b)}~
\ipic{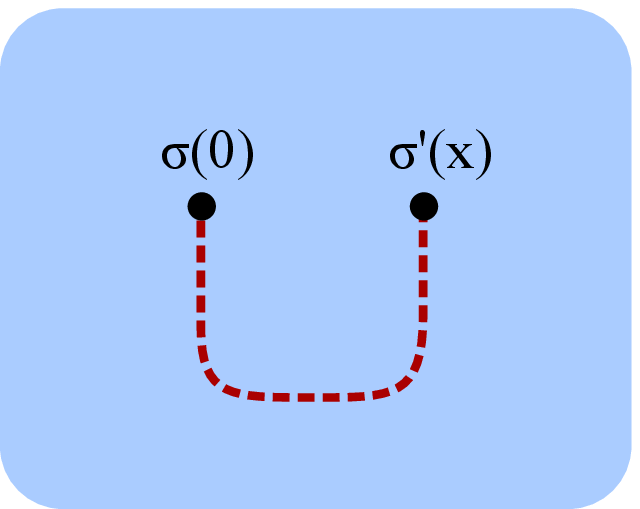}
\end{center}
\caption{Branch cut conventions implicit in the mode exchange relations given in Definition \ref{def:vertex-op}: a) for currents in the presence of insertions $\sigma(x)\phi(0)$, where $\sigma$ is twisted and $\phi$ is untwisted,  and b) for currents in the presence of $\sigma(x)\sigma'(0)$, where both $\sigma$ and $\sigma'$ are twisted.}
\label{fig:branch-cut-3pt}
\end{figure}

To get some practise with this convention, let us recover the mode commutator postulated in cases c) and d) of Definition \ref{def:vertex-op}. For case c) consider the current $a(z)$ in the presence of insertions $\sigma(x)\phi(0)$. We place the branch cut of the twist field $\sigma(x)$ as shown in Figure \ref{fig:branch-cut-3pt}\,a) and take $z^n \sqrt{z-x}$ to have the same choice of branch cut (and to be positive for $z \gg 0$). Then by construction, $z^n \sqrt{z-x} \, a(z) \sigma(x) \phi(0)$ is single valued. Depending on whether $z$ is close to 0 or $\infty$, we have the expansions
\ba
  |z|>|x| &~:~
  z^n \sqrt{z-x}
  = 
  z^n s(z) \big(1-\tfrac{x}{z}\big)^{\frac12}
  =
  \sum_{k=0}^\infty {{1/2} \choose k} (-x)^{k} z^{n-k} s(z) \ ,
  \nonumber \\
  |z|<|x| &~:~ 
  z^n \sqrt{z-x}
  = 
  z^n i \sqrt{x-z} 
  = 
  i z^n \sqrt{x} \big(1-\tfrac{z}{x}\big)^{\frac12}
  =
  i x^{\frac12} \sum_{k=0}^\infty {{1/2} \choose k} (-x)^{-k} z^{n+k}  \ .
\ea
That $+i$ appears instead of $-i$ in the expansion for $|z|<|x|$ follows from our choice of branch cuts. In fact, the whole point of carefully defining the branch cuts is to keep track of these signs. 

Suppose now that $\sigma$ is a ground state. Then by definition, $a_m \sigma = 0$ for $m>0$ and so for all $n \in \Zb_{>0}$
\ba
0~ 
&=~ 
\oint_{|z-x|\ll 1} \hspace{-2em} z^n \sqrt{z-x} \cdot a(z) \sigma(x) \phi(0) \tfrac{dz}{2 \pi i}
~=~ 
\oint_{|z|>|x|} \hspace{-1.5em} (\cdots) \tfrac{dz}{2 \pi i}
-
\oint_{|z|<|x|} \hspace{-1.5em} (\cdots) \tfrac{dz}{2 \pi i}
\nonumber \\
&\overset{(*)}=~
 \sum_{k=0}^\infty {{1/2} \choose k} (-x)^{k} a_{n-k+\frac12} \big( \sigma(x)\phi(0) \big)
\nonumber \\
& \hspace{2em}
 -
 (-1)^{|a||\sigma|}
 i x^{\frac12} \sum_{k=0}^\infty {{1/2} \choose k} (-x)^{-k} \sigma(x) (a_{n+k}\phi)(0) \ .
\ea
The omissions both stand for the original integrand. In (*) the expansions in the relevant domains have been substituted and the definition of the mode integrals was inserted. The resulting identity is case c) of Definition \ref{def:vertex-op}. (The parity sign disappears in Definition \ref{def:vertex-op} because the vertex operator there has not been evaluated on a particular element of the representation $A$.)

For case d) we consider the current $a(z)$ in the presence of insertions $\sigma(x)\sigma'(0)$. The branch cuts for $a(z)$ and for the function $z^{n-1} \sqrt{z} \, \sqrt{z-x}$ are as in Figure \ref{fig:branch-cut-3pt}\,b). The expansions now are
\ba
  |z|>|x| ~:~~
  z^{n-1} \sqrt{z} \, \sqrt{z-x}
  ~&=~ 
  z^n \big(1-\tfrac{x}{z}\big)^{\frac12}
  ~=~
  \sum_{k=0}^\infty {{1/2} \choose k} (-x)^{k} z^{n-k} \ ,
  \nonumber \\
  |z|<|x| ~:~~ 
  z^{n-1} \sqrt{z} \, \sqrt{z-x}
  ~&=~ 
  i z^{n-1} s(z) \sqrt{x} \big(1-\tfrac{z}{x}\big)^{\frac12}
  \nonumber \\
  &=~
  i x^{\frac12} \sum_{k=0}^\infty {{1/2} \choose k} (-x)^{-k} z^{n-1+k} s(z) \ .
\ea
An analogous contour integral calculation to the one before now gives case d) of Definition \ref{def:vertex-op}.

\subsection{Cases 011, 101, 110 from differential equations and translation invariance}

Let $R,S,P\in\Rep\lf(\h)$ and $X,Y\in\svect$ and let 
$f: S \otimes P \to R$,
$g: X \otimes Y \to P$.
We will compute the block
\be 
F(x,y) := \bL{\cbB RSXYP{f;x}{\rule{0pt}{.68em}g;y}} \quad : ~~ S \otimes X \otimes Y \to R \ .
\label{eq:case011-aux1}
\ee
by finding a differential equation for the $y$-dependence and then using the asymptotics to fix the remaining $x$-dependence. Specialising to $x=1$ gives the four-point block in case 011 in Table \ref{tab:4pt-blocks}. By not specialising to $x=1$ from the start we can then use translation invariance to obtain the blocks in the cases 101 and 110.

\medskip

Fix $\phi^*$ in the dual $R^*$ of the space of ground states in $\Ind(R)$, $\phi$ a ground state in $S$, $\tau \in \Ind(X)$ and a ground state $\sigma \in Y$. Assume that $a_m \tau = 0$ for all $m > \frac12$ (so that $\tau$ is built by acting with modes of grade $-\frac12$ from a ground state).

Consider a conformal block of the form $B(z) := \langle \phi^*| a(z) \phi(x) \tau(y) \sigma(0) \rangle$ for some $a \in \h$ and the function $f(z) := z^{\frac12} \, (z-y)^{-\frac12}$. We choose the branch cuts for $B(z)$ and $f(z)$ as in Figure \ref{fig:branch-cut-3pt}\,b). For large $z$ and for $z$ close to $y$ we have the expansions
\be
  f(z) = 1 + O(z^{-1}) ~~ , \quad
  f(z) = y^{\frac12} (z-y)^{-\frac12} + \tfrac12 y^{-\frac12} (z-y)^{\frac12} + O((z-y)^{\frac32}) \ .
\ee
Via contour deformation one obtains
\ba
  &\langle \phi^*| a_0 \, \big( \phi(x) \tau(y) \sigma(0) \big) \rangle
  = \int_{|z|>|x|} \hspace{-1em} f(z) \, B(z) \tfrac{dz}{2 \pi i}
  \nonumber\\ &
  = x^{\frac12} (x-y)^{-\frac12} \langle \phi^*| (a_0 \phi)(x) \tau(y) \sigma(0) \rangle 
  \nonumber\\ & \qquad
  + (-1)^{|a||\phi|} \Big( y^{\frac12} \langle \phi^*| \phi(x) (a_{-\frac12}\tau)(y) \sigma(0) \rangle
  +  \tfrac12 y^{-\frac12} \langle \phi^*| \phi(x) (a_{\frac12}\tau)(y) \sigma(0) \rangle \Big) \ .
  \label{eq:case011-aux2}
\ea
Abbreviate $\xi = b_{-\frac12} a_{-\frac12} \sigma'$ for a ground state $\sigma'$ and where $a$ and $b$ have the same parity. Iterating the above identity twice, starting with $\tau \leadsto a_{-\frac12} \sigma'$ and $a \leadsto b$, gives (the coordinates are not written out)
\ba
  y \cdot \langle \phi^*| \phi \xi \sigma \rangle
  =& 
  -\tfrac14 (b,a) \langle \phi^*| \phi \sigma' \sigma \rangle
  + \langle \phi^*| b_0 a_0 \, ( \phi \sigma' \sigma )\rangle
\nonumber \\ &
  - \big(1-\tfrac yx\big)^{-\frac12} \Big( \langle \phi^*| b_0 \, ( (a_0 \phi) \sigma' \sigma )\rangle
  + (-1)^{|a||b|}  \langle \phi^*| a_0 \, ( (b_0 \phi) \sigma' \sigma )\rangle \Big)
\nonumber \\ &
  + \big(1-\tfrac yx\big)^{-1}  \langle \phi^*|  (b_0 a_0 \phi) \sigma' \sigma \rangle \ .
  \label{eq:case011-aux3}
\ea
To obtain the Knizhnik-Zamolodchikov differential equation \cite{Knizhnik:1984nr} we set $\tau = L_{-1} \sigma'$. By \eqref{eq:Virasoro-modes-twisted} the Virasoro mode $L_{-1}$ acts on a twisted representation as
\be \label{eq:L-1_twisted}
  L_{-1} =
 \tfrac12 \sum_{i \in \Ic} \beta_{-\frac12}^i \alpha_{-\frac12}^i + 
    \sum_{m \in \Zb_{\ge 0} + \frac32} \sum_{i\in\Ic} \beta_{-m}^i \alpha_{-1+m}^i  \ ,
\ee 
so that in particular, $\tau = \frac12 \sum_{i \in \Ic} \beta_{-\frac12}^i \alpha_{-\frac12}^i \sigma'$. Combining this with \eqref{eq:case011-aux3} leads to the following first oder differential equation for $F(x,y)$:
\ba
 y \tfrac{\partial}{\partial y} F(x,y) =& 
 \big( \tfrac12 \Omega^{(11)} - \tfrac{\sdim(\h)}8 \big) \circ F(x,y)
  -  \big(1-\tfrac yx\big)^{-\frac12} \sum_{i\in\Ic} \beta^i \circ F(x,y) \circ (\alpha^i)^{(1)}
\nonumber \\ &
  + \tfrac12 \, \big(1-\tfrac yx\big)^{-1} F(x,y) \circ \Omega^{(11)}
  \label{eq:case011-aux4}
\ea
The logarithmic derivatives are easily integrated:
\be
  \big(1-\tfrac yx\big)^{-\frac12}  = y \tfrac{\partial}{\partial y} \ln\tfrac{1-\sqrt{1-y/x}}{1+\sqrt{1-y/x}} \quad , \qquad
  \big(1-\tfrac yx\big)^{-1} = y \tfrac{\partial}{\partial y} \ln\tfrac{y}{x-y} \ .
\ee
It remains to find the solution with the correct asymptotics for $y \to 0$ to match the product of vertex operators in \eqref{eq:case011-aux1}. From \eqref{eq:V(f)-on-ground-states} we read off
\ba
   F(x,y) 
   \sim f & \circ \exp\!\Big( \big(\ln(x)+\ln(4)\big)\cdot \Omega^{(12)} + \tfrac12\ln(y)\,\big(-\tfrac{\sdim \h}4 + \Omega^{(22)}\big)\Big) 
\nonumber \\ &
   \circ \big\{ \id_S \otimes \big( g \circ (1 \otimes \id_{X \otimes Y})\big)\big\} \ .
  \label{eq:case011-aux5}
\ea
This motivates an ansatz of the form $F(x,y) =  f \circ E(x,y) \circ \big\{ \id_S \otimes \big( g \circ (1 \otimes \id_{X \otimes Y})\big)\big\}$. Substituting the ansatz into the differential equation \eqref{eq:case011-aux4} and using that $f : S \otimes P \to R$ is an $\h$-intertwiner we get the solution
\ba
  E(x,y) = 
  \exp\!\Big( 
   & \ln(y) \cdot \big( \tfrac12 \Omega^{(11)}+ \Omega^{(12)}+ \tfrac12 \Omega^{(22)} - \tfrac{\sdim(\h)}8 \big) 
\nonumber \\ &
  - \ln\tfrac{1-\sqrt{1-y/x}}{1+\sqrt{1-y/x}} \cdot \big(\Omega^{(11)}+ \Omega^{(12)}\big)
  + \tfrac12 \, \ln\tfrac{y}{x-y} \cdot \Omega^{(11)} + u(x) \Big) \ ,
\ea
where $u(x)$ is a function which will be determined now. The asymptotic behaviour of $E(x,y)$ for small $y$ is
\ba
  E(x,y) \sim 
   \exp\!\Big( &
    \ln(y) \cdot \big( \tfrac12 \Omega^{(22)} - \tfrac{\sdim(\h)}8 \big) 
\nonumber \\ &
    +  \ln(x) \cdot \big(\tfrac12 \Omega^{(11)} + \Omega^{(12)} \big)
   + \ln(4) \cdot  \big(\Omega^{(11)}+ \Omega^{(12)}\big)
  + u(x) \Big) \ ,
 \ea
and comparing to \eqref{eq:case011-aux5} we read off $u(x) = - \big( \tfrac12 \ln(x) + \ln(4)\big) \cdot \Omega^{(11)}$.  Altogether, for $x>y>0$,
\ba
  &F(x,y)
  \equiv \Pgs \circ V(f;x) \circ (\id_S \otimes V(g;y) \big)|_{S \otimes X \otimes Y}
\nonumber \\ &
  = y^{- \frac{\sdim(\h)}8 } \cdot 
  f \circ 
  \exp\!\Big( 
    \ln(y) \cdot \big( \tfrac12 \Omega^{(11)}+ \Omega^{(12)}+ \tfrac12 \Omega^{(22)}  \big) 
  - \ln\tfrac{1-\sqrt{1-y/x}}{1+\sqrt{1-y/x}} \cdot \big(\Omega^{(11)}+ \Omega^{(12)}\big)
\nonumber \\ & \hspace{7em}
  + \ln\tfrac{y}{x(x-y)} \cdot \tfrac12 \Omega^{(11)} - \ln(4) \cdot \Omega^{(11)}
    \Big) 
  \circ \big\{ \id_S \otimes \big( g \circ (1 \otimes \id_{X \otimes Y})\big)\big\} \ .
  \label{eq:case-011-with-xy}
\ea
Specialising $(x,y)$ to $(1,x)$ gives the formula in case 011 of Table \ref{tab:4pt-blocks}. 

\medskip

In order to use translation invariance to find the four-point blocks in cases 101 and 110 we need to bring \eqref{eq:case-011-with-xy} into a form where none of the $\Omega^{(ij)}$ act on the ``internal'' representation $P$. This is achieved in the followings slightly more complicated looking reformulation
\ba
  &F(x,y)  =
  \big\{ \widetilde\ev_S \otimes \id_R \big\}
  \circ 
  \exp\!\Big( 
  \tfrac12 \big( - \tfrac{\sdim(\h)}4 + \Omega^{(33)} \big) \ln(y) 
   -\ln \tfrac{1-\sqrt{1-y/x}}{1+\sqrt{1-y/x}} \cdot \Omega^{(13)} 
\nonumber \\
& \hspace{17em}
  + \ln\tfrac{y}{x(x-y)} \cdot \tfrac12 \Omega^{(11)} 
   - \ln(4) \cdot \Omega^{(11)} \Big)
\nonumber \\
&  \qquad
  \circ
  \big\{ \id_{S \otimes S^*} \otimes \big( f \circ (\id_S \otimes g) \big) \big\}
  \circ
  \big\{ \id_S \otimes \widetilde\coev_S \otimes 1 \otimes \id_{X \otimes Y} \big\} \ ,
    \label{eq:case-011-with-xy-external}
\ea
where the exponential is now an endomorphism of $S \otimes S^* \otimes R$ and where we used the evaluation and coevaluation maps $\widetilde\ev_S : S \otimes S^* \to \Cb$ and
$\widetilde\coev_S : \Cb \to S^* \otimes S$ of $\svect$.

As in Section \ref{app:010-100-transinv} we write $\nu(x,y,z) = F(x-z,y-z)$. Then the functional form in cases 101 and 110 in Table \ref{tab:4pt-blocks} is given by $\nu(x,1,0)$ and $\nu(0,1,x)$, respectively.

\subsubsection*{Case 101}

The specialised four-point block reads
\ba
\nu(x,1,0)
  =
  \big\{ \widetilde\ev_S \otimes \id_R \big\}
  &\circ 
  \exp\!\Big(\! 
   -\ln \tfrac{ i \nu \, \sqrt{1/x-1}+ 1}{ i \nu \, \sqrt{1/x-1} - 1} \cdot \Omega^{(13)} 
  - \ln(x(1{-}x)) \cdot \tfrac12 \Omega^{(11)} 
   - \ln(4) \, \Omega^{(11)} \Big)
\nonumber \\
& 
  \circ
  \big\{ \id_{S \otimes S^*} \otimes \big( \text{morph.} \big) \big\}
  \circ
  \big\{ \id_S \otimes \widetilde\coev_S \otimes \id_{X \otimes Y} \big\} \ ,
    \label{eq:case-101-2ndorder}
\ea
The ambiguity in choosing the square root branch is captured by $\nu \in \{ \pm 1\}$ while the choice of branch for the logarithm has been absorbed into the morphism. The expansion of the exponential is
\be
  \exp\!\big(\cdots\big)
  = 
  \exp\!\big(- \big(\tfrac12 \ln(x)+\ln(4)\big) \cdot  \Omega^{(11)} \big) \circ
  \big( \id_{S \otimes S^* \otimes R} + 2 i \nu x^{\frac12} \cdot \Omega^{(13)} + O(x) \big) \ .
    \label{eq:case-101-aux1}
\ee

Pick a morphism $f : X \ast Z \equiv U(\h) \otimes X \otimes Z \to R$ and $g : S \otimes Y \to Z$. The expansion of the product of vertex operators is given by
\ba
&\bL{\cbB RXSYZ{f;1}{\rule{0pt}{.68em}g;x}} 
= \Pgs \circ V(f;1) \circ (\id_X \otimes V(g;x)) = B_0 + B_1 + \dots
\nonumber\\
&\overset{(*)}=
 (\id_R \otimes \ev_S) \circ
  \exp\!\big(- ( \tfrac12\ln(x)+\ln(4) ) \cdot \Omega^{(33)}\big) \circ \big( \id_{R \otimes S^* \otimes S} - 2 i x^{\frac12} \cdot \Omega^{(13)}   + O(x) \big)
\nonumber\\
&\qquad \circ \big\{  (f \circ (1 \otimes \id_X \otimes g))  \otimes \id_{S^* \otimes S} \big\} 
\nonumber\\
&\qquad 
\circ 
\big\{ \id_X \otimes \big[ (\tau_{Y,S} \otimes \id_{S^* \otimes S}) \circ (\id_Y \otimes \coev_S \otimes \id_S)
\circ \tau_{S,Y} \big] \big\} \ .
    \label{eq:case-101-aux2}
\ea
Relative to \eqref{eq:case-101-aux1} the ``$S$-sling'' has been moved from the left to the right. In any case, we see that in \eqref{eq:case-101-2ndorder} we must take $\nu=-1$ and fix the morphism part to be $(f \circ (1 \otimes \id_X \otimes g))$. This recovers case 101 in Table \ref{tab:4pt-blocks}.

To verify (*) in \eqref{eq:case-101-aux2} compute
\ba
B_0 &\overset{\eqref{eq:V(f)-on-ground-states}}= 
f \circ (1 \otimes \id_{X} \otimes g) \circ \exp\!\big(- ( \tfrac12\ln(x)+\ln(4) ) \cdot \Omega^{(22)}\big) \ ,
\nonumber \\
B_1 &\overset{\eqref{eq:id-expansion-twist}}=  2 \sum_{j\in\Ic}  \Pgs \circ V(f;1) \circ (\beta^j_{-\frac12})^{(2)} \circ (\id_R \otimes  \Pgs ) \circ (\alpha^j_{\frac12})^{(2)} \circ \big(\id_R \otimes V(g;x)\big)
\nonumber \\
& \overset{(**)}= - 2 i x^{\frac12}  \sum_{j\in\Ic}  \beta^j  \circ B_0 \circ (\alpha^j)^{(2)} \ .
\ea
For (**) first recall from Definition \ref{def:vertex-op} that $T(a)_{-\frac12}^{x,+} = a_{-\frac12} + (\text{pos.\ modes})$ and $T(a)_{0}^{x,-} = a_0 + (\text{neg.\ modes})$. Case d) of that definition then gives
$\Pgs \circ V(f;1) \circ (\beta^j_{-\frac12})^{(2)} = - i  \beta^j \circ \Pgs \circ V(f;1)$,
while case b) gives $\alpha^j_{\frac12} \circ V(g;x)|_{S \otimes Y} = x^{\frac12} \, V(g;x) \circ (\alpha^j)^{(1)}|_{S \otimes Y}$. Finally, to obtain (*), insert evaluation and coevaluation maps for $S$ to produce an ``$S$-sling'' which allows one to move the modes acting on $S$ to the left.

\subsubsection*{Case 110}

Here one finds
\ba
\nu(0,1,x) 
  &=
  \big\{ \widetilde\ev_S \otimes \id_R \big\}
  \circ 
  \exp\!\Big( 
  \tfrac12 \big( - \tfrac{\sdim(\h)}4 + \Omega^{(33)} \big) \ln(1{-}x) 
   -\ln \tfrac{1-\sqrt{x}}{1+\sqrt{x}} \cdot \Omega^{(13)} 
\nonumber \\
& \hspace{17em}
  + \ln\tfrac{1{-}x}{x} \cdot \tfrac12 \Omega^{(11)} 
   - \ln(4) \cdot \Omega^{(11)} \Big)
\nonumber \\
&  \qquad
  \circ
  \big\{ \id_{S \otimes S^*} \otimes \big( \text{morph.} \big) \big\}
  \circ
  \big\{ \id_S \otimes \widetilde\coev_S \otimes \id_{X \otimes Y} \big\} \ ,
    \label{eq:case-110-2ndorder}
\ea
There is no square root ambiguity to determine, but let us compare the first order expansions anyway as a consistency check. We have
\be
  \exp\!\big(\cdots\big)
  = 
  \exp\!\big(- \big(\tfrac12 \ln(x)+\ln(4)\big) \cdot  \Omega^{(11)} \big) \circ
  \big( \id_{S \otimes S^* \otimes R} + 2  x^{\frac12} \cdot \Omega^{(13)} + O(x) \big) 
\ .
\ee
Fix $f : X \ast Z \equiv U(\h) \otimes X \otimes Y \to R$ and $g : Y \otimes S \to Z$. Then
\ba
&\bL{\cbB RXYSZ{f;1}{\rule{0pt}{.68em}g;x}}
= \Pgs \circ V(f;1) \circ (\id_X \otimes V(g;x)) = B_0 + B_1 + \dots
\nonumber \\
&\overset{(*)}=
 (\id_R \otimes \ev_S) \circ
  \exp\!\big(- ( \tfrac12\ln(x)+\ln(4) ) \cdot \Omega^{(33)}\big) \circ \big( \id_{R \otimes S^* \otimes S} + 2 x^{\frac12} \cdot \Omega^{(13)}   + O(x) \big)
\nonumber\\
&\qquad \circ \big\{  (f \circ (1 \otimes \id_X \otimes g))  \otimes \id_{S^* \otimes S} \big\} 
\circ 
\big\{ \id_{X \otimes Y} \otimes \coev_S \otimes \id_S \big\} \ .
\ea
Thus in \eqref{eq:case-101-2ndorder} we must again take the morphism part to be $(f \circ (1 \otimes \id_X \otimes g))$. This recovers case 110 in Table \ref{tab:4pt-blocks}. For (*) note
\ba
B_0 &\overset{\eqref{eq:V(f)-on-ground-states}}= 
f \circ (1 \otimes \id_{X} \otimes g) \circ \exp\!\big(- ( \tfrac12\ln(x)+\ln(4) ) \cdot \Omega^{(33)}\big) \ ,
\nonumber \\
B_1 &\overset{\eqref{eq:id-expansion-twist}}=  2 \sum_{j\in\Ic}  \Pgs \circ V(f;1) \circ (\beta^j_{-\frac12})^{(2)} \circ (\id_R \otimes  \Pgs ) \circ (\alpha^j_{\frac12})^{(2)} \circ \big(\id_R \otimes V(g;x)\big)
\nonumber \\
& \overset{(**)}= 2 x^{\frac12}  \sum_{j\in\Ic}  \beta^j  \circ B_0 \circ (\alpha^j)^{(2)} \ .
\ea
For (**) we have $\Pgs \circ V(f;1) \circ (\beta^j_{-\frac12})^{(2)} = - i  \beta^j \circ \Pgs \circ V(f;1)$ as in case 101, as well as $\alpha^j_{\frac12} \circ V(g;x) = i x^{\frac12} \, V(g;x) \circ (\alpha^j)^{(2)}$ as in case 010.

\subsection{Differential equation for case 111}\label{app:4ptblock-111}

\begin{figure}[bt]
\begin{center}
\ipic{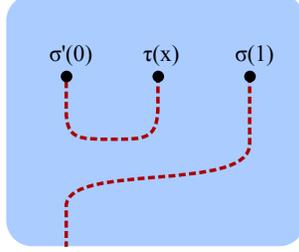}
\end{center}
\caption{Branch cut convention for currents in the presence of three twisted insertions $\sigma(1), \tau(x), \sigma'(0)$.}
\label{fig:branch-cut-4twisted}
\end{figure}

Pick super-vector spaces $W,X,Y,Z$ and a locally finite $\h$-module $P$, as well as maps $f : X \otimes P \to W$ and $g: Y \ast Z \equiv U(\h) \otimes Y \otimes Z \to W$. We need to compute the four-point block
\be 
F(x) := \bL{\cbB WXZYP{f;1}{\rule{0pt}{.68em}g;x}} \quad : ~~ X \otimes Y \otimes Z \to W \ .
\label{eq:case111-aux1}
\ee
To obtain the Knizhnik-Zamolodchikov equation in this case we need a second trick, due to \cite{Dixon:1986qv,Zamolodchikov:1987ae}. Namely, we will compute the four-point block with an additional insertion of $a(z)$ in terms of blocks without current insertions. Let
\be
  c(z) := \langle \sigma^*| \sigma(1) \, Q \, \big( a(z) \tau(x) \sigma'(0) \big) \rangle \ ,
\ee
where $Q$ is a polynomial in the zero modes, $\sigma^*,\sigma,\sigma'$ are ground states and $\tau \in Y$ has the property that $a_m \tau = 0$ for all $m > \frac12$. The choice of branch cuts for $c(z)$ is shown in Figure \ref{fig:branch-cut-4twisted}. Let 
\be \label{eq:u-3x-sqrt-def}
  u(z) = \sqrt{z(1-z)(z-x)}
\ee
have the same choice of branch cuts as $c(z)$ and be positive for $z \in (x,1)$. Then the function $\varphi(z) = u(z)c(z) / (z-x)$ is a holomorphic function on $\Cb \setminus \{x\}$. There is no singularity at $0$ since $\sigma'$ is a ground state and $u(z) \sim \text{(const)}\cdot z^{\frac12}$. For the same reason, there is no singularity at $1$. At $x$ the asymptotic behaviour is
\be
  \frac{u(z)}{z-x} = 
  \big( x(1-x) \big)^{\frac12} \cdot \frac{s(z-x)}{z-x} 
  + 
  \tfrac12 \, (1 - 2x) \, \big( x(1-x) \big)^{-\frac12} \cdot s(z-x)  + O\big((z-x)^{\frac32} \big)  \ .
\ee
On the other hand, since 
\be
  a(z)\tau(x) = \frac{s(z-x)}{(z-x)^2} \cdot a_{\frac12} \tau + \frac{s(z-x)}{z-x} \cdot a_{-\frac12} \tau + O((z-x)^{\frac12}) 
\ee   
we see that the asymptotic behaviour of $\varphi(z)$ close to $x$ is (the insertion points are omitted for brevity)
\ba
  \varphi(z) ~=~ &\Bigg( \frac{\big( x(1-x) \big)^{\frac12} }{ (z-x)^{2} } + \frac{\tfrac12 \, (1-2x) \, \big( x(1-x) \big)^{-\frac12}}{z-x} \Bigg)
  \langle \sigma^*| \sigma \, Q \, (a_{\frac12} \tau) \sigma'  \rangle
\nonumber \\
  &+~
  \frac{\big( x(1-x) \big)^{\frac12} }{ z-x } \cdot
  \langle \sigma^*| \sigma\,  Q \, (a_{-\frac12} \tau) \sigma' \rangle
  ~+~ \eta(z) \ ,
\ea
where $\eta(z)$ is now regular at $x$, i.e.\ it is a holomorphic function on $\Cb$. For $z\to \infty$, the leading behaviour of $u(z)/(z-x)$ is $z^{\frac12}$ and that of $c(z)$ is $z^{-\frac32}$ (since $\langle \sigma^*| a_{\frac12} \sigma Q \tau \sigma' \rangle$ is the most positive mode which contributes). Thus $\varphi(z)$ goes to zero as $z \to\infty$, and therefore also $\eta(z)$. It follows that $\eta(z)=0$. Altogether,
\ba
  c(z) ~=~ &\Bigg( \frac{\big( x(1-x) \big)^{\frac12} }{ u(z) (z-x) } + \frac{ \tfrac12\,(1-2 x) \, \big( x(1-x) \big)^{-\frac12}}{u(z)} \Bigg)
  \langle \sigma^*| \sigma \, Q \, (a_{\frac12} \tau) \sigma'  \rangle
\nonumber \\
  &+~
  \frac{\big( x(1-x) \big)^{\frac12} }{ u(z) } \cdot
  \langle \sigma^*| \sigma\,  Q \, (a_{-\frac12} \tau) \sigma' \rangle
  \ .
\label{eq:case111-aux2}
\ea
On the other hand,
\be
  \oint_{x<|z|<1} \hspace{-2em} c(z) \,\, \tfrac{dz}{2 \pi i} ~=~ \langle \sigma^*| \sigma(1) \, (Q a_0) \, \big( \tau(x) \sigma'(0) \big) \rangle \ .
\label{eq:case111-aux3}
\ee
To compute the same contour integral in \eqref{eq:case111-aux2} we will need
\ba
\oint_{x<|z|<1} \hspace{-1.5em} u(z)^{-1}  \,\, \tfrac{dz}{2 \pi i} 
&\overset{\text{(1)}}=  \tfrac{(- 2)(- i)}{2 \pi i}  \int_0^x\big( z (1-z) (x-z) \big)^{-\frac12} \, dz
\nonumber \\
&\overset{\text{(2)}}=
 \tfrac{2}{ \pi }  \int_0^1 \frac{dt}{\sqrt{ (1 -  t^2) (1-x t^2) }}  ~=~ \tfrac{2}{ \pi } K(x) \ .
\label{eq:case111-aux4}
\ea
In step 1, the factor of $-2$ arises from collapsing the contour to the interval $(0,1)$: the upper contribution gives a minus sign from the change of integration direction and the lower contribution has to pass through the branch cut, giving a relative minus sign which compensates the sign from the change in integration direction. The factor of $-i$ arises since by the choice of branch cuts for $u(z)$ in Figure \ref{fig:branch-cut-4twisted}, for $z \in (0,x)$ we have $u(z) = i \sqrt{(z (1-z) (x-z)}$, where the square root on the right hand side is real and positive. In step 2 the substitution $z = x \, t^2$ was made. $K(x)$ is the complete elliptic integral of the first kind. 
Since $\frac{\partial}{\partial x} u(z)^{-1} = \frac12 \, (z-x)^{-1} u(z)^{-1}$, we can also conclude that
\be
\oint_{x<|z|<1} \hspace{-1.5em} (z-x)^{-1} \, u(z)^{-1}  \,\, \tfrac{dz}{2 \pi i} = \tfrac{4}{ \pi } K'(x)  \ .
\label{eq:case111-aux5}
\ee
Putting these ingredients together provides us with the relation
\ba
\tfrac{ \pi }{2} \, \big( x(1-x) \big)^{-\frac14}
\langle \sigma^*| \sigma \, (Q a_0) \, \big( \tau \sigma' \big) \rangle 
 ~=~  & 
 \langle \sigma^*| \sigma \, Q \, (a_{\frac12} \tau) \sigma'  \rangle
 \cdot 2\tfrac{\partial}{\partial x}  \big( x^{\frac14} (1-x)^{\frac14} K(x)  \big) 
\nonumber\\
  & +~
  \langle \sigma^*| \sigma\,  Q \, (a_{-\frac12} \tau) \sigma' \rangle 
  \cdot
  x^{\frac14} (1-x)^{\frac14}  K(x) 
  \ .
\ea
Iterating this relation gives, for $\sigma''$ also a ground state,
\ba
\tfrac{ 1 }{2} \, 
\langle \sigma^*| \sigma \, (b_{-\frac12} a_{-\frac12} \sigma'') \sigma' \big) \rangle 
 ~=~  & 
 \frac{\pi^2/8}{x(1-x) K(x)^2} 
 \langle \sigma^*| \sigma \, (b_0 a_0) \, \sigma'' \sigma' \big) \rangle 
\nonumber\\
  & -~
 \frac{\frac12 (b,a) \cdot \tfrac{\partial}{\partial x}  \big( x^{\frac14} (1-x)^{\frac14} K(x)  \big) }{x^{\frac14} (1-x)^{\frac14} K(x)} 
 \langle \sigma^*| \sigma \sigma'' \sigma' \big) \rangle 
  \ .
\label{eq:case111-aux6}
\ea
From the above equation for conformal blocks we can read off the differential equation for \eqref{eq:case111-aux1}. Namely, for $E$ an endomorphism of $\Ind(P)$ abbreviate
\be
  H(E;x) := \Pgs \circ V(f;1) \circ (\id_X \otimes E)\circ (\id_X \otimes V(g;x)) \ .
\label{eq:case111-aux6.5}
\ee
and $G(x) = x^{\frac14} (1-x)^{\frac14} K(x)$, and recall the expansion \eqref{eq:L-1_twisted} for $L_{-1}$. Then \eqref{eq:case111-aux6} becomes
\be
  \tfrac{\partial}{\partial x} H(\id;x) =  \frac{\pi^2/8}{x(1-x) K(x)^2} \, H(\Omega^{(11)},x)
 -
 \frac{\frac12 \sdim(\h) \, G'(x)}{G(x)} \, H(\id;x)
  \ .
\ee
A solution to this first order differential equation is
\ba
  F(x) \equiv  H(\id;x) 
  = 
  \big(\tfrac2\pi G(x)\big)^{- \frac{\sdim(\h)}2} 
\cdot  f
&\circ 
\exp\! \big(\! - \tfrac\pi2 \, \tfrac{K(1{-}x)}{K(x)} \cdot  \Omega^{(22)}  \big)  
\nonumber \\
&\circ \big\{ \id_X \otimes \big( g \circ (1 \otimes \id_{Y \otimes Z})\big)\big\} \ .
\label{eq:case111-aux7}
\ea
The only non-trivial point in verifying that the above function is a solution is the identity
\be
  \frac{\partial}{\partial x} \frac{K(1-x)}{K(x)} = - \frac{\pi/4}{x(1-x)\,K(x)^2} \ ,
\ee
which in turn follows from Legendre's identity for the complete elliptic integrals of the first and second kind \cite[17.3.13]{Abra-Steg-book}. To establish that \eqref{eq:case111-aux7} is indeed the solution we are looking for, we need to verify the asymptotic behaviour. From \eqref{eq:V(f)-on-ground-states} we know
\be
  F(x) 
   \sim 
   x^{-\frac{\sdim(\h)}8} \cdot
   f \circ \exp\!\big(  (\tfrac12\ln(x) - \ln(4))\cdot \Omega^{(22)}\big) 
   \circ \big\{ \id_X \otimes \big( g \circ (1 \otimes \id_{Y \otimes Z})\big)\big\} \ .
  \label{eq:case111-aux8}
\ee
To find the asymptotic behaviour of \eqref{eq:case111-aux7} use \cite[17.3.9\,\&\,17.3.21]{Abra-Steg-book}
\ba
  K(x) ~&=~ \tfrac{\pi}2 \, {}_2F_1\big(\tfrac12,\tfrac12;1;x\big) ~=~ \tfrac{\pi}2 \big(1 + \tfrac14  x + O(x^2) \big) \ ,
  \nonumber \\
  q(x) ~&=~ \exp(- \pi K(1-x)/K(x)) ~=~ \tfrac{x}{16} + 8 \big( \tfrac{x}{16} )^2 + O(x^3) \ .
\ea
The expansion for $q(x)$ implies $-\tfrac{\pi}2  K(1-x)/K(x) = \frac12 \ln(x/16) + O(x)$, so that the asymptotic behaviour of \eqref{eq:case111-aux7} for $x \to 0$ reproduces \eqref{eq:case111-aux8}. This concludes the description how case 111 in Table \ref{tab:4pt-blocks} was found.

\begin{example} In the two examples treated in Section \ref{sec:examples-VO}, the four-point block $F(x)$ in \eqref{eq:case111-aux7} specialises as follows:

\smallskip\noindent
(i) Single free boson, $\h = \Cb^{1|0}$: Take $P$ to be the one-dimensional $\h$-representations $P = \Cb_{p}$ for some $p \in \Cb$ and take $W = X = Y = Z = \Cb^{1|0}$. Then $\Omega^{(11)}$ acts on $P$ as $p^2$ and we find
\be
  F(x) = \text{(const)} \cdot \big( x(1-x) \big)^{-\frac18} \, K(x)^{-\frac12} \, q(x)^{\frac12 p^2} \ .
\ee
This matches \cite[Eqn.\,(1.1)]{Zamolodchikov:1987ae}.

\smallskip\noindent
(ii) Single pair of symplectic fermions, $\h = \Cb^{0|2}$: Take $P = U(\h)$ and once more $W = X = Y = Z = \Cb^{1|0}$. Then $\Omega^{(11)}$ squares to zero on $P$ and the exponential sum terminates after two terms:
\be
  F(x) = 
  \text{(const${}_1$)} \cdot \big( x(1-x) \big)^{\frac14} K(x) 
  +
  \text{(const${}_2$)} \cdot \big( x(1-x) \big)^{\frac14} K(1-x) \ ,
\ee
in agreement with the two linearly independent solutions to the level-two null vector condition, see \cite{Kausch:1995py} and \cite[Eqn.\,(E.9)]{Gaberdiel:2006pp}.
\end{example}

\subsubsection*{Evaluating case 111 for descendent states}

Take $W,X,Y,Z$ and $P$ be as in the previous section, let $\sigma, \sigma', \sigma''$ be ground states and let $\tau^* \in \Indtw(W)'$ have the property that $\tau^* a_m = 0$ for all $m < -\frac12$. 
For $Q$ a polynomial in zero modes set
\be
  d(z) := \langle \tau^*| \sigma(1) a(z) \, Q \, \big( \sigma''(x) \sigma'(0) \big) \rangle \ ,
\ee
with choice of branch cuts as in Figure \ref{fig:branch-cut-4twisted}. Recall the function $u(z)$ from \eqref{eq:u-3x-sqrt-def}. Since $\sigma, \sigma', \sigma''$ are ground states, the function $\psi(z) := u(z) d(z)$ extends to a holomorphic function on all of $\Cb$. From the choice of branch cuts we see that for $z>1$ and real, $u(z) = - i \sqrt{z(z-1)(z-x)}$, where the square root is real and positive. For large $z$ we have the expansions
\ba
  u(z) &= - i \cdot z^{\frac32} + \tfrac{i}2 (x+1) \cdot z^{\frac12} + O(z^{-\frac12}) \ ,
\nonumber \\
  d(z) &= 
  z^{-\frac12} \cdot A_-  +  z^{-\frac32} \cdot A_+ + O(z^{-\frac52})  \ ,
\nonumber \\
  \psi(z) &=  - i z \cdot A_- +  \tfrac{i}2 (x+1) A_- - i A_+ + O(z^{-1}) \ ,
\label{eq:case111desc-aux2}
\ea
where $A_\pm = (-1)^{|\sigma||a|} \langle \tau^*| a_{\pm \frac12} \sigma(1) \, Q \, \big( \sigma''(x) \sigma'(0) \big) \rangle$. Since $\psi(z)$ is holomorphic on $\Cb$, the $O(z^{-1})$ contribution is actually zero, so that we have expressed $\psi$ in terms of $A_\pm$. 

Next we want to remove the $z$-dependence via integration along a contour enclosing $(0,x)$ as in \eqref{eq:case111-aux3}. For this we need the integral
\be
\oint_{x<|z|<1} \hspace{-1.5em} z \, u(z)^{-1}  \,\, \tfrac{dz}{2 \pi i} 
= \tfrac{2}{\pi} \big( K(x) - E(x) \big) \ ,
\label{eq:case111desc-aux1}
\ee
which follows by writing $z = 1 - (1-z)$. Up to the factor of pi over two, the contour integral of $(1-z) / u(z)$ produces the complete elliptic integral of the second kind, $E(x) = \int_0^1 (1-x t^2)^{\frac12} / (1-t^2)^{\frac12} dt$.

Dividing the identity we found for $\psi(z)$ in \eqref{eq:case111desc-aux2} by $u(z)$ and integrating as in \eqref{eq:case111-aux3} gives
\be
 \tfrac{i \pi}{2} \langle \tau^*| \sigma \, (a_0 Q) \, \big( \sigma'' \sigma') \big) \rangle
  = K(x) \cdot A_+ + \big(\tfrac{1-x}2  K(x) -  E(x) \big) \cdot A_-      \ .
\ee
From this we can formulate an identity satisfied by the composition of twisted vertex operators. To this end define a version of $H(E,x)$ as in \eqref{eq:case111-aux6.5} but without the projection to ground states,
\be
  J(Q;x) := V(f;1) \circ Q^{(2)} \circ (\id_X \otimes V(g;x)) ~:~ X \otimes Y \otimes Z \to \overline\Indtw(W) \ .
\label{eq:case111desc-J-def}
\ee
Then, if $\tau^* \in \Indtw(W)'$ satisfies $\tau^* a_m = 0$ for all $m < -\frac12$ as before, we obtain
\be
  \tfrac{i \pi}{2} \, \tau^* J(a_0 Q;x)
  = 
  K(x) \cdot \tau^* a_{\frac12}^{(1)} J(Q;x)
  + 
  \big(\tfrac{1-x}2  K(x) -  E(x) \big) \cdot \tau^* a_{-\frac12}^{(1)} J(Q;x)      \ .
\label{eq:case111desc-J-rec}
\ee

\newcommand\arxiv[2]      {\href{http://arXiv.org/abs/#1}{#2}}
\newcommand\doi[2]        {\href{http://dx.doi.org/#1}{#2}}
\newcommand\httpurl[2]    {\href{http://#1}{#2}}

\end{document}